\documentclass[reqno]{amsart}
\usepackage{amsthm}
\usepackage{hyperref}
\usepackage{xcolor}
\usepackage{amsfonts}
\usepackage[T1]{fontenc}
\usepackage{enumitem}
\usepackage{amsmath,amssymb,amsfonts,amsthm}
\usepackage{bm}
\usepackage{latexsym}
\usepackage{amscd}
\usepackage[all,cmtip]{xy}
\usepackage{xfrac}
\usepackage{mathrsfs}
\usepackage{upgreek}
\usepackage{graphicx}
\usepackage{enumitem} 
\usepackage{braket}
\usepackage{verbatim}
\usepackage{stmaryrd}
\usepackage{commath}
\usepackage{cancel}
\usepackage{esint}
\usepackage[mathscr]{eucal}
\usepackage{eqnarray}

\hypersetup{
    colorlinks = true,
    citecolor = blue,
    allcolors = blue
}

\calclayout

\allowdisplaybreaks

\numberwithin{equation}{section}

\theoremstyle{definition}
\newtheorem{definition}{Definition}[section]
\newtheorem{remark}{Remark}[section]

\newtheorem*{claim}{Claim}

\theoremstyle{remark}
\newtheorem*{claimproof}{Proof of claim}

\theoremstyle{plain}
\newtheorem{proposition}{Proposition}[section]
\newtheorem{theorem}[proposition]{Theorem}
\newtheorem{lemma}[proposition]{Lemma}
\newtheorem{corollary}[proposition]{Corollary}

\makeatletter

\def\chaptermark#1{}

\def\chapter{%
  \if@openright\cleardoublepage\else\clearpage\fi
  \thispagestyle{plain}\global\@topnum\z@
  \@afterindenttrue \secdef\@chapter\@schapter}

\def\@chapter[#1]#2{\refstepcounter{chapter}%
  \ifnum\c@secnumdepth<\z@ \let\@secnumber\@empty
  \else \let\@secnumber\thechapter \fi
  \typeout{\chaptername\space\@secnumber}%
  \def\@toclevel{0}%
  \ifx\chaptername\appendixname \@tocwriteb\tocappendix{chapter}{#2}%
  \else \@tocwriteb\tocchapter{chapter}{#2}\fi
  \chaptermark{#1}%
  \addtocontents{lof}{\protect\addvspace{10\p@}}%
  \addtocontents{lot}{\protect\addvspace{10\p@}}%
  \@makechapterhead{#2}\@afterheading}
\def\@schapter#1{\typeout{#1}%
  \let\@secnumber\@empty
  \def\@toclevel{0}%
  \ifx\chaptername\appendixname \@tocwriteb\tocappendix{chapter}{#1}%
  \else \@tocwriteb\tocchapter{chapter}{#1}\fi
  \chaptermark{#1}%
  \addtocontents{lof}{\protect\addvspace{10\p@}}%
  \addtocontents{lot}{\protect\addvspace{10\p@}}%
  \@makeschapterhead{#1}\@afterheading}
\newcommand\chaptername{Chapter}

\def\@makechapterhead#1{\global\topskip 7.5pc\relax
  \begingroup
  \fontsize{\@xivpt}{18}\bfseries\centering
    \ifnum\c@secnumdepth>\m@ne
      \leavevmode \hskip-\leftskip
      \rlap{\vbox to\z@{\vss
          \centerline{\normalsize\mdseries
              \uppercase\@xp{\chaptername}\enspace\thechapter}
          \vskip 3pc}}\hskip\leftskip\fi
     #1\par \endgroup
  \skip@34\p@ \advance\skip@-\normalbaselineskip
  \vskip\skip@
 }

\def\@makeschapterhead#1{\global\topskip 7.5pc\relax
  \begingroup
  \fontsize{\@xivpt}{18}\bfseries\centering
  #1\par \endgroup
  \skip@34\p@ \advance\skip@-\normalbaselineskip
  \vskip\skip@ }
\def\appendix{\par
  \c@chapter\z@ \c@section\z@
  \let\chaptername\appendixname
  \def\theHchapter{\@Alph\c@chapter}
  \def\theHsection{\@Alph\c@section}
 }
 \newcounter{chapter}
 
\addtocontents{toc}{\setcounter{tocdepth}{1}}

\newif\if@openright

\makeatother

\newcommand{\NN}{\mathbb{N}}
\newcommand{\N}{\mathbb{N}}

\newcommand{\R}{\mathbb{R}}
\renewcommand{\SS}{\mathbb{S}}

\newcommand{\cF}{\mathcal F}

\newcommand{\sB}{\mathscr{B}}
\newcommand{\sA}{\mathscr{A}}
\newcommand{\sF}{\mathscr{F}}
\newcommand{\sH}{\mathcal{H}}
\newcommand{\sL}{\mathscr{L}}

\newcommand{\sM}{\mathscr{M}}

\newcommand{\dist}{\mathrm{dist}}
\newcommand{\eps}{\epsilon}

\DeclareMathOperator{\divr}{div}
\DeclareMathOperator{\diam}{diam}

\DeclareMathOperator{\spt}{spt}

\DeclareMathOperator{\Lip}{Lip}

\DeclareMathOperator{\dH}{\, d\sH^{n-1}}
\DeclareMathOperator{\Id}{Id}
\DeclareMathOperator{\Hm}{\sH^{n-1}}
\DeclareMathOperator{\ex}{\textnormal{\textbf{e}}}
\DeclareMathOperator{\fl}{\textnormal{\textbf{f}}}
\DeclareMathOperator{\qq}{\textnormal{\textbf{q}}}
\DeclareMathOperator{\pp}{\textnormal{\textbf{p}}}
\newcommand{\B}{\textbf{\textit{B}}}
\newcommand{\C}{\textbf{\textit{C}}}
\newcommand{\D}{\textbf{\textit{D}}}
\newcommand{\K}{\textbf{\textit{K}}}
\newcommand{\W}{\textbf{\textit{W}}}

\renewcommand{\phi}{\varphi}
\renewcommand{\epsilon}{\varepsilon}

\newcommand{\la}{\langle}
\newcommand{\ra}{\rangle}
\newcommand{\wkly}{\overset{\ast}{\rightharpoonup}}
\newcommand{\loc}{\overset{\text{loc}}{\rightarrow}}
\renewcommand{\bar}{\overline}
\renewcommand{\tilde}{\widetilde}

\newcommand{\rstr}{\mathbin{\vrule height 1.ex depth 0pt width
0.1ex\vrule height 0.1ex depth 0pt width 1.ex}\:}

\newcommand{\rb}{\partial^*\!}
\newcommand{\p}{\partial}
\newcommand{\subsetcc}{\subset\subset}
\newcommand{\cl}{\overline}

\author[Simmons]{David A. Simmons}
\thanks{The author was partially supported by NSF FRG 1853993.}
\date{\today}

\title[Regularity for H\"older-Coefficient Surface Energies]{Regularity of Almost-Minimizers of H\"older-Coefficient\\ Surface Energies}

\address{Department of Mathematics, University of Washington, Seattle, WA}
\email{simmons3@uw.edu}

\begin{document}
\begin{abstract}
	We study almost-minimizers of anisotropic surface energies defined by a H\"older continuous matrix of coefficients acting on the unit normal direction to the surface. In this generalization of the Plateau problem, we prove almost-minimizers are locally H\"older continuously differentiable at regular points and give dimension estimates for the size of the singular set. We work in the framework of sets of locally finite perimeter and our proof follows an excess-decay type argument.
\end{abstract}

\maketitle

\tableofcontents

\section{Introduction}

	The Plateau problem is a classical geometric variational problem. It consists in minimizing surface area among all surfaces with a certain prescribed boundary. The analogous physical phenomenon occurs in soap films as they seek to minimize surface tension, an equivalent to minimizing surface area. The existence and regularity of solutions to the Plateau problem has been the subject of study in a variety of settings and continues to be a centerpiece of much mathematical research (to name a few, see \cite{douglas, rado, degiorgifront, reifenberg, allard, taylor, harrisonpugh2016, delellisghiraldinmaggidirect, KinMagStu19}). A natural generalization of the Plateau problem is to study minimizers of surface energies other than surface area. Anisotropic surface energies are those which depend on the normal direction to the surface and possibly the spatial location of the surface as well. This means that the energy assigned to a surface depends not only on its geometry but also on how and where the surface sits in space. Such anisotropic energies arise in physical phenomena such as the formation of crystals and in crystalline materials.
	
	Almgren was the first to study regularity of minimizers to anisotropic variational problems in his paper \cite{almgrenexistence}. This initial work as well as much of the subsequent work in the area was done in the setting of varifolds and currents with many of the results applying to surfaces of arbitrary codimension but with rather strong regularity assumptions on the integrands of the anisotropic energies. 
	
    In this paper we work in the setting of sets of locally finite perimeter and study the existence and regularity of minimizers of anisotropic surface energies of the form
	\begin{align}\label{anisotropic surface energy}
		\sF_A(E; U)=\int_{U\cap \rb E} \la A(x)\nu_E(x),\nu_E(x)\ra^{1/2}\dH(x)
	\end{align}
	where $A=(a_{ij}(x))_{i,j=1}^n$ is a uniformly elliptic, H\"older continuous matrix-valued function, $E$ is a set of locally finite perimeter in $\R^n$, and $U$ is an open set. Here $\rb E$ denotes the $(n-1)$-dimensional reduced boundary of $E$ and $\nu_E$ denotes its outward unit normal vector. We note that H\"older continuity is a rather weak regularity assumption and the previously known regularity results for general integrands do not apply (see the discussion below). Our main regularity result applies to almost-minimizers which are sets of locally finite perimeter in $\R^n$ satisfying the minimality condition
	\begin{align}\label{initial minimality condition}
		\sF_A(E;\B(x,r))\leq \sF_A(F;\B(x,r))+\kappa r^{\alpha+n-1}	
	\end{align}
	whenever $E\Delta F\subsetcc U\cap\B(x,r)$, $x\in U$, and $r<r_0$ (see Section \ref{preliminaries section} for full definitions and notation).
	
	\begin{theorem}[Regularity of almost-minimizers]\label{main theorem}
		Let $n\geq 2$ and $U$ be an open set in $\R^n$. Suppose $\sF_A$ is the anisotropic energy given by \eqref{anisotropic surface energy} for a uniformly elliptic, H\"older continuous matrix-valued function $A=(a_{ij}(x))_{i,j=1}^n$ with H\"older exponent $\alpha\in (0,1)$. If $E$ is a $(\kappa,\alpha)$-almost-minimizer of $\sF_A$ in $U$, that is, it satisfies \eqref{initial minimality condition}, then $U\cap\rb E$ is a $C^{1,\alpha/4}$-hypersurface which is relatively open in $U\cap \p E$, while the singular set of $E$ in $U$,
		\begin{align}
		 	\Sigma(E;U)=U\cap (\p E\setminus \rb E),
		\end{align}
		satisfies the following:
		 \begin{enumerate}
		 	\item[(i)] if $2\leq n\leq 7$, then $\Sigma(E;U)$ is empty;
		 	\item[(ii)] if $n=8$, then $\Sigma(E;U)$ has no accumulation points in $U$;
		 	\item[(iii)] if $n\geq 9$, then $\sH^{s}(\Sigma(E;U))=0$ for $s>n-8$.
		 \end{enumerate}
	\end{theorem}
	
	A regularity result of the form of Theorem \ref{main theorem} was first proved by De Giorgi in \cite{degiorgifront} for minimizers of surface area. De Giorgi worked within the framework of sets of locally finite perimeter which he had introduced and shown to be equivalent to the earlier notion of Caccioppoli sets. Shortly thereafter Reifenberg also proved a similar regularity result for minimizers of surface area in \cite{reifenberg, reifenbergepiperimetric, reifenberganalyticity}. In \cite{tamanini1982,tamanini1984regularity}, Tamanini extended De Giorgi's result to almost-minimizers of perimeter satisfying the minimality condition $P(E;\B(x,r))\leq P(F;\B(x,r))+\kappa r^{\alpha+n-1}$, proving $C^{1,\beta}$-regularity at points in the reduced boundary for each $\beta\in (0,\alpha/2)$. In fact, his result applies with a more general error term.

    The anisotropic surface energies treated by Almgren in \cite{almgrenexistence} are given in terms of the integral of a bounded, continuous, elliptic integrand $f=f(x,\xi)$ over the surface. Here $x$ denotes the spatial variable and $\xi$ denotes the directional variable. Almgren proved that if $f$ is $C^k$ for some $k\geq 3$, then minimal surfaces with respect to $f$ are $C^{k-1}$-regular almost everywhere. Bombieri extended this to the case $k=2$ by showing in \cite{bombieriregularity} that if $f$ is  $C^2$, then minimal surfaces with respect to $f$ are $C^1$-regular almost everywhere. In \cite{schoensimonnewproof}, Schoen and Simon provided an alternate proof of this type of regularity result with weakened hypotheses. They showed that if $f$ is Lipschitz in the spatial variable $x$ and $C^{2,\beta}$ in the directional variable $\xi$, then minimizers are $C^{1,\alpha}$-regular almost everywhere for any $\alpha\in (0,1)$.
	
	A characterization of the singular set for codimension one oriented hypersurfaces as in Theorem \ref{main theorem} was proved in the case of the area integrand $f\equiv 1$ in a series of papers by various authors. Miranda proved in \cite{miranda} that $\Hm$-measure of the singular set is zero. The rest of the results deal with the Bernstein problem which asks about the existence of global minimizers of surface area in $\R^n$. Fleming and Almgren proved some intermediate results of nonexistence in singular minimizing cones in $\R^3$ and $\R^4$, respectively, in \cite{fleming, almgrensomeinterior}. The next result was by De Giorgi in \cite{degiorgibernstein} where he showed that the non-existence of a singular minimal cone in $\R^n$ implies non-existence in $\R^{n-1}$. Simons showed the non-existence of singular minimal cones in dimensions $2\leq n\leq 7$ in \cite{simonsminimal} and Bombieri, De Giorgi, and Giusti demonstrated in \cite{bombieridegiorgigiusti} that Simons' cone
	\begin{align}
	\Sigma=\big\{x\in\R^8: x_1^2+x_2^2+x_3^2+x_4^2=x_5^2+x_6^2+x_7^2+x_8^2\big\}
	\end{align}
	is a singular minimal cone in $\R^8$ with singular set $\{0\}$. Federer concluded in \cite{federer} by proving the Hausdorff dimension of the singular set is less than or equal to $n-8$. In the anisotropic case, it was shown in \cite{almgrenschoensimonregularity} that $\sH^{n-3}$-measure of the singular set is zero for elliptic integrands which are $C^3$.
	
	Surface energies of the particular form of \eqref{anisotropic surface energy} first appeared in the paper \cite{taylorellipsoidal} by Jean Taylor. This is a follow-up paper to her celebrated paper \cite{taylor} in which she proves that the structure of singularities of soap-like minimal surfaces in $\R^3$ are exactly as conjectured by the experimental physicist Joseph Plateau. In \cite{taylorellipsoidal}, she proves that minimizers of $\sF_A$ in $\R^3$ are locally $C^{1,\alpha}$ at regular points and possess a singular set with the same general structure as in the case of surface area minimizers. Taylor worked with varifolds as her notion of surface and only with $2$-dimensional surfaces in $\R^3$. This enabled her to utilize the classification of $2$-dimensional surface area minimizing cones in $\R^3$. Such a classification is not known in higher dimensions. Note that the singularities dealt with by Jean Taylor cannot occur within our setting of sets of locally finite perimeter as $2$-dimensional minimizing cones in $\R^3$ come from non-oriented surfaces. This is why, for instance, we do not have singularities when $2\leq n\leq 7$, even though there are singularities in lower dimensions when working with varifolds.
	
	Allard's work in \cite{allard} established some important results for the Plateau problem in the setting of varifolds, some of which have been generalized to the anisotropic setting. Allard first proved that a varifold $V$ with bounded first-variation $\delta V$ is rectifiable. He then proved regularity by showing that if there are $L^p$-type bounds on the generalized mean-curvature of $V$ for $p$ large enough (depending on the dimension of $V$), then $V$ is locally $C^{1,\alpha}$ for some $\alpha\in (0,1)$ (depending on $p$ and the dimension of $V$) outside a closed singular set of measure zero. A recent breakthrough was made in the setting of anisotropic integrands in \cite{dephillipisderosaghiraldinrectifiability} to prove rectifiability. There, the authors were the first to successfully compute the first-variation $\delta_f V$ with respect to an anisotropic integrand $f$. Using this they showed that if $f$ is an elliptic $C^1$-integrand satisfying the so-called atomic condition (equivalent to ellipticity in codimension one), then a varifold whose anisotropic first-variation $\delta_f V$ is locally bounded is indeed rectifiable. Further regularity is currently not known as the monotonicity formula which is essential in Allard's regularity arguments does not exist for general integrands as demonstrated in \cite{allardcharacterization}. Much of the related relevant literature in contained in \cite{delellisderosaghirladindirectanisotropic, dephilippisderosaghiraldindirect, derosaminimization, dephilippisderosaghiraldinexistence, derosakolasinski}.
	
	Another related problem of interest is volume constrained minimization. Regularity is known in the case of volume constrained perimeter minimizers \cite{gonmasstam} and some results are known in anisotropic settings \cite{parks1984, lamboleyandsicbaldi}.
	
	Let us briefly describe the organization of this paper. We start in Section \ref{preliminaries section} by providing the essential definitions pertaining to sets of locally finite perimeter, our anisotropic surface energies, and almost-minimizers. In Section \ref{existence section} we follow the Direct Method of the Calculus of Variations to establish the existence of minimizers to our formulation of the anisotropic Plateau problem. The rest of the paper is devoted to the study of the regularity of almost-minimizers and to the characterization of the singular set. In Section \ref{basic properties section} we cover a key change of variable that allows us to assume $A(x_0)=I$ (the identity matrix) at a given point $x_0$, as well as prove many important properties of almost-minimizers. These include an almost-monotonicity formula, Theorem \ref{almost monotonicity formula}, volume and perimeter bounds, Proposition \ref{volume and perimeter bounds}, and compactness of the class of almost-minimizers, Proposition \ref{closedness}. Next in Section \ref{excess and height bound section} we define the excess, \eqref{excess def}, an important notion in regularity theory, and recall some of its properties. There we also state the height bound, Proposition \ref{height bound}, which allows us to control the height of the boundary of an almost-minimizer given a small excess assumption. Following this we show in Section \ref{Lipschitz approximation theorem section} that a small excess assumption together with the assumption $A(x_0)=I$ allows us to find a Lipschitz function that well approximates $\p E$ and is `almost-harmonic' with a controlled error, Theorem \ref{Lipschitz approximation theorem}. In Section \ref{reverse Poincare inequality section} we prove a reverse Poincar\'e inequality, Theorem \ref{reverse poincare}, which in Section \ref{tilt excess decay section} we combine with a harmonic approximation of the Lipschitz function from  Section \ref{Lipschitz approximation theorem section} to prove a tilt-excess decay result, Theorem \ref{excess-improvement by tilting}. Finally, in Section \ref{regularity section} we use this and an iteration argument to prove our main regularity result, Theorem \ref{regularity theorem}. We conclude the paper in Section \ref{analysis of singular set section} by using blow-up analysis and a Federer reduction argument to prove the characterization of singular set, Theorem \ref{dimension estimates}.
	
	\section{Preliminaries}\label{preliminaries section}

    We will work in $\R^{n}$ for a fixed $n\geq 2$. The \textbf{open ball} centered at $x\in\R^n$ of radius $r>0$ is defined by
    \begin{align}
    	\B(x,r)=\{y\in\R^n: |y-x|<r\},
    \end{align}
    where $|\cdot|$ denotes the standard Euclidean norm and we write $\B_r$ for $\B(0,r)$. We denote the volume of the $n$-dimensional ball by $\omega_n$.

    Like De Giorgi and Tamanini, we shall also work with sets of locally finite perimeter following much of the notation and definitions given in the insightful expository book by Maggi \cite{maggi}. Throughout this paper we will follow a scheme inspired by the one presented there.
 
     A Lebesgue measurable set $E\subset\R^n$ is said to be of \textbf{locally finite perimeter} if there exists an $\R^n$-valued Radon measure $\mu_E$ (called the \textbf{Gauss-Green measure} of $E$) such that the Gauss-Green formula 
	\begin{align}
		\int_E \nabla \phi\:dx=\int_{\R^n} \phi\: d\mu_E,\qquad \forall \phi\in C_c^1(\R^n) 
	\end{align}
	holds. The induced total-variation measure $|\mu_E|$ is called the \textbf{perimeter measure} of $E$ and is denoted by $P(E;\:\cdot\:)$. The set $E$ is said to be of \textbf{finite perimeter} if $P(E)=P(E;\R^n)<\infty$. The set of those $|\mu_E|$-a.e. $x\in \spt \mu_E$ for which
	\begin{align}
		D_{|\mu_E|}\mu_E(x)=\lim_{r\to 0^+}\frac{\mu_E(\B(x,r))}{|\mu_E|(\B(x,r))}\text{ exists and is in }\SS^{n-1} 
	\end{align}
	is called the \textbf{reduced boundary} of $E$ and is denoted by $\rb E$. The \textbf{measure-theoretic outer unit normal} to $E$ is then defined to be the measurable function $\nu_E:\rb E\to \SS^{n-1}$ given by 
	\begin{align}
		\nu_E(x)=\lim_{r\to 0^+}\frac{\mu_E(\B(x,r))}{|\mu_E|(\B(x,r))}. 
	\end{align}
	The De Giorgi structure theorem states that $\rb E$ is $(n-1)$-rectifiable and that $\mu_E=\nu_E \Hm\rstr\: \rb E$ where $\Hm$ denotes the $(n-1)$-dimensional Hausdorff measure. We may modify a set of locally finite perimeter on and/or up to a set of Lebesgue measure zero without changing its perimeter measure. As a consequence, the topological boundary $\p E$ of a generic set of locally finite perimeter may be quite messy and might not be well related to $\rb E$. However, we may always modify our set of locally finite perimeter $E$ so that $\spt \mu_E=\partial E$ without changing its perimeter measure, in which case $\bar {\rb E}=\p E$ (see \cite[Remark 16.11, Remark 15.3]{maggi}). When discussing boundary regularity of a set of locally finite perimeter we shall always choose this representative of $E$.

\subsection{Anisotropic surface energies with H\"older coefficients}\
	
	Now let's provide precise definitions for the anisotropic energies and almost-minimizers we will study. Denote by $\R^n\otimes\R^n$ the set of real $n\times n$-matrices equipped with the operator norm $||\cdot||$. Let $A=(a_{ij}(x))_{i,j=1}^n$ be a bounded, measurable function on $\R^n$ that takes values in $\R^n\otimes\R^n$. We say that $A$ is \textbf{symmetric} if $A(x)=A(x)^t$ for all $x\in\R^n$, where ${\: \cdot \:}^t$ denotes the matrix transpose.  We say that $A$ is \textbf{uniformly elliptic} if there exist constants $0<\lambda\leq\Lambda<+\infty$ such that
	\begin{align}
		\lambda|\xi|^2\leq \la A(x)\xi,\xi\ra\leq \Lambda |\xi|^2
	\end{align}
	for all $x,\xi\in\R^n$, where $\la\:\cdot\:,\:\cdot\:\ra$ denotes the standard Euclidean inner product. We say that $A$ is \textbf{H\"older continuous} with exponent $\alpha\in(0,1)$ if 
	\begin{align}
		||A||_{C^\alpha}=\sup_{x\not= y} \frac{||A(x)-A(y)||}{|x-y|^\alpha}<\infty 
	\end{align}
	and call $||A||_{C^\alpha}$ the \textbf{H\"older seminorm} of $A$. In particular, 	
	\begin{align}\label{Holder norm ineq}
		||A(x)-A(y)||\leq ||A||_{C^\alpha}|x-y|^\alpha
	\end{align}
	holds for all $x,y\in\R^n$.

	\begin{definition}[$\sF_A$-surface energy] Let $A=(a_{ij}(x))_{i,j=1}^n$ be uniformly elliptic, and H\"older continuous. Given a set of locally finite perimeter $E$ in $\R^n$ and a Borel set $F$, we define the \textbf{$\sF_A$-surface energy} of $E$ in $F$ by
	\begin{align}
		\sF_A(E;F)=\int_{F\cap\rb E} \la A(x)\nu_E(x),\nu_E(x)\ra^{1/2}\dH(x)\in [0,\infty].
	\end{align}
		Note that $\sF_A(E;\:\cdot\:)$ defines a Borel measure on $\R^n$ and we will often denote $\sF_A(E;\R^n)$ by $\sF_A(E)$.
	\end{definition}
	
	\begin{remark}[Symmetry of $A$]\label{symmetry of A}
		We may assume without loss of generality that $A$ is symmetric which we do throughout this paper. We may make this assumption as the equality $\la A(x)\xi,\xi\ra=\big\la \frac{1}{2}\big(A(x)+A(x)^t\big)\xi,\xi\big\ra$ holds for all $x,\xi\in\R^n$. Hence we can always symmetrize $A$ without changing the values of $\sF_A$.
	\end{remark}
	
	\begin{remark}[Ellipticity]
	    The integrand $f(x,\xi)=\la A(x)\xi,\xi\ra^{1/2}$ is elliptic in the sense of Almgren in \cite{almgrenexistence}. In our setting this means that for every bounded set $U$ there is a constant $c>0$ such that for every set of locally finite perimeter $E$, half-space $H$, and $x_0\in U$,
	    \begin{align}\label{ellipticity}
		    \sF_{A(x_0)}(E; \B(x_0,r))-\sF_{A(x_0)}(H; \B(x_0,r))\geq  c\:[\Hm(\rb E\cap \B(x_0,r))-\Hm(\p H\cap \B(x_0,r))]
    	\end{align}
    	whenever $E\Delta H\subsetcc U\cap\B(x_0,r)$, $r>0$. Here $\sF_{A(x_0)}(E;\:\cdot\:)$ denotes the energy associated to the frozen integrand $f_{x_0}(\xi)=\la A(x_0)\xi,\xi\ra^{1/2}$. As Almgren notes, this notion is equivalent to uniform convexity in codimension one as is our case by uniform ellipticity of $A$ and \eqref{ellipticity} holds with $c=\lambda$. Ellipticity ensures that half-spaces are the unique minimizers when compared with their compactly contained variations.	 
	\end{remark}
	
	\begin{remark}[H\"older continuity of integrand of $\sF_A$]
			The integrand $f(x,\xi)=\la A(x)\xi,\xi\ra^{1/2}$ is H\"older continuous with respect to the spatial variable $x$, that is, 
	\begin{align}
		\big|\la A(x)\xi,\xi\ra^{1/2}-\la A(y)\xi,\xi\ra^{1/2}\big|\leq  \frac{1}{2\lambda}\: ||A||_{C^\alpha}|x-y|^\alpha.
	\end{align}
	for all $x,y,\xi\in\R^n$ with $|\xi|=1$. This follows from \eqref{Holder norm ineq} combined with the useful inequality
	\begin{align}\label{integrand ineq1}
		\big|\la A(x)\xi,\xi\ra^{1/2}-\la A(y)\xi,\xi\ra^{1/2}\big|=\frac{\big|\la A(x)\xi,\xi\ra-\la A(y)\xi,\xi\ra\big|}{\la A(x)\xi,\xi\ra^{1/2}+\la A(y)\xi,\xi\ra^{1/2}}\leq \frac{1}{2\lambda}||A(x)-A(y)||
	\end{align}
	for all $x,y,\xi\in\R^n$ with $|\xi|=1$. Note our regularity assumption is much weaker than in \cite{almgrenexistence} where he assumes the integrand $f=f(x,\xi)$ is $C^k$ for some $k\geq 3$ and weaker than the assumption in \cite{schoensimonnewproof} where they assume the integrand $f=f(x,\xi)$ is Lipschitz in $x$.
	\end{remark}
	
	\begin{remark}[Comparability to perimeter]
		$\sF_A(E;\:\cdot\:)$ is comparable to $P(E;\:\cdot\:)$ since it follows for all Borel sets $F$ that
		\begin{align}
			\lambda^{1/2}P(E;F)\leq \sF_A(E;F)\leq \Lambda^{1/2}P(E;F).	\label{comparability to perimeter}
		\end{align}
		by the uniformly ellipticity of $A$. When $A$ equals the identity matrix $I$ we have the isotropic case $\sF_A(E;\:\cdot\:)=P(E;\:\cdot\:)$.
	\end{remark}
	
	\begin{remark} The complement $E^c=\R^n\setminus E$ of a set of locally finite perimeter is also a set of locally finite perimeter with $\mu_{E^c}=-\nu_E\Hm\rstr\rb E$ and so $\sF_A(E^c;\:\cdot\:)=\sF_A(E;\:\cdot\:)$.
	\end{remark}

	\subsection{Notions of almost-minimizers}\
	
	We are interested in studying the boundary regularity of those sets of locally finite perimeter which are almost-minimizers of the $\sF_A$-surface energy in an open set when compared to their local compactly contained variations. Recent work addressing regularity of almost-minimizers for other variational problems can be found in \cite{spolaor2018free,deQT, desvgt, jpsvgminimizers} and the notions of almost-minimizers we consider are similar. 

	Fix universal constants $n\geq 2$, $0<\lambda\leq\Lambda<+\infty$, $\kappa\geq 0$, $\alpha\in (0,1)$ and $r_0\in (0,+\infty)$, and let $A=(a_{ij}(x))_{i,j=1}^n$ be a symmetric, uniformly elliptic, and H\"older continuous with respect to $\lambda$, $\Lambda$, and $\alpha$ and fix an open set $U$ in $\R^n$.
	
	\begin{definition}[$(\kappa,\alpha)$-almost-minimizer of $\sF_A$]\label{almost-minimizer} We say a set of locally finite perimeter $E$ in $\R^n$ is a \textbf{$(\kappa,\alpha)$-(additive) almost-minimizer} of $\sF_A$ in $U$ at scale $r_0$ if $\spt\mu_E=\p E$ and
	\begin{align}\label{minimality condition}
		\sF_A(E; \B(x,r))\leq \sF_A(F;\B(x,r))+\kappa r^{\alpha+n-1}
	\end{align}
		whenever $E\Delta F\subsetcc U\cap \B(x,r)$ where $F$ is a set of locally finite perimeter, $x\in U$, and $r<r_0$.
	\end{definition}
	
	When \eqref{minimality condition} holds with $\kappa=0$, we say that $E$ is a \textbf{local minimizer} of $\sF_A$ in $U$ at scale $r_0$, and when \eqref{minimality condition} holds for all scales $r_0\in (0,+\infty)$, we say that $E$ is a \textbf{minimizer} of $\sF_A$ in $U$.
	Typically we will omit the descriptor \textit{additive} when discussing almost-minimizers. However, we will include it when we wish to  highlight the difference from the following alternative notion of almost-minimality.
	
	\begin{definition}[$(\kappa,\alpha)$-multiplicative almost-minimizer of $\sF_A$] We say a set of locally finite perimeter $E$ in $\R^n$ is a \textbf{$(\kappa,\alpha)$-multiplicative almost-minimizer} of $\sF_A$ in $U$ at scale $r_0$ if $\spt\mu_E=\p E$ and
	\begin{align}
		\sF_A(E; \B(x,r))\leq (1+\kappa r^\alpha)\sF_A(F;\B(x,r))
	\end{align}
		whenever $E\Delta F\subsetcc U\cap \B(x,r)$ where $F$ is a set of locally finite perimeter, $x\in U$, and $r<r_0$.
	\end{definition}
	
	Note that Taylor worked with this notion of multiplicative almost-minimizer in \cite{taylorellipsoidal} but handled a more general error term. We now show that multiplicative almost-minimizers are also additive almost-minimizers. To prove this, we need an upper bound for perimeter bounds of multiplicative almost-minimizers at points in the topological boundary. Whenever we write $C$ we mean a constant (which may change from line to line) that depends only on the universal constants $n,\lambda,\Lambda,\kappa,\alpha,r_0$ and an upper bound for $||A||_{C^\alpha}$, but does not depend on $E$ or $x_0$. If we wish to specify dependence on fewer constants and write for example, $C(n)$ for constants that only depend on $n$.
	 
	 \begin{lemma}\label{multiplicative perimeter bounds} There exists a positive constants $C=C(n,\lambda,\Lambda,\kappa,\alpha,r_0)$ with the following property. If $E$ is a $(\kappa,\alpha)$-multiplicative almost-minimizer of $\sF_A$ in $U$ at scale $r_0$, then for every $x_0\in U\cap\p E$ with $r<d=\min\{\dist(x_0,\p U), r_0\}<\infty$,
		\begin{align}
			\frac{P(E; \B(x_0,r))}{r^{n-1}}\leq C
		\end{align}
	\end{lemma}
	
	\begin{proof}
		Since $\Hm\rstr \rb E$ is Radon, $\Hm(\rb E\cap \p \B(x_0,r))=0$
	for a.e. $r\in (0,d)$. Choose one such radius $r$ and for $s\in (r,d)$ consider the comparison set $F=E\setminus\B(x_0,r)$ in $\B(x_0,s)$. Then $E\Delta F\subset \B(x_0,r)\subsetcc \B(x_0,s)$. It follows from comparability to perimeter \eqref{comparability to perimeter} and the multiplicative almost-minimality of $E$ that
		\begin{align}
			\lambda^{1/2} P(E; \B(x_0,s)) & \leq \sF_A(E; \B(x_0,s))\nonumber\\
			&\leq  (1+\kappa s^{\alpha})\sF_A(E\setminus \B(x_0,r); \B(x_0,s))\nonumber\\
			&\leq    (1+\kappa r_0^\alpha)\Lambda^{1/2} P(E\setminus \B(x_0,r); \B(x_0,s)).
		\end{align}
		Hence
		\begin{align}
			P(E; \B(x_0,s) )& \leq C\: P(E\setminus \B(x_0,r); \B(x_0,s))\nonumber\\
			&= C\big(\Hm(E^{(1)}\cap \p \B(x_0,r))+P(E;\B(x_0,s)\setminus\cl{\B}(x_0,r))\big) 
		\end{align}
		since $\Hm(\rb E\cap \p \B(x_0,r))=0$. Sending $s\to r^+$ and noting	 $\Hm(E^{(1)}\cap \p \B(x_0,r))\leq n\omega_{n}r^{n-1}$ gives \begin{align}
			P(E; \B(x_0,r)) \leq C\Hm(E^{(1)}\cap \p \B(x_0,r))\leq C r^{n-1}.
		\end{align}
		By density of these radii, this holds for all $r\in (0,d)$.
		\end{proof}
		
	\begin{proposition}[Multiplicative almost-minimizers are (additive) almost-minimizers]\label{multiplicative are additive}
			If $E$ is a $(\kappa,\alpha)$-multiplicative almost-minimizer of $\sF_{A}$ in $U$ at scale $r_0$, then for each open set $V\subsetcc U$, there is a constant $\kappa'=\kappa'(n,\lambda,\Lambda,\kappa,\alpha,r_0)$ such that $E$ is a $(\kappa',\alpha)$-(additive) almost-minimizer of $\sF_A$ in $V$ at scale $r_0'=\min\{(1/2)r_0,(1/4)\dist(V,U^c)\}$.
		\end{proposition}
		
		\begin{proof}
			Let $E\Delta F\subsetcc \B(x,r)\cap V$, $x\in V$, and $r<r_0'$. Suppose $E$ is a $(\kappa,\alpha)$-multiplicative almost-minimizer of $\sF_{A}$ in $U$ at scale $r_0$. The minimality condition is trivially satisfied if $\sF_A(E;\B(x,r))\leq \sF_A(F;\B(x,r))$ or $P(E;\B(x,r))=0$. So suppose $\sF_A(F;\B(x,r))\leq \sF_A(E;\B(x,r))$ and $P(E;\B(x,r))>0$. Then there is $y\in\B(x,r)\cap \p E$. So by Lemma \ref{multiplicative perimeter bounds}, which applies since $2r<r_0$ and $\B(y,2r)\subset \B(x,4r)\subset U$, we have $P(E;\B(x,r))\leq P(E;\B(y,2r))\leq C(2r)^{n-1}$. Hence by comparability to perimeter \eqref{comparability to perimeter} we have $\sF_A(F;\B(x,r))\leq \Lambda^{1/2}P(E;\B(x,r))\leq C r^{n-1}$. It follows that
			\begin{align}
				\sF_A(E;\B(x,r))\leq \sF_A(F;\B(x,r))+\kappa r^\alpha \sF_A(F;\B(x,r))\leq \sF_A(F;\B(x,r))+\kappa' r^{\alpha+n-1}  
			\end{align} 
			for some $\kappa'=\kappa'(n,\lambda,\Lambda,\kappa,\alpha, r_0)$.	
		\end{proof}
		
		Thus Proposition \ref{multiplicative are additive} implies that any interior regularity results for (additive) almost-minimizers shall also apply to multiplicative almost-minimizers. We shall focus on proving a regularity theorem for (additive) almost-minimizers and shall henceforth only work with (additive) almost-minimizers which we simply refer to as almost-minimizers.
	
	\section{{Existence of Anisotropic Minimizers}}\label{existence section}

	Our first order of business is to establish existence of solutions to the anisotropic Plateau problem for $\sF_A$. The existence of anisotropic minimizers in the setting of varifolds and currents is known in general in the framework of varifolds and currents (see \cite[Chapter 5]{federergmt}) which should imply existence of minimizers of $\sF_A$ in the framework of sets of locally finite perimeter. However, for completeness, we present our own full proof of this result in our setting. Additionally, the lower semicontinuity result of Proposition \ref{lsc} will prove useful at several places in the regularity portion of our paper.
	
	Let $A=(a_{ij}(x))_{i,j=1}^n$ be a symmetric, uniformly elliptic, continuous function on $\R^n$ with values in $\R^n \otimes\R^n$ (we do not need H\"older continuity to show existence of minimizers) and consider the $\sF_A$-surface energy. Fix an open bounded set $U$ and a set of finite perimeter $E_0$ in $\R^n$. The \textbf{anisotropic Plateau problem} for $\sF_A$ in $U$ with \textbf{boundary data} $E_0$ is to show that the infimum
		\begin{align}\label{anistropic pp}
			\gamma_A(E_0, U)=\inf \Big\{\sF_A(E): E \text{ a set of finite perimeter in } \R^n \text{ with } E\setminus U=E_0\setminus U\Big\}
		\end{align}
		is attained (Cf. \cite[(12.29)]{maggi}). That is, we minimize $\sF_A$ in $\R^n$ among those sets of finite perimeter which agree with $E_0$ outside of $U$.
	
	To show that \eqref{anistropic pp} is achieved by a set of finite perimeter, we follow the Direct Method of the Calculus of Variations. This consists of $(i)$ taking a sequence $\{E_h\}_{h\in\NN}$ of competitors such that $\sF_A(E_h)\to \gamma_A(E_0,U)$, $(ii)$ using a key compactness result in an appropriate topology to extract a subsequence $\{E_{h(k)}\}_{k\in\NN}$ converging to some competitor $E$ satisfying $E\setminus U=E_0\setminus U$, and $(iii)$ applying lower semicontinuity of $\sF_A$ with respect to the convergence in the chosen topology which shows that $\sF_A(E)$ equals the infimum $\gamma_A(E_0,U)$ in \eqref{anistropic pp}.

	\subsection{Compactness of sets of locally finite perimeter}\
	
	The first key ingredient of the Direct Method is compactness of our class of admissible competitors. One of the primary reasons that sets of locally finite perimeter provide a suitable setting to work on geometric variational problems is that they possess compactness with respect to local convergence of sets. Let's recall the definition of this convergence and a known compactness theorem for sets of locally finite perimeter.
	
	We say that a sequence of sets of locally finite perimeter $\{E_h\}_{h\in\NN}$ in $\R^n$ \textbf{converges locally} to $E$ (and write $E_h\loc E$) if 
	\begin{align}
		|(E_h\Delta E)\cap K|\to 0\text{ as } h\to \infty
	\end{align}
	for each compact $K\subset \R^n$, and say $\{E_h\}_{h\in\NN}$ \textbf{converges} to $E$ (and write $E_h\to E$) if
		\begin{align}
		|E_h\Delta E|\to 0\text{ as } h\to \infty.
	\end{align}
	Recall that $E\Delta F=(E\setminus F)\cup (F\setminus E)$ and that $|\cdot|$ denotes Lebesgue measure on $\R^n$. 
	
	\begin{theorem}[Compactness from perimeter bounds, {\cite[Theorem 12.26]{maggi}}]\label{compactness}
		If $R>0$ and $\{E_h\}_{h\in\NN}$ are sets of finite perimeter in $\R^n$, with
		\begin{align}
			E_h\subset \B_R,\ \forall h\in\NN,\qquad \text{ and } \qquad \sup_{h\in\NN} P(E_h)<\infty,
		\end{align}
		then there exist a set $E$ of finite perimeter in $\R^n$ and indices $h(k)\to\infty$ as $k\to\infty$, with 
		\begin{align}
			E_{h(k)}\to E,\qquad \mu_{E_{h(k)}}\wkly\mu_E,\qquad\text{ and }\qquad E\subset \B_R.
		\end{align}
	\end{theorem}
	
	\subsection{Lower semicontinuity of $\sF_A$}\ 
	
	The second key ingredient of the Direct Method is to show lower semicontinuity of the $\sF_A$-surface energy. Here we have some work to do and start with a couple lemmas. The first lemma deals with lower semicontinuity when $A$ is constant, while the second one is a technical lemma we need in the proof when $A$ is no longer constant.

	\begin{lemma}[Lower semicontinuity for constant $A$]\label{lsc constant}
		If $A$ is a constant, uniformly elliptic matrix, and  $\{E_h\}_{h\in\N}$ and $E$ are sets of locally finite perimeter with $\nu_{E_h}\Hm\rstr\:\rb E_h\wkly \nu_E\Hm\rstr\:\rb E$, then for any open set $U$,
		\begin{align}
			\sF_A(E; U)\leq\liminf_{h\to\infty} \sF_A(E_h;U).
		\end{align}
	\end{lemma}

	\begin{proof}
		By Remark \ref{symmetry of A} we may assume $A$ is symmetric and by uniform ellipticity its eigenvalues are positive.  So by the spectral theorem we can write $A=VDV^{-1}$ where $D$ is a diagonal matrix with the eigenvalues of $A$ and where $V$ is the matrix of corresponding orthonormal eigenvectors. Setting $A^{1/2}=VD^{1/2}V^{-1}$, we have $A=A^{1/2}A^{1/2}$ with $A^{1/2}$ symmetric since $V^{-1}=V^t$. So $\la A\xi,\xi\ra^{1/2}=|A^{1/2}\xi|$. Define $\R^n$-valued Radon measures on $\R^n$,
		\begin{align}
		\mu_h=A^{1/2}\nu_{E_h}\Hm\rstr\: \rb E_h\qquad\text{ and }\qquad\mu=A^{1/2}\nu_{E}\Hm\rstr\: \rb E.
		\end{align}  
		It then follows that $\mu_h\wkly\mu$ because, given any $\phi\in C_c(\R^n;\R^n)$, we have $A^{1/2}\phi\in C_c(\R^n;\R^n)$ and thus
		\begin{align}
			\lim_{h\to\infty}\int \phi\cdot d\mu_h &=\lim_{h\to\infty}\int_{\rb E_h} \la\phi, A^{1/2}\nu_{E_h}\ra\dH =\lim_{h\to\infty}\int_{\rb E_h} \la A^{1/2}\phi,\nu_{E_h}\ra\dH\nonumber\\
			&= \int_{\rb E} \la A^{1/2}\phi, \nu_{E}\ra\dH
			=\int_{\rb E} \la \phi, A^{1/2}\nu_{E}\ra \dH =\int \phi\cdot d\mu.
		\end{align}
		By lower semicontinuity of the total variation of weak-star convergent vector-valued Radon measures (\cite[Proposition 4.19]{maggi}), we have
		\begin{align}
			\sF_A(E; U) &=\int_{U\cap \rb E} |A^{1/2}\nu_{E}|\dH=|\mu|(U)\leq 		\liminf_{h\to\infty} |\mu_h|(U)\nonumber\\
			&=\liminf_{h\to\infty}\int_{U\cap \rb E_h} |A^{1/2}\nu_{E_h}|\dH=\liminf_{h\to\infty} \sF_A(E_h; U)
		\end{align}
		which concludes the proof.
	\end{proof}

	\begin{lemma}\label{technical lemma}
		Let $\{\Phi_h\}_{h\in\N}$ and $\Phi$ be Radon measures on $\R^n$ and $\phi\in C_c(\R^n; [0,\infty))$ such that\\ $\displaystyle\limsup_{h\to\infty} \Phi_h(\{\phi>0\})<\infty$. Then the following two statements hold:
		\begin{enumerate}
			\item[(i)] If $\displaystyle \Phi(U)\leq\liminf_{h\to\infty} \Phi_h(U)$ for any open set $U$, then
			\begin{align}
				(\phi\Phi)(U)\leq\liminf_{h\to\infty} (\phi\Phi_h)(U)
			\end{align}
			for any open set $U$ in $\R^n$.
			\item[(ii)] If $\displaystyle\limsup_{h\to\infty} \Phi_h(K)\leq \Phi(K)$ for any compact set $K$, then 
			\begin{align}
				\limsup_{h\to\infty} (\phi\Phi_h)(K)\leq (\phi\Phi)(K)
			\end{align}
			for any compact set $K$ in $\R^n$.
		\end{enumerate}
	\end{lemma}
	
	\begin{proof}
		Let $\epsilon>0$ and choose $0=t_0<t_1<\cdots<t_{N-1}<\sup\phi<t_N$ such that $t_j-t_{j-1}<\epsilon$ and $\Phi(\{\phi=t_j\})=0$ for $j=1,\dots,N$. This is possible since $\Phi$ is Radon and so $\Phi(\{\phi=t\})>0$ for at most countably many $t$. Set
		\begin{align}
		U_j=\{t_{j-1}<\phi<t_j\}\qquad\text{and}\qquad K_j=\bar U_j
		\end{align}
		and note that the $U_j$'s are open and the $K_j$'s are compact.

		Proof of $(i)$: Assume the hypothesis and let $U$ be an open set. Observe that
		\begin{align}
			(\phi\Phi)(U)=\int_U \phi\: d\Phi=\sum_{j=1}^N \int_{U\cap U_j} \phi\: d\Phi+\sum_{j=0}^{N-1} t_j\Phi(U\cap\{\phi=t_j\})=\sum_{j=1}^N \int_{U\cap U_j} \phi\: d\Phi
		\end{align}
		since $\Phi(U\cap \{\phi=t_j\})=0$ for $j=1,\dots, N-1$.
		Since $\phi<t_j$ on $U\cap U_j$ and $\Phi(U\cap U_j)\leq\liminf_{h} \Phi_h(U\cap U_j)$ for $j=1,\dots,N$, we have
		\begin{align}
			(\phi\Phi)(U)\leq \sum_{j=1}^N \int_{U\cap U_j} t_j\: d\Phi\leq\sum_{j=1}^N t_j\liminf_{h\to\infty} \Phi_h(U\cap U_j)=\liminf_{h\to\infty} \sum_{j=1}^N t_j\Phi_h(U\cap U_j),
		\end{align}
		where we used the property that $\liminf_{h} a_h+\liminf_h b_h\leq \liminf_{h} (a_h+b_h)$ for any sequences $\{a_h\}, \{b_h\}$. Note that $t_j<t_{j-1}+\epsilon<\phi+\epsilon$ on $U\cap U_j$ and so
		\begin{align}
			(\phi\Phi)(U)&\leq \liminf_{h\to\infty} \sum_{j=1}^N t_j\Phi_h(U\cap U_j)\leq \liminf_{h\to\infty} \sum_{j=1}^N \int_{U\cap U_j} (\phi+\epsilon)\: d\Phi_h\nonumber\\
			&\leq \liminf_{h\to\infty} (\phi \Phi_h)(U)+\epsilon\limsup_{h\to\infty} \Phi_h(U\cap\{\phi>0\})
		\end{align}
		since $U\cap \{\phi>0\}=\bigcup_{j=1}^N U\cap U_j$. Sending $\epsilon\to 0^+$ completes the proof of $(i)$.
	
		Proof of $(ii)$: Assume the hypothesis and let $K$ be a compact set. Recalling $K_j=\bar U_j$, observe that
		\begin{align}
			(\phi \Phi)(K)=\int_K \phi\:d\Phi=\sum_{j=1}^N \int_{K\cap K_j} \phi\:d\Phi-\sum_{j=0}^{N-1} t_j \Phi(K\cap\{\phi=t_j\})=\sum_{j=1}^N \int_{K\cap K_j} \phi\:d\Phi,
		\end{align}
		since $\Phi(K\cap\{\phi=t_j\})=0$ for $j=1,\dots, N-1$, and
		\begin{align}
			(\phi \Phi_h)(K)=\int_K \phi\:d\Phi_h=\sum_{j=1}^N \int_{K\cap K_j} \phi\:d\Phi_h-\sum_{j=0}^{N-1} t_j \Phi_h(K\cap\{\phi=t_j\})\leq \sum_{j=1}^N \int_{K\cap K_j} \phi\:d\Phi_h.
		\end{align}
		It follows that
		\begin{align}
			\limsup_{h\to\infty}(\phi \Phi_h)(K)	\leq \limsup_{h\to\infty}\sum_{j=1}^N \int_{K\cap K_j} \phi\:d\Phi_h\leq \sum_{j=1}^N\limsup_{h\to\infty} \int_{K\cap K_j} \phi\:d\Phi_h,
		\end{align}
	  	where we used the property that $\limsup_h(a_h+b_h)\leq \limsup_{h} a_h+\limsup_h b_h$ for any sequences $\{a_h\}, \{b_h\}$. Since $\phi<t_j$ on $K\cap K_j$ and $\limsup_{h} \Phi_h(K\cap K_j)\leq \Phi(K\cap K_j)$, we have
	  	\begin{align}
		 	\limsup_{h\to\infty}(\phi\Phi_h)(K)\leq \sum_{j=1}^N\limsup_{h\to\infty}\int_{K\cap K_j}\phi\:d\Phi_h\leq\sum_{j=1}^N t_j\Phi(K\cap K_j).
		\end{align}
	  	Note that $t_j<t_{j-1}+\epsilon\leq\phi+\epsilon$ on $K\cap K_j$ and so
	 	\begin{align}
		 	\limsup_{h\to\infty}(\phi\Phi_h)(K)\leq \sum_{j=1}^N (\phi+\epsilon)\Phi(K\cap K_j)=(\phi\Phi)(K)+\epsilon\Phi(K\cap\spt \phi)
		\end{align}
	 	where we used $\Phi(K\cap\{\phi=t_j\})=0$ for $j=1,\dots,N-1$. Sending $\epsilon\to 0^+$ completes the proof of $(ii)$.
	\end{proof}

	With these lemmas in hand, we are now ready to state and prove the lower semicontinuity of $\sF_A$.
	
	\begin{proposition}[Lower semicontinuity of $\sF_A$]\label{lsc} Let $A=(a_{ij}(x))_{i,j=1}^n$ be a symmetric, uniformly elliptic, continuous function on $\R^n$ with values in $\R^n\otimes\R^n$. Suppose $\{E_h\}_{h\in\N}$ is a sequence of sets of locally finite perimeter in $\R^n$ and $E$ is Lebesgue measurable, with
		\begin{align}
			E_h\loc E,\qquad\text{ and }\qquad \limsup_{h\to\infty} P(E_h;K)<\infty
		\end{align}
		for every compact set $K$ in $\R^n$. Then $E$ is a set of locally finite perimeter in $\R^n$ with 
		\begin{align}
			\nu_{E_h}\Hm\rstr\: \rb E_h\wkly \nu_E\Hm\rstr\: \rb E,
		\end{align}
		and for any open set $U$ in $\R^n$,
		\begin{align}
			\sF_A(E; U)\leq\liminf_{h\to\infty} \sF_A(E_h; U).
		\end{align}
	\end{proposition}

	\begin{proof}
		That $E$ is of locally finite perimeter and $\nu_{E_h}\Hm\rstr\: \rb E_h\wkly\nu_E\Hm\rstr\: \rb E$ follow from \cite[Proposition 12.15]{maggi}.
		Thus we need only to prove the lower semicontinuity. 
		
		First assume that $U$ is bounded. By taking a subsequence of $\{\sF_A(E_h; U)\}_{h\in\NN}$, we may assume up to relabeling that
		\begin{align}
			\lim_{h\to\infty}\sF_A(E_h; U)=\liminf_{h\to\infty} \sF_A(E_h; U)<\infty.
		\end{align}
		Note this subsequence depends on $U$ but this is not an issue. Since $\limsup_{h\to\infty} P(E_h;K)<\infty$ for every compact set $K$, there is a further subsequence $\{E_{h(k)}\}_{k\in\NN}$ and a Radon measure $\Psi$ such that $\Hm\rstr\: \rb E_{h(k)}\wkly\Psi$ as $k\to\infty$ (see \cite[Remark 4.35]{maggi}).
		
		Let $V\subsetcc U$ be open and fix $\epsilon>0$. Since $A$ is uniformly continuous on $\bar U$, there exists  $0<r<\dist(\bar V, U^c)$ such that for any $x,y\in \bar U$,  we have $||A(x)-A(y)||<\epsilon$ whenever $|x-y|<r$. Thus by the inequality \eqref{integrand ineq1}, for any $x,y\in\bar U$,
		\begin{align}
			\la A(x)\xi,\xi\ra^{1/2}\leq \la A(y)\xi,\xi\ra^{1/2}+\frac{1}{2\lambda}\epsilon
		\end{align}
		whenever $|x-y|<r$ and $|\xi|=1$.
			
		Since $V$ is compactly contained in $U$ and $\spt \Phi_E=\p E=\overline{\rb E}$, there exist finitely many balls $\{\B(x_j,r)\}_{j=1}^N$ each of radius $r$ and center $x_j\in V\cap \rb E$ which cover $\bar V\cap \p E$. Take a partition of unity $\{\phi_j\}_{j=1}^N$ with $\phi_j\in C_c(\B(x_j,r),[0,1])$ such that $\sum_{j=1}^N \phi_j=1$ on $V\cap \rb E$ and $\sum_{j=1}^N \phi_j\leq 1$ elsewhere.  It follows that 
		\begin{align}
			\liminf_{h\to\infty} \sF_A(E_h;U) &=\lim_{k\to\infty}\sF_A(E_{h(k)};U)\nonumber\\
			& \geq  \lim_{k\to\infty}\sum_{j=1}^N \int_{\rb E_{h(k)}} \phi_j \la A(x)\nu_{E_{h(k)}},\nu_{E_{h(k)}}\ra^{1/2}\dH\nonumber\\
			&\geq \lim_{k\to\infty}\sum_{j=1}^N \bigg[\int_{\rb E_{h(k)}} \phi_j \la A(x_j)\nu_{E_{h(k)}},\nu_{E_{h(k)}}\ra^{1/2}\dH-\frac{1}{2\lambda}\epsilon\int_{\rb E_{h(k)}}\phi_j\dH\bigg]\nonumber\\
			&\geq \sum_{j=1}^N \liminf_{k\to\infty}\bigg[\int_{\rb E_{h(k)}} \phi_j \la A(x_j)\nu_{E_{h(k)}},\nu_{E_{h(k)}}\ra^{1/2}\dH-\frac{1}{2\lambda}\epsilon\int_{\rb E_{h(k)}}\phi_j\dH\bigg]\nonumber\\
			&\geq \sum_{j=1}^N \bigg[\int_{\rb E} \phi_j \la A(x_j)\nu_{E},\nu_{E}\ra^{1/2}\dH-\frac{1}{2\lambda}\epsilon\limsup_{k\to\infty} \int_{\rb E_{h(k)}}\phi_j\dH\bigg],
		\end{align}
		where in the last inequality, for each $j=1,\dots, N$, we applied part $(i)$ of Lemma \ref{technical lemma} to $\phi_j$ and the measures $d\Phi_k=\la A(x_j)\nu_{E_{h(k)}},\nu_{E_{h(k)}}\ra^{1/2}\dH\rstr\ \rb E_{h(k)}$ and $d\Phi=\la A(x_j)\nu_E,\nu_E\ra^{1/2}\dH\rstr\ \rb E$ which by Lemma \ref{lsc constant} satisfies the lower semicontinuity hypothesis.
		By part $(ii)$ of Lemma \ref{technical lemma}, applied to 
		$\Hm\rstr\: \rb E_{h(k)}\wkly\Psi$,
		\begin{align}
			\limsup_{k\to\infty} \int_{\rb E_{h(k)}} \phi_j\dH\leq \int_{\spt \phi_j} \phi_j\: d\Psi
		\end{align}
		for each $j=1,\dots, N$, and so 
		\begin{align}
			\sum_{j=1}^N\limsup_{k\to\infty}\int_{\rb E_{h(k)}}\phi_j\dH\leq \sum_{j=1}^N \int_{\spt \phi_j} \phi_j\: d\Psi\leq \Psi(\bar U)
		\end{align}
		since $\bigcup_{j=1}^N\spt \phi_j\subset \bar U$ and $\sum_{j=1}^N \phi_j\leq 1$. It follows that
		\begin{align}
			\liminf_{h\to\infty}\sF_A(E_h; U)
			&\geq\bigg[\sum_{j=1}^N\int_{\rb E} \phi_j \la A(x_j)\nu_{E},\nu_{E}\ra^{1/2}\dH\bigg]-\frac{1}{2\lambda}\epsilon\Psi(\bar U)\nonumber\\
			&\geq
			\sum_{j=1}^N\bigg[\int_{\rb E} \phi_j \la A(x)\nu_{E},\nu_{E}\ra^{1/2}\dH-\frac{1}{2\lambda}\epsilon\int_{\rb E}\phi_j\dH\bigg]-\frac{1}{2\lambda}\epsilon\Psi(\bar U)\nonumber\\
			&\geq \sF_A(E;V)-\frac{1}{2\lambda}\epsilon\Big[\Hm(\bar U\cap \rb E)+\Psi(\bar U)\Big]
		\end{align}
		since, as above,
		\begin{align}
			\sum_{j=1}^N\int_{\rb E}\phi_j\dH\leq \sum_{j=1}^N \int_{\spt \phi_j\cap \rb E} \phi_j\dH\leq \Hm(\bar U\cap\rb E)
		\end{align} 
		by $\bigcup_{j=1}^N\spt \phi_j\subset \bar U$ and $\sum_{j=1}^N \phi_j\leq 1$.
		Letting $\epsilon\to 0^+$, we obtain $
			\sF_A(E;V)\leq \liminf_{h} \sF_A(E_h; U)$. Approximating $U$ by $V$ from below and using monotone convergence, we obtain $\sF_A(E;U)\leq \liminf_{h} \sF_A(E_h; U)$.
			
		For the case when $U$ is unbounded, we have $
			\sF_A(E;V)\leq \liminf_{h\to\infty} \sF_A(E_h; V)\leq \liminf_{h\to\infty} \sF_A(E_h; U)$
		for every bounded open set $V\subset U$. We conclude by approximating $U$ from below by bounded open sets $V\subset U$ and using monotone convergence.	
	\end{proof}
	
	\subsection{Existence theorem of minimizers for
	$\sF_A$}\ 
	
	We now show that the anisotropic Plateau problem for $\sF_A$ given by \eqref{anistropic pp} has a solution. We follow a similar approach as \cite[Theorem 12.29]{maggi}.
	
	\begin{theorem}[Existence of minimizers for the anisotropic Plateau problem for $\sF_A$] Let $A=(a_{ij}(x))_{i,j=1}^n$ be a uniformly elliptic, continuous function on $\R^n$ with values in $\R^n\otimes\R^n$, let $E_0$ be a set of finite perimeter in $\R^n$, and let $U$ be an open bounded set. There exists a set of finite perimeter $E$ in $\R^n$ with $E\setminus U=E_0\setminus U$ such that $\sF_A(E)=\gamma_A(E_0, U)$ from \eqref{anistropic pp}. In particular, $E$ is a minimizer of $\sF_A$ in $U$.
	\end{theorem}
	
	\begin{proof}
		Let $\{E_h\}_{h\in\NN}$ be a sequence of sets of finite perimeter in $\R^n$ with $E_h\setminus U=E_0\setminus U$ such that $\sF_A(E_h)\to\gamma_A(E_0;U)$ as $h\to \infty$ and $\sF_A(E_h)\leq \sF_A(E_0)<\infty$. Consider $M_h=E_h\Delta E_0\subset U$. Noting that by \cite[Theorem 16.3]{maggi}, (in particular, by \cite[Exercise 16.5]{maggi}),
		\begin{align}
			P(M_h)\leq P(E_h)+P(E_0)\leq 2\lambda^{-1/2}\sF_A(E_0)<\infty.
		\end{align}
		Hence $\sup_h P(M_h)<\infty$. Choose $R>0$ with $U\subset \B_R$ so that $M_h\subset \B_R$. By Theorem \ref{compactness}, there is a set of finite perimeter $M\subset \B_R$ and $h(k)\to \infty$ as $k\to\infty$ such that $M_{h(k)}\to M$. Up to modifying by a set of measure zero $M\subset U$. Set $E=M\Delta E_0$. Then $E\setminus U=E_0\setminus U$ and note that $E_h=M_h\Delta E_0$. Hence $E_{h(k)}\to E$ since $|E_{h(k)}\Delta E|=|M_{h(k)}\Delta M|\to 0$ as $k\to\infty$. Finally, observe that
		\begin{align}
			\limsup_{k\to\infty}  P(E_{h(k)})\leq \lambda^{-1/2}\limsup_{k\to\infty}  \sF_A(E_{h(k)})\leq \lambda^{-1/2}\sF_A(E_0)<\infty.
		\end{align}
		Consequently, by Proposition \ref{lsc},
		\begin{align}
			\gamma_A(E_0; U)\leq \sF_A(E)\leq\liminf_{k\to\infty}\sF_A(E_{h(k)})=\gamma_A(E_0; U).
		\end{align}
		Thus $\sF_A(E)=\gamma_A(E_0; U)$.
		
		Suppose $E\Delta F\subsetcc U\cap\B(x,r)$, $x\in U$, and $r<r_0$. Then $F\setminus U=E_0\setminus U$ and so $\sF_A(E)\leq \sF_A(F)$.  Since $E\Delta F\subsetcc \B(x,r)$, we have $\sF_A(E;\R^n\setminus \B(x,r))=\sF_A(F;\R^n\setminus \B(x,r))$. Hence $\sF_A(E;\B(x,r))\leq \sF_A(F;\B(x,r))$.
	\end{proof}
	
	\section{Basic Properties of Almost-Minimizers}\label{basic properties section}
	
		In this section we begin our journey toward proving regularity of almost-minimizers by proving some fundamental properties that almost-minimizers possess and which play a crucial role in our excess-decay argument.

	\subsection{Invariance under an affine change variable}\
	
		One of the key ideas that allows us to adapt the standard excess-decay arguments for perimeter minimizers to the setting is a certain change of variable.
		
		If $A$ is a constant matrix, then by symmetry we can orthogonally diagonalize $A$ and write $A=VDV^{-1}$, where $D$ is a diagonal matrix with the eigenvalues of $A$ and where $V$ is the matrix of corresponding orthonormal eigenvectors. By ellipticity the eigenvalues of $A$ are bounded below and above by the positive constants $\lambda$ and $\Lambda$. Setting $A^{1/2}=V D^{1/2} V^{-1}$, we have $A =A^{1/2}A^{1/2}$. Note that $A^{1/2}$ and $A^{-1/2}$ are symmetric since $V^{-1}=V^t$. In the coordinate system of $V$, the matrix $A^{-1/2}$ is diagonal and so almost-minimizers of $\sF_A$ can be viewed as almost-minimizers of perimeter when deformed by the change of variable $y=T(x)=A^{-1/2}x$ (see Proposition \ref{invariance} below). Of course this change of variable preserves any regularity of almost-minimizers and we know by Tamanini's work in \cite{tamanini1984regularity} that almost-minimizers of perimeter are H\"older continuously differentiable. 
		
		If $A=(a_{ij}(x))_{i,j=1}^n$ varies H\"older continuously, then almost-minimizers of $\sF_A$ cannot simply be viewed as almost-minimizers of perimeter since deformation varies from point to point. However, philosophically it is reasonable to expect a similar amount of regularity since the deformation varies H\"older continuously. In subsequent sections we will prove decay estimates for the excess at points $x_0\in\p E$ with small excess on some ball or cylinder. In the proofs of these estimates it will be convenient to be able to assume that $A(x_0)=I$, allowing us to think of $\sF_A$ as a perturbation of perimeter at the point $x_0$. In order to make this assumption, we shall do the following change of variable which was similarly used in \cite{desvgt, jpsvgminimizers} for almost-minimizers of other types of functionals involving coefficients $A=(a_{ij}(x))_{i,j=1}^n$. As in the constant case, for each fixed $x_0\in\R^n$ we can write $A(x_0)=VDV^{-1}$, where $D$ is a diagonal matrix with the eigenvalues of $A(x_0)$ and where $V$ is the matrix of corresponding orthonormal eigenvectors. Setting $A^{1/2}(x_0)=V D^{1/2} V^{-1}$, we have that $A^{1/2}(x_0)$ and $A^{-1/2}(x_0)$ are symmetric since $V^{-1}=V^t$ and satisfy
	\begin{align}
		\lambda^{1/2}|\xi|\leq |A^{1/2}(x_0)\xi|\leq \Lambda^{1/2}|\xi|,\qquad \Lambda^{-1/2}|\xi|\leq |A^{-1/2}(x_0)\xi|\leq \lambda^{-1/2}|\xi|.
	\end{align}	
	In particular, $\lambda^{1/2}\leq ||A^{1/2}(x_0)\xi||\leq \Lambda^{1/2}$ and $\Lambda^{-1/2}\leq ||A^{-1/2}(x_0)||\leq\lambda^{-1/2}$. 
	
	Define the affine change of variable $T_{x_0}$ at $x_0\in\p E$ by
	\begin{align} 
		T_{x_0}(x)=A^{-1/2}(x_0)(x-x_0)+x_0,\qquad 
		T_{x_0}^{-1}(y)=A^{1/2}(x_0)(y-x_0)+x_0,
	\end{align}
	and define 
	\begin{align} 
	E_{x_0}=T_{x_0}(E),\qquad U_{x_0}=T_{x_0}(U),\qquad A_{x_0}(y)=A^{-1/2}(x_0)A(T_{x_0}^{-1}(y))A^{-1/2}(x_0).
	\end{align}
	Note that $T_{x_0}(x_0)=x_0$, $A_{x_0}(x_0)=I$, while $A_{x_0}$ is symmetric, uniformly elliptic with constants $0<\lambda/\Lambda\leq \Lambda/\lambda<+\infty$ and H\"older continuous with exponent $\alpha$ and H\"older seminorm $||A_{x_0}||_{C^\alpha}\leq (\Lambda^{\alpha/2}/\lambda)\: ||A||_{C^\alpha}$. The uniform ellipticity constants follow from
		\begin{align}
			(\lambda/\Lambda)|\xi|^2\leq\lambda|A^{-1/2}(x_0)\xi|^2\leq\la A(T_{x_0}^{-1}(y))A^{-1/2}(x_0)\xi,A^{-1/2}(x_0)\xi\ra\leq \Lambda|A^{-1/2}(x_0)\xi|^2\leq (\Lambda/\lambda)|\xi|^2
		\end{align}
		and the bound on the H\"older norm follows from estimate that for all $x,y\in\R^n$ there holds
	\begin{align}
		||A_{x_0}(x)-A_{x_0}(y)||&=||A^{-1/2}(x_0)\big[A(T_{x_0}^{-1}(x))-A(T_{x_0}^{-1}(y))\big]A^{-1/2}(x_0))||\nonumber\\
		&\leq \lambda^{-1/2}||A(T_{x_0}^{-1}(x))-A(T_{x_0}^{-1}(y))||\lambda^{-1/2}\nonumber\\
		&\leq \lambda^{-1}||A||_{C^\alpha}|T_{x_0}^{-1}(x)-T_{x_0}^{-1}(y)|^\alpha
		\nonumber\\
		&\leq \lambda^{-1}||A||_{C^\alpha}\Lambda^{\alpha/2} |x-y|^\alpha. 
	\end{align}
	Thus constants for $A_{x_0}$ depend on the same universal constants as $A$.
	
	The \textbf{ellipsoid} at $x_0\in\R^n$ of radius $r>0$ is defined by
	\begin{align}
		\W_{x_0}(x_0,r)=T_{x_0}^{-1}(\B(x_0,r)).
	\end{align}
	We use $\W_{x_0}$ for our notation as this is the Wulff shape, introduced in \cite{wulff}, for the integrand $f(x_0,\xi)=\la A(x_0)\xi,\xi)^{1/2}$. The ellipsoid $\W_{x_0}(x_0,r)$ has axial directions corresponding to the eigenvectors of $A^{1/2}(x_0)$ and axial lengths corresponding to the eigenvalues scaled by a factor of $r$. Since the eigenvalues of $A^{1/2}(x_0)$ are bounded between $\lambda^{1/2}$ and $\Lambda^{1/2}$, we have
	\begin{align}\label{comparability to Wulff shapes}
		\B(x_0,\lambda^{1/2} r)\subset \W_{x_0}(x_0,r)\subset \B(x_0,\Lambda^{1/2}r).
	\end{align}
	We now prove the invariance of almost-minimizers under the change of variable $T_{x_0}$ and refer readers to the change of variable formula given in Proposition \ref{change of variable} in Appendix \ref{appendix A}.

	\begin{proposition}[Invariance of almost-minimizers under the change of variable $T_{x_0}$]\label{invariance}
			If $E$ is a $(\kappa,\alpha)$-almost-minimizer of $\sF_A$ in $U$ at scale $r_0$, then $E_{x_0}$ is a $(\Lambda^{(\alpha+n-1)/2}\lambda^{-n/2}\kappa,\alpha)$-almost-minimizer of $\sF_{A_{x_0}}$ in $U_{x_0}$ at scale $r_0/\Lambda^{1/2}$.
	\end{proposition}
	
		\begin{proof}
			 Suppose $E_{x_0}\Delta F_{x_0}\subsetcc \B(z,r)\cap U_{x_0}$ for some $z\in U_{x_0}$ and $r<r_0/\Lambda^{1/2}$ (here we write  $F_{x_0}$ as an arbitrary competitor for $E_{x_0}$ whose image $F$ under $T_{x_0}^{-1}$ will be a competitor for $E$). Applying Proposition \ref{change of variable} with $y=T_{x_0}(x)$, noting that $Jf=JT_{x_0}=\det A^{-1/2}(x_0)$ and $(\nabla g\circ f)^t=A^{1/2}(x_0)$ since $A^{1/2}(x_0)$ is symmetric,
			   	 we have
			\begin{align}
			 	\sF_{A_{x_0}}(E_{x_0};\B(z,r))&=\int_{\B(z,r)\cap \rb E_{x_0}}\la A_{x_0}(y)\nu_{E_{x_0}},\nu_{E_{x_0}}\ra^{1/2}\dH(y)\nonumber\\
			 	&=\int_{T_{x_0}^{-1}(\B(z,r))\cap\rb E}\la A_{x_0}(T_{x_0}(x))A^{1/2}(x_0)\nu_E,A^{1/2}(x_0)\nu_E\ra^{1/2} \det A^{-1/2}(x_0)\dH(x)\nonumber\\
			 	&=\int_{T_{x_0}^{-1}(\B(z,r))\cap\rb E}\la A^{1/2}(x_0) A_{x_0}(T_{x_0}(x))A^{1/2}(x_0)\nu_E,\nu_E\ra^{1/2} \det A^{-1/2}(x_0)\dH(x).
			\end{align}
			Note that $A_{x_0}(T_{x_0}(x))=A^{-1/2}(x_0)A(x)A^{-1/2}(x_0)$ and so $A^{1/2}(x_0) A_{x_0}(T_{x_0}(x))A^{1/2}(x_0)=A(x)$. Hence 
			\begin{align}
				\sF_{A_{x_0}}(E_{x_0};\B(z,r))=\det A^{-1/2}(x_0)\sF_{A}(E;T_{x_0}^{-1}(\B(z,r))) .
			\end{align}
			Likewise, $
				\sF_{A_{x_0}}(F_{x_0};\B(z,r))=\det A^{-1/2}(x_0)\sF_{A}(F;T_{x_0}^{-1}(\B(z,r)))$.
 			Note that 
			\begin{align} 
				E\Delta F\subsetcc T_{x_0}^{-1}(\B(z,r))\cap U\subset \B(T_{x_0}^{-1}(z), \Lambda^{1/2}r)\cap U,\ T_{x_0}^{-1}(z)\in U,\text{ and }\Lambda^{1/2}r<r_0.
			\end{align}
				 Thus $\sF_A(E;\B(T_{x_0}^{-1}(z), \Lambda^{1/2}r))\leq \sF_A(E;\B(T_{x_0}^{-1}(z), \Lambda^{1/2}r))+\Lambda^{(\alpha+n-1)/2}\kappa\: r^{\alpha+n-1}$ by the minimality condition. This simplifies to  $\sF_A(E;T_{x_0}^{-1}(\B(z,r))\leq \sF_A(E;T_{x_0}^{-1}(\B(z,r))+\Lambda^{(\alpha+n-1)/2}\kappa \: r^{\alpha+n-1}$. It then follows that
			\begin{align}
				\sF_{A_{x_0}}(E_{x_0};\B(z,r))&=\det A^{-1/2}(x_0)\sF_{A}(E;T_{x_0}^{-1}(\B(z,r)))\nonumber\\
					&\leq \det A^{-1/2}(x_0)\sF_{A}(F;T_{x_0}^{-1}(\B(z,r)))+\det A^{-1/2}(x_0)\Lambda^{(\alpha+n-1)/2}\kappa \: r^{\alpha+n-1}\nonumber\\
					&\leq \sF_{A_{x_0}}(F_{x_0};\B(z,r))+\Lambda^{(\alpha+n-1)/2}\lambda^{-n/2}\kappa \: r^{\alpha+n-1}
			\end{align}
			as desired.
		\end{proof}
	
	Hence any of the properties or estimates we prove for $(\kappa,\alpha)$-almost-minimizers also hold for the set $E_{x_0}$ (with any bounds or estimates having modified constants but which depend only on the same universal constants). Working with $E_{x_0}$ will allow us to assume $A(x_0)=I$ in the proof of many of our estimates and will in turn allow us to prove additional properties and estimates for general $(\kappa,\alpha)$-almost-minimizers.
	
    As previously mentioned, whenever we write $C$ we mean a constant (which may change from line to line) that depends only on the universal constants $n,\lambda,\Lambda,\kappa,\alpha,r_0$, and upper bounds for $||A||_{C^\alpha}$, but does not depend on the set $E$ or the point $x_0$. In cases where we wish to emphasize that a constant depends on fewer constants such as, for example, on the dimension $n$ only, we write $C(n)$.
	
	\subsection{Scaling of the energy $\sF_A$}\
	
	In Section \ref{reverse Poincare inequality section} we will use the scaling of the energy $\sF_A$ to simplify and work at scale $1$ instead of of scale $r$ and in Section \ref{analysis of singular set section} we will utilize blow-up analysis to study the singular set almost-minimizers. The \textbf{blow-ups} $E_{x_0,r}$ of a set $E$ at a point $x_0\in \R^n$ and scale $r>0$ are defined by
	\begin{align}
		E_{x_0,r}=\frac{E-x_0}{r}=\Phi_{x_0,r}(E)
	\end{align}
	 where $\Phi_{x_0,r}:\R^n\to \R^n$ is the map defined by 
	\begin{align}
		\Phi_{x_0,r}(x)=\frac{x-x_0}{r}.
	\end{align}
	We denote the inverse of $\Phi_{x_0,r}$ by $\Psi_{x_0,r}$, that is, $\Psi_{x_0,r}(y)=ry+x_0$. Given a matrix-valued function $A=(a_{ij}(x))_{i,j=1}^n$, we denote by $A_{x_0,r}$ the matrix-valued function
	\begin{align}
		A_{x_0,r}(y)=A(ry+x_0)=A\circ\Psi_{x_0,r}(y)
	\end{align}
	(this is not to be confused with $A_{x_0}$ from the previous subsection). Note that $||A_{x_0,r}||_{C^\alpha}=r^\alpha||A||_{C^\alpha}$.
	\begin{proposition}[Scaling of $\sF_A$]\label{scaling of the anisotropic energy} If $E$ is a set of locally finite perimeter in $\R^n$, $x_0\in\R^n$, $r>0$, then
	\begin{align}
		\sF_{A_{x_0,r}}(E_{x,r};F_{{x_0},r})= \frac{\sF_A(E;F)}{r^{n-1}}.
	\end{align}
	for Borel sets $F$. In particular, if $E$ is a $(\kappa,\alpha)$-almost-minimizer of $\sF_A$ in $U$ at scale $r_0$, then $E_{{x_0},r}$ is a $(\kappa r^\alpha, \alpha)$-almost-minimizer of $\sF_{A_{{x_0},r}}$ in $U_{{x_0},r}$ at scale $r_0/r$.			
	\end{proposition}

	\begin{proof}
		We apply Proposition \ref{change of variable} with the change of variable $y=f(x)=\Phi_{{x_0},r}(x)=(x-{x_0})/r$ and integrand $(x,\xi)\mapsto\la A_{{x_0},r}(x)\xi,\xi \ra^{1/2}$. Then $g(y)=\Psi_{{x_0},r}(y)=ry+{x_0}$, $\nabla g=r I$, and $Jf=r^{-n}$. So
			$|(\nabla g\circ f)^t\nu_E|=r$ and it follows that
		\begin{align}\label{scaling change}
			\sF_{A_{{x_0},r}}(E_{{x_0},r};F_{{x_0},r})&=\int_{F_{{x_0},r}\cap\rb E_{{x_0},r}} \la A(ry+{x_0})\nu_{E_{{x_0},r}},\nu_{E_{{x_0},r}}\ra^{1/2}\dH(y)\nonumber\\
			&=\int_{F\cap\rb E} \la A(x)\nu_{E},\nu_{E}\ra^{1/2}\ r^{-n}r\dH(x)\nonumber\\
			&=\frac{\sF_A(E;F)}{r^{n-1}}
		\end{align}
		Now, let $F$ be a set of locally finite perimeter in $\R^n$ with $E_{{x_0},r}\Delta F_{{x_0},r}\subsetcc \B(x, s)\cap U_{{x_0},r}$ for $x\in U_{{x_0},r}$ and $s<r_0/r$. Then $E\Delta F\subsetcc \Psi_{{x_0},r}(\B(x,s))\cap U$. Note $\Psi_{{x_0},r}(\B(x,s))=\B(rx+{x_0},rs)$ with $rx+{x_0}\in U$ and $rs<r_0$.
		Applying \eqref{scaling change} to $\B(x,s)$ and using the almost-minimality of $E$ in $U$ at scale $r_0$, we have
		\begin{align}
			\sF_{A_{{x_0},r}}(E_{{x_0},r}; \B(x,s)) &=\frac{\sF_A(E;\B(rx+{x_0},rs))}{r^{n-1}}\nonumber\\
			&\leq \frac{\sF_A(F;\B(rx+{x_0},rs))+\kappa(rs)^{\alpha+n-1}}{r^{n-1}}\nonumber\\
			&= \sF_{A_{{x_0},r}}(F_{{x_0},r}; \B(x,s))+\kappa r^\alpha s^{\alpha+n-1},
		\end{align}
		that is, $E_{{x_0},r}$ is an $(\kappa r^\alpha, \alpha)$-almost-minimizer of $\sF_{A_{{x_0},r}}$ in $U_{{x_0},r}$ at scale $r_0/r>0$.	
	\end{proof}
	
	\subsection{Comparison sets}\
	
	To utilize the almost-minimality condition we will often construct competitors by modifying $E$ inside an open set. The following proposition allows us to do this.
		
	\begin{proposition}[Comparison sets by replacements]\label{comparison sets} 
		If $E$ and $F$ are sets of locally finite perimeter in $\R^n$ and $G$ is an open set of finite perimeter in $\R^n$ such that
		\begin{align}
			\Hm(\rb G\cap \rb E)=\Hm(\rb G\cap \rb F)=0, 
		\end{align}
		then the set defined by
		\begin{align}
			F_0=\big(F\cap G\big)\cup \big(E\setminus G\big) 
		\end{align}
			is a set of locally finite perimeter in $\R^n$. Moreover, if $G\subsetcc U$ and $U$ is open, then	
		\begin{align}\label{decomposition}
			\sF_A(F_0; U)=\sF_A(F;G)+\sF_A(E; U\setminus\overline G)+\sF_A(G; E^{(1)}\Delta F^{(1)}).
		\end{align}
	\end{proposition}
	
	\begin{proof}
		In the proof of \cite[Theorem 16.16]{maggi} the decomposition, see (16.35),
		\begin{align}
			\mu_{F_0}=\mu_F\rstr G+\mu_G(F^{(1)}\cap E^{(0)})+\mu_E\rstr(\R^n\setminus\overline G)-\mu_G\rstr (E^{(1)}\cap F^{(0)}) 
		\end{align}
		is proved. Since all of the measures on the right-hand side are concentrated on disjoint sets and since the measures $\sF_A(G^c;\:\cdot\:)$ and $\sF_A(G;\:\cdot\:)$ are equal and $\mu_{G^c}=-\mu_G$, we have 
		\begin{align}\label{initial decomposition}
			\sF_A(F_0; U)=\sF_A(F; G)+\sF_A(E; U\setminus\overline G)+\sF_A(G; (F^{(1)}\cap E^{(0)})\cup (E^{(1)}\cap F^{(0)})) 
		\end{align}
		by additivity of $\sF_A$. By $\Hm(\rb G\cap \rb E)=0$ and $\Hm(\R^n\setminus (E^{(0)}\cup E^{(1)}\cup \rb E))=0$,
		\begin{align}
		    \sF_A(G; F^{(1)}\cap E^{(0)})=\sF_A(G; F^{(1)}\setminus E^{(1)}).
		\end{align}
    	Likewise, $\Hm(\rb G\cap \rb F)=0$ and $\Hm(\R^n\setminus (F^{(0)}\cup F^{(1)}\cup \rb F))=0$ and so
		\begin{align}
		    \sF_A(G; E^{(1)}\cap F^{(0)})=\sF_A(G; E^{(1)}\setminus F^{(1)}).
		\end{align}
		These along with \eqref{initial decomposition} prove \eqref{decomposition}.
	\end{proof}

\subsection{Volume and perimeter bounds and the almost-monotonicity formula}\

	One important property which almost-minimizers of $\sF_A$ possess is bounds on both the volume and the perimeter of $E$ on balls centered at points in their topological boundary. Recall that we require $\spt \mu_E=\p E$ for almost-minimizers. The full set of estimates is given in Proposition \ref{volume and perimeter bounds} but we have some work to do to prove this. The first step is showing the upper bound on perimeter.
	
	Define the \textbf{perimeter density ratio} of $E$ at $x_0$ by
	\begin{align}
		\theta(E, x_0, r)=\frac{P(E; \B(x_0,r))}{r^{n-1}}.
	\end{align} and
	 \textbf{perimeter density} of $E$ at $x_0$ by
	\begin{align}
		\theta(E,x_0)=\lim_{r\to 0^+}	\theta(E, x_0, r)
	\end{align}
	whenever the limit exists.

	\begin{lemma}[Upper perimeter bound]\label{upper perimeter bound lemma} There exists a positive constant $C=C(n,\lambda,\Lambda,\kappa,\alpha,r_0)$ with the following property. If $E$ is a $(\kappa,\alpha)$-almost-minimizer of $\sF_A$ in $U$ at scale $r_0$, then for every $x_0\in U\cap\p E$ with $r<d=\min\{\dist(x_0,\p U), r_0\}<\infty$,
	\begin{align}\label{upper perimeter bound}
		\frac{P(E; \B(x_0,r))}{r^{n-1}}\leq C.
	\end{align}
	\end{lemma}
	
	\begin{proof}
		Consider the function $m\colon(0,d)\to\R$ defined by $m(r)=|E\cap \B(x_0,r)|$. Note that $m$ is increasing, $m'(r)=\Hm(E^{(1)}\cap\p \B(x_0,r))$ for a.e. $r$ by the coarea formula, and $\Hm(\rb E\cap \p \B(x_0,r))=0$ for a.e. $r$ because $\Hm\rstr\: \rb E$ is a Radon measure. Let $r\in (0,d)$ be one of the a.e. radii that satisfies both $m'(r)=\Hm(E^{(1)}\cap\p \B(x_0,r))$ and $\Hm(\rb E\cap \p \B(x_0,r))=0$. For $s\in (r,d)$ consider the comparison set $F=E\setminus\B(x_0,r)$ in $\B(x_0,s)$. Then $E\Delta F\subset \B(x_0,r)\subsetcc \B(x_0,s)$. It follows from comparability to perimeter and the almost-minimality that
		\begin{align}
			\lambda^{1/2} P(E; \B(x_0,s)) & \leq \sF_A(E; \B(x_0,s))\nonumber\\
			&\leq \sF_A(E\setminus \B(x_0,r); \B(x_0,s))+\kappa s^{\alpha+n-1}\nonumber\\
			&\leq  \Lambda^{1/2} P(E\setminus \B(x_0,r); \B(x_0,s))+\kappa s^{\alpha +n-1} 
		\end{align}
		and so
		\begin{align}
			P(E; \B(x_0,s) )& \leq C\big( P(E\setminus \B(x_0,r); \B(x_0,s))+s^{\alpha +n-1}\big)\nonumber\\
			&= C\big(\Hm(E^{(1)}\cap \p \B(x_0,r))+P(E;\B(x_0,s)\setminus\cl{\B}(x_0,r))+s^{\alpha +n-1}\big) 
		\end{align}
		since $\Hm(\rb E\cap \p \B(x_0,r))=0$. Sending $s\to r^+$ yields the inequality
		\begin{align}\label{perimeter density inequality}
			P(E; \B(x_0,r)) \leq C\big(\Hm(E^{(1)}\cap \p \B(x_0,r))+r^{\alpha +n-1}\big). 
		\end{align}
		This, together with $\Hm(E^{(1)}\cap \p \B(x_0,r))\leq n\omega_{n}r^{n-1}$ and $r<r_0$, gives $
			P(E; \B(x_0,r)) \leq Cr^{n-1}$.
	\end{proof}
		
	To obtain the lower perimeter bound for almost-minimizers of $\sF_A$, we shall adapt an argument given by Tamanini for almost-minimizers of perimeter in \cite{tamanini1982, tamanini1984regularity} which makes use of an almost-monotonicity formula. Monotonicity formulas are often times a valuable tool in regularity theory. For example, the monotonicity of density ratios for minimizers of surface area is heavily relied upon in \cite{allard, taylor} as well as in many other papers. By this we mean the fact that if $E$ is a perimeter minimizer in $U$, $x_0\in U$, then the density ratio
		\begin{align}
			\theta(E; x_0, r)=\frac{P(E; \B(x_0,r))}{r^{n-1}}.
		\end{align}
		is monotonically increasing in $r$ (see, for example, \cite[Theorem 17.16]{maggi}). In \cite{allardcharacterization}, Allard demonstrated for integrands depending solely on the direction variable $\nu_E$ (and not on the spatial variable $x$) that monotonicity formulas exist if and only if the integrand is a linear change of variable from the area integrand. Under the change of variable $T_{x_0}$ we have $A(x_0)=I$ so that our sets satisfy the condition for almost-minimality of perimeter when making comparisons on balls centered at $x_0$ as shown in Lemma \ref{perturbation lemma} below. A key observation is that we only need these comparisons to apply the standard cone-competitor argument to obtain an almost-monotonicity formula as we do in Lemma \ref{cone competitor argument}.		
		\begin{lemma}\label{perturbation lemma}
			There exists a positive constant $C=C(n,\lambda,\Lambda, \kappa,\alpha, r_0)$ with the following property. If $E$ is a $(\kappa,\alpha)$-almost-minimizer of $\sF_A$ in $U$ at scale $r_0>0$, $x_0\in U\cap \p E$, and $A(x_0)=I$, then
			\begin{align}\label{comparison property}
				P(E; \B(x_0,r))\leq P(F; \B(x_0,r))+C(\kappa+||A||_{C^\alpha}) r^{\alpha+n-1}
			\end{align}
			whenever $E\Delta F\subsetcc \B(x_0,r)$ and $r<d=\min\{r_0, \dist(x_0,\p U)\}$.
			\end{lemma}

		\begin{proof}
			Let $E\Delta F\subsetcc \B(x_0,r)
			\subset U$ and $r<r_0$. If $P(E;\B(x_0,r))\leq P(F; \B(x_0,r))$, then \eqref{comparison property} trivially holds true. So consider the case when $P(F; \B(x_0,r))\leq P(E;\B(x_0,r))$.
		
		By inequality \eqref{integrand ineq1} and $A(x_0)=I$, we have 
		\begin{align}
		   |\nu_E|\leq \la A(x)\nu_E,\nu_E\ra^{1/2}+ \frac{1}{2\lambda}||A(x_0)-A(x)||\leq  \la A(x)\nu_E,\nu_E\ra^{1/2}+ \frac{1}{2\lambda}||A||_{C^\alpha}|x-x_0|^\alpha
		\end{align}
		   and so $|\nu_E|\leq \la A(x)\nu_E,\nu_E\ra^{1/2}+ (1/2\lambda)||A||_{C^\alpha}r^\alpha$ for $x\in\B(x_0,r)$. Integrating with respect to $\Hm\rstr\: \rb E$ gives
		\begin{align}\label{mono ineq1}
			P(E;\B(x_0,r))\leq \sF_A(E;\B(x_0,r))+\frac{1}{2\lambda}||A||_{C^\alpha}r^\alpha P(E;\B(x_0,r))
		\end{align}
		Similarly, $\la A(x)\nu_F,\nu_F\ra^{1/2}\leq |\nu_F|+(1/2\lambda)||A||_{C^\alpha}r^\alpha$ for $x\in\B(x_0,r)$ and so
		\begin{align}\label{mono ineq2}
			\sF_A(F;\B(x_0,r))\leq P(F;\B(x_0,r))+\frac{1}{2\lambda}||A||_{C^\alpha} r^\alpha P(F;\B(x_0,r))
		\end{align}
		Combining the almost-minimizer inequality with \eqref{mono ineq1}, \eqref{mono ineq2}, and $P(F; \B(x_0,r))\leq P(E;\B(x_0,r))$ gives $P(E;\B(x_0,r))\leq  P(F;\B(x_0,r))+\kappa r^{\alpha+n-1}+(1/2\lambda)||A||_{C^\alpha} r^\alpha P(E;\B(x_0,r))$. The upper perimeter bound $P(E;\B(x_0,r))\leq C r^{n-1}$ gives $P(E;\B(x_0,r))\leq  P(F;\B(x_0,r))+C(\kappa+||A||_{C^\alpha}) r^{\alpha+n-1}$.
		\end{proof}

		\begin{lemma}\label{cone competitor argument}
			There exists a positive constant $C=C(n,\lambda,\Lambda,\kappa, \alpha, r_0, ||A||_{C^\alpha})$ such that the following holds. If $E$ is a $(\kappa,\alpha)$-almost-minimizer of $\sF_A$ in $U$ at scale $r_0>0$ with $x_0\in U\cap \p E$ and $A(x_0)=I$, then the function
			\begin{align}
				r\mapsto\frac{P(E;\B(x_0,r))}{r^{n-1}}+C r^{\alpha}
			\end{align}
			is monotonically increasing on $(0,d)$ where $d=\min\{r_0,\dist(x_0,\p U)\}>0$.
		\end{lemma}
		
		\begin{proof}			
			Without loss of generality assume $x_0=0$ and write $\B_r=\B(x_0,r)$. Define the function $\Phi\colon (0,d)\to (0,\infty)$ by $\Phi(r)=P(E;\B_r)$. $\Phi$ is increasing and hence differentiable for a.e. $r\in(0,d)$. Thus it suffices to prove
			\begin{align}
				\frac{d}{dr}\Big(\frac{\Phi(r)}{r^{n-1}}+C r^{\alpha}\Big)\geq 0\qquad\text{ for a.e. } r\in(0,d), 
			\end{align}
			which can be rewritten as
			\begin{align}\label{differential inequality}
				\Phi(r)\leq \frac{r}{n-1}\Phi'(r)+C r^{\alpha+n-1}\qquad\text{ for a.e. } r\in(0,d).
			\end{align}
			
			The idea of the proof of \eqref{differential inequality} is to construct cone competitors over $E\cap \p\B_r$ with vertex at $0$ for each $r>0$ to use in the comparison inequality \eqref{comparison property}. To do this we will need to approximate $E$ by open sets with smooth boundary and construct the cone competitors for the approximating sets.
			
			By \cite[Theorem 13.8]{maggi}, there is a sequence $\{E_h\}_{h\in\NN}$ of open sets with smooth boundary in $\R^n$ such that $E_h\loc E$ and $|\mu_{E_h}|\wkly |\mu_E|$. For now hold $h\in\NN$ fixed. The set $E_h\cap \p\B_r$ is relatively open in $\p\B_r$ for every $r>0$. By Sard's lemma, 
			\begin{align}\label{smooth n-1}
				\p E_h\cap \p\B_r\text{ is a smooth $(n-2)$-dimensional surface for a.e. } r>0.
			\end{align}
			Consider the cones with vertex $0$ over $E_h\cap \p\B_r$,
		    \begin{align}
				K_h(r)=\big\{\lambda x\in\R^n: \lambda>0,\: x\in E_h\cap\p\B_r\big\}.
		    \end{align}
		    For the a.e. $r>0$ such that \eqref{smooth n-1} holds we have that $K_h(r)$ is a set of locally finite perimeter in $\R^n$ with
			\begin{align}\label{cone properties}
			 \mu_{K_h(r)}=\nu_{K_h(r)}\Hm\rstr\: \p K_h(r),\qquad\text{and}\qquad \nu_{K_h(r)}(x)\cdot x=0,\qquad \forall x\in K_h(r)\setminus\{0\}.
			\end{align}
			For $r>0$ such that \eqref{smooth n-1} holds, the coarea formula for $(n-1)$-dimensional rectifiable sets (see \cite[Theorem 18.8]{maggi}) on $\p K_h(r)$ with $u(x)=|x|$ yields
			\begin{align}
				P(K_h(r); \B_r)=\int_0^r \sH^{n-2}(\p K_h(r)\cap \p\B_t)\: dt 
			\end{align}
			since $|\nabla^{\p K_h(r)} u|=|\nabla u|=1$. Note that for $t<r$ we have
			\begin{align}
				\p K_h(r)\cap \p\B_t=\Big(\frac{t}{r}\Big)\big(\p K_h(r)\cap \p\B_r\big)=\Big(\frac{t}{r}\Big)\big(\p E_h\cap \p\B_r\big) 
			\end{align}
			and hence for $r$ such that \eqref{smooth n-1} holds we have
			\begin{align}\label{cone perimeter}
				P(K_h(r); \B_r)=\int_0^r \bigg(\frac{t}{r}\bigg)^{n-2}\sH^{n-2}(\p E_h\cap \p\B_r)\: dt=\frac{r}{n-1}\sH^{n-2}(\p E_h\cap \p\B_r)
			\end{align}
			
			Consider a radius $r>0$ such that for all $h\in\NN$ both \eqref{smooth n-1} and $\Hm(\rb E\cap \p\B_r)=\Hm(\p E_h\cap \p\B_r)=0$ hold (and consequently \eqref{cone perimeter} as well). This true for a.e. $r\in (0,d)$ since $\Hm \rstr\: \rb E$ and $\Hm\rstr\: \p E_h$ are Radon measures and by Sard's lemma. Consider the comparison sets $F_h=(K_h(r)\cap \B_r)\cup (E\setminus \B_r)$. Let $s$ be such that $r<s<d$. By (16.32) of \cite{maggi}  we have
			\begin{align}
				P(F_h; \B_s)=P(K_h(r); \B_r)+P(E; \B_s\setminus\bar \B_r)+\Hm\big(\big(E^{(1)}\Delta K_h(r))\cap \p\B_r\big).
			\end{align}
			Since $E\Delta F_h\subset \B_r\subsetcc \B_s$, applying Lemma \ref{perturbation lemma} gives
			\begin{align}
				P(E; \B_s)\leq P(K_h(r); \B_r)+P(E; \B_s\setminus\bar\B_r)+\Hm\big(\big(E^{(1)}\Delta K_h(r))\cap \p\B_r\big)+Cs^{\alpha+n-1}
			\end{align}
			which by subtracting $P(E; \B_s\setminus\bar \B_r)$ from each side together with \eqref{cone perimeter} simplifies to 
			\begin{align}
				P(E; \B_r)\leq \frac{r}{n-1}\sH^{n-2}(\p E_h\cap \p\B_r)+\Hm\big(\big(E^{(1)}\Delta E_h)\cap \p\B_r\big)+Cs^{\alpha+n-1}. 
			\end{align}
			Sending $s\to r^+$ gives
			\begin{align}
				P(E; \B_r)\leq \frac{r}{n-1}\sH^{n-2}(\p E_h\cap \p\B_r)+\Hm\big(\big(E^{(1)}\Delta E_h)\cap \p\B_r\big)+C r^{\alpha+n-1}. 
			\end{align}			
			This inequality holds for a.e. $r\in (0,d)$ and integrating over the interval $(s,t)\subset (0,d)$ yields
			\begin{align}\label{integral inequality}
				\int_s^t P(E; \B_r)\: dr\leq \frac{1}{n-1}\int_s^t r\:\sH^{n-2}(\p E_h\cap \p\B_r)\: dr+\Hm\big(\big(E^{(1)}\Delta E_h)\cap  \B_d\big)+C (t^{\alpha+n}-s^{\alpha+n}).
			\end{align}
			Applying the coarea formula for $(n-1)$-dimensional rectifiable sets (\cite[Theorem 18.8]{maggi}) on $\p E_h$ with $u(x)=|x|$, and $g=|x|$, gives
			\begin{align}\label{coarea rectifiable}
				\int_s^t r\:\sH^{n-2}(\p E_h\cap \p\B_r)\: dr= \int_{\p E_h\cap (\B_t\setminus \bar
				\B_s)} |x||\nabla^{\p E_h}u|\leq t\: P(E_h; \bar\B_t\setminus \B_s)
			\end{align}
			since $|\nabla^{\p E_h}u|\leq |\nabla u|=1$. Thus combining \eqref{integral inequality} and \eqref{coarea rectifiable} gives
			\begin{align}
				\int_s^t P(E; \B_r)\: dr\leq \frac{t}{n-1}P(E_h; \bar \B_t\setminus \B_s)+\Hm\big(\big(E^{(1)}\Delta E_h)\cap  \B_d\big)+C (t^{\alpha+n}-s^{\alpha+n}). 
			\end{align}
			By $E_h\loc E$ and $|\mu_{E_h}|\wkly|\mu_E|$, sending $h\to\infty$ gives
			\begin{align}
				\int_s^t P(E; \B_r)\: dr\leq \frac{t}{n-1}P(E; \bar \B_t\setminus \B_s)+C (t^{\alpha+n}-s^{\alpha+n}).
			\end{align}
			Dividing by $t-s$ and sending $t\to s^+$ at points of differentiability of $\Phi$ yields \eqref{differential inequality} as desired.
		\end{proof}
		
		Now we are able to use a perturbation argument and the change of variable $T_{x_0}$ to obtain an almost-monotonicity formula when $A(x_0)$ is not assumed to equal $I$.
		
		\begin{theorem}[Almost-monotonicity formula]\label{almost monotonicity formula}
			There exists a positive constant\\ $C=C(n,\lambda,\Lambda,\kappa,\alpha, r_0, ||A||_{C^\alpha})$ with the following property. If $E$ is a $(\kappa,\alpha)$-almost-minimizer of $\sF_A$ in $U$ at scale $r_0$, then for every $x_0\in U\cap \p E$, we have
			\begin{align}
				\frac{\sF_A(E; \W_{x_0}(x_0,s))}{s^{n-1}}\leq \frac{\sF_A(E; \W_{x_0}(x_0,r))}{r^{n-1}}+C r^\alpha	
			\end{align}
			whenever $0<s\leq r<d$ where $d=\Lambda^{-1/2}\min\{r_0, \dist(x_0,\p U)\}$.
		\end{theorem}
		
		\begin{proof}
			Applying Lemma \ref{cone competitor argument} to $E_{x_0}$ and $\sF_{A_{x_0}}$ gives that $
			  r\mapsto r^{-(n-1)}P(E_{x_0};\B(x_0,r))+C  r^\alpha$
			is monotone increasing on $(0,d_{x_0})$ where $d_{x_0}=\min\{\Lambda^{-1/2}r_0,\dist(x_0,\p U_{x_0})\}$. The change of variable $y=T_{x_0}(x)$ applied to $E$ gives $P(E_{x_0};\B(x_0,r))=\det A^{-1/2}(x_0)\sF_{A(x_0)}(E;\W_{x_0}(x_0,r))$. This and the bound $(\det A^{-1/2}(x_0))^{-1}\leq \Lambda^{n/2}$ imply that  
			\begin{align}\label{monotonicity ineq}
			 		r\mapsto \frac{\sF_{A(x_0)}(E;\W_{x_0}(x_0,r))}{r^{n-1}}+Cr^\alpha	
			\end{align}	
    		is monotone increasing on $(0,d_{x_0})$. Note $U=T_{x_0}^{-1}(U_{x_0})$ and $\Lip T_{x_0}^{-1}= ||A^{1/2}(x_0)||\leq\Lambda^{1/2}$. Given $x\in\p U$, setting $y=T_{x_0}(x)\in \p U_{x_0}$, it follows that
    		\begin{align}
    		    |x-x_0|=|T_{x_0}^{-1}(y)-T_{x_0}^{-1}|\leq \Lip T_{x_0}^{-1}|y-x_0|\leq \Lambda^{1/2}\dist(x_0,\p U_{x_0}).
    		\end{align}
    		Hence $\dist(x_0,\p U)\leq \Lambda^{1/2}\dist(x_0,\p U_{x_0})$ and so $d=\Lambda^{-1/2}\min\{r_0,\dist(x_0,\p U\}\leq d_{x_0}$. Thus \eqref{monotonicity ineq} is monotone increasing on $(0,d)$. By \eqref{integrand ineq1} we have
    	    \begin{align}
    	        \la A(x)\nu_E,\nu_E\ra^{1/2}&\leq \la A(x_0)\nu_E,\nu_E\ra^{1/2}+ \la (A(x)-A(x_0))\nu_E,\nu_E\ra^{1/2}\nonumber\\
    	        &\leq \la A(x_0)\nu_E,\nu_E\ra^{1/2}+\frac{1}{2\lambda}||A||_{C^\alpha}|x-x_0|^\alpha
    	    \end{align} 
    	    and so $\la A(x)\nu_E,\nu_E\ra^{1/2}\leq \la A(x_0)\nu_E,\nu_E\ra^{1/2}+ C||A||_{C^\alpha}s^\alpha$ for $x\in\W_{x_0}(x_0,s)$ by \eqref{comparability to Wulff shapes}. It follows that
			\begin{align}\label{monotonicity ineq2}
				\frac{\sF_A(E; \W_{x_0}(x_0,s))}{s^{n-1}} &\leq \frac{\sF_{A(x_0)}(E; \W_{x_0}(x_0,s))}{s^{n-1}}+Cs^\alpha \frac{P(E; \W_{x_0}(x_0,s))}{s^{n-1}}\nonumber\\
				 &\leq \frac{\sF_{A(x_0)}(E; \W_{x_0}(x_0,s))}{s^{n-1}}+Cs^{\alpha} 
			\end{align}
			where we used that $P(E; \W_{x_0}(x_0,s))\leq P(E; \B(x_0, \Lambda^{1/2}s))\leq Cs^{n-1}$ by the upper perimeter bound \eqref{upper perimeter bound}.
			Similarly, we have
			\begin{align}			
				\frac{\sF_{A(x_0)}(E; \W_{x_0}(x_0,r))}{r^{n-1}}\leq \frac{\sF_{A}(E; \W_{x_0}(x_0,r))}{r^{n-1}}+Cr^{\alpha} 
			\end{align}
			Combining this last inequality and \eqref{monotonicity ineq2} with \eqref{monotonicity ineq} and $s\leq r$ yields
			\begin{align}			
				\frac{\sF_A(E; \W_{x_0}(x_0,s))}{s^{n-1}}\leq\frac{\sF_{A}(E; \W_{x_0}(x_0,r))}{r^{n-1}}+C r^\alpha
			\end{align}	
			as desired.		
		\end{proof}
		
		For $x_0\in\p E$, define the \textbf{$\sF_A$-density ratio} of $E$ at $x_0$ by
			\begin{align}
				\theta_A(E, x_0, r)=\frac{\sF_A(E; \W_{x_0}(x_0,r))}{r^{n-1}}.
			\end{align}
		and the \textbf{$\sF_A$-density} of $E$ at $x_0$ by
		\begin{align}
				\theta_A(E, x_0)=\lim_{r\to 0^+} \theta_A(E, x_0, r)
		\end{align}
		when the limit exists.

		\begin{corollary}[Existence of densities]\label{existence of densities}	
		    If $E$ is a $(\kappa,\alpha)$-almost-minimizer of $\sF_A$ in $U$ at scale $r_0$, then for every $x_0\in U\cap \p E$	the density		\begin{align}
				\theta_A(E, x_0)=\lim_{r\to 0^+} \theta_A(E, x_0, r)
			\end{align}
			exists.
	   	\end{corollary}
			
		\begin{proof}
			For every $0<s\leq r<d$ we have by almost-monotonicity that $\theta_A(E, x_0, s)\leq \theta_A(E, x_0, r)+C r^\alpha$.
			Taking the $\limsup$ as $s\to 0^+$ followed by the $\liminf$ as $r\to 0^+$ yields			
			\begin{align}
			\limsup_{s\to 0^+}\theta_A(E,x_0,s)\leq \liminf_{r\to 0^+} \theta_A(E, x_0, r)+ \limsup_{r\to 0^+} Cr^\alpha=\liminf_{r\to 0^+} \theta_A(E, x_0, r) 
			\end{align}
			Hence $\theta_A(E, x_0)=\lim_{r\to 0^+} \theta_A(E, x_0, r)$
			exists.
		\end{proof}
		
		Using the almost-monotonicity formula, we are now able to control the perimeter density ratios from below.
		\begin{proposition}\label{lower bound on density ratio}
		There exists a positive constant $C=C(n,\lambda,\Lambda,\kappa,\alpha, r_0, ||A||_{C^\alpha})$ with the following property. If $E$ is a $(\kappa,\alpha)$-almost-minimizer of $\sF_A$ in $U$ at scale $r_0$, then for every $x_0\in U\cap \p E$, we have
			\begin{align}\label{almost lower perimeter bound}
		 		\omega_{n-1}(\lambda/\Lambda)^{n/2}-C r^\alpha \leq \frac{P(E;\B(x_0,r))}{r^{n-1}}
		 	\end{align}
		for $r<\min\{r_0, \dist(x_0,\p U)\}$.
		\end{proposition}

		\begin{proof} Let $r<\min\{r_0, \dist(x_0,\p U)\}$. First consider the case $x_0\in\rb E$. The limit of perimeter density rations at a point in the reduced boundary converge to $\omega_{n-1}$ as $r\to 0^+$ \cite[Corollary 15.8]{maggi}. Note that $\Lambda^{-1/2}r<\Lambda^{-1/2}\min\{r_0, \dist(x_0,\p U)\}$ and so for $s<r$ we can apply Theorem \ref{almost monotonicity formula} with $\Lambda^{-1/2}s\leq \Lambda^{-1/2}r$
		 to obtain
		 \begin{align}				
		 \frac{\sF_A(E; \W_{x_0}(x_0,\Lambda^{-1/2}s))}{(\Lambda^{-1/2}s)^{n-1}}\leq\frac{\sF_{A}(E; \W_{x_0}(x_0,\Lambda^{-1/2}r))}{(\Lambda^{-1/2}r)^{n-1}}+C r^\alpha	 \end{align}
		By comparability to perimeter and \eqref{comparability to Wulff shapes}, we have
			\begin{align}
				\omega_{n-1}&=\lim_{s\to 0^+}\frac{P(E;\B(x_0,(\lambda/\Lambda)^{1/2} s))}{((\lambda/\Lambda)^{1/2} s)^{n-1}}\leq \lim_{s\to 0^+} \frac{1}{\lambda^{n/2}}\frac{\sF_A(E;\W_{x_0}(x_0,\Lambda^{-1/2}s))}{(\Lambda^{-1/2}s)^{n-1}}\nonumber\\
		 		&\leq \frac{1}{\lambda^{n/2}}\frac{\sF_A(E;\W_{x_0}(x_0,\Lambda^{-1/2}r))}{(\Lambda^{-1/2}r)^{n-1}}+C r^\alpha\leq \frac{\Lambda^{(n-1)/2}}{\lambda^{n/2}}\frac{\sF_A(E;\B(x_0,r))}{r^{n-1}}+Cr^\alpha\nonumber\\
		 		&\leq 	
		 		\Big(\frac{\Lambda}{\lambda}\Big)^{n/2}\frac{P(E;\B(x_0,r))}{r^{n-1}}+C r^\alpha.
			\end{align}
			Hence \eqref{almost lower perimeter bound} holds for $x_0\in\rb E$.
			
			Now consider the general case $x_0\in\p E$ (but perhaps not in $\rb E$). Given $0<s<r$, there is $y_0\in\rb E$ with $\B(y_0,s)\subset\B(x_0,r)$ by $\spt\mu_{E}=\p E=\bar{\rb E}$. It follows that
			\begin{align}
				\Big(\frac{s}{r}\Big)^{n-1}\frac{P(E;\B(y_0,s))}{s^{n-1}}	\leq \frac{P(E;\B(x_0,r))}{r^{n-1}}	 
			\end{align}
			and so applying \eqref{almost lower perimeter bound} at $y_0\in\rb E$ gives
			\begin{align}
				\Big(\frac{s}{r}\Big)^{n-1}\Big(\omega_{n-1}(\lambda/\Lambda)^{n/2}-C s^\alpha\Big)\leq\frac{P(E;\B(x_0,r))}{r^{n-1}}.	 
			\end{align}
			Sending $s\to r^-$ completes the proof.
		\end{proof}
		
		Let us recall a definition. For a set of locally finite perimeter of $E$, the \textbf{essential boundary} of $E$, denoted by $\p^e E$, is the set of points with neither full nor zero volume density, that is,
	\begin{align}
		\p^e E=\R^n\setminus (E^{(0)}\cup E^{(1)}).
	\end{align}
	Here $E^{(t)}$ denotes the points of volume density $t$, that is,
	\begin{align}
	 	E^{(t)}=\Big\{x\in\R^n : \lim_{r\to 0^+}\frac{|E\cap \B(x,r)|}{\omega_n r^n}=t \Big\}.
	\end{align}
    In general, we always have $\p^e E\subset \p E$. Federer's theorem states that $\Hm(\p^e E\setminus\rb E)=0$ for sets of locally finite perimeter in $\R^n$.

	A consequence of the volume density bounds \eqref{volume bounds} in the following proposition is that the topological boundary of an almost-minimizer $E$ cannot contain any points of zero or full volume density, that is, the essential boundary $\p^e E$ in $U$ equals the topological boundary $\p E$ in $U$. This fact precludes the existence of sharp cusps in the topological boundary of $E$ as well as prevents two sheets of the topological boundary from touching tangentially. The perimeter bounds \eqref{perimeter bounds} show that the perimeter measure for $E$ is $(n-1)$-Ahlfors regular up to scale $r_0$.
	
	\begin{proposition}[Volume and perimeter bounds for almost-minimizers]\label{volume and perimeter bounds} There exist positive constants  $c=c(n,\lambda,\Lambda)\in (0,1)$, $C=C(n,\lambda,\Lambda,\kappa,\alpha, r_0)$, and $\epsilon=\epsilon(n,\lambda,\Lambda,\kappa,\alpha, r_0, ||A||_{C^\alpha})$ with the following property. with the following property. If $E$ is a $(\kappa,\alpha)$-almost-minimizer of $\sF_A$ in $U$ at scale $r_0$, then for every $x_0\in U\cap\p E$ with $r<d=\min\{r_0,\dist(x_0,\p U),\epsilon\}<\infty$,
			\begin{align}\label{volume bounds}
				c\leq \frac{|E\cap \B(x_0,r)|}{\omega_n r^n}\leq 1-c,
			\end{align}
			and
			\begin{align}\label{perimeter bounds}
				c\leq \frac{P(E; \B(x_0,r))}{r^{n-1}}\leq C
			\end{align}
			Moreover, the volume density bounds \eqref{volume bounds} imply $\p E\cap U=\p^e E\cap U$ and so Federer's theorem gives
			\begin{align}\label{singular set n-1 measure zero}
				\Hm(U\cap (\p E\setminus\rb E))=0.
			\end{align}
		\end{proposition}
			
		\begin{proof}
			The upper bound of \eqref{perimeter bounds} was proved in Lemma \ref{upper perimeter bound lemma}. For the lower bound of \eqref{perimeter bounds} take $\epsilon>0$ small enough so that $C\eps^\alpha\leq (1/2)\omega_{n-1}(\lambda/\Lambda)^{n/2}$ where $C$ is the constant in Proposition \ref{lower bound on density ratio}.
			
			Recall from the proof of Lemma \ref{upper perimeter bound lemma}   that for $m(r)=|E\cap\B(x_0,r)|$ we have $m'(r)=\Hm(E^{(1)}\cap\B(x_0,r))$ for a.e. $r<d$. Then the inequality \eqref{perimeter density inequality} becomes
			\begin{align}
				P(E;\B(x_0,r))\leq C (m'(r)+r^{\alpha+n-1}).	
			\end{align}
			So by the lower bound of \eqref{perimeter bounds} we have		
			\begin{align}
					c r^{n-1}\leq C m'(r)+Cr^{\alpha+n-1}.	
			\end{align}
			Taking $\epsilon$ small enough so that $C\epsilon^\alpha\leq c/2$ and relabeling $c/2$ to $c$ gives $c r^{n-1}\leq  m'(r)$ for a.e. $r<d$. Integrating on $(0,r)$ and modifying constants gives $c\omega_n r^n\leq m(r)=|E\cap\B(x_0,r)|$ which is the lower bound of \eqref{volume bounds}. Since $E^c$ is also a $(\kappa,\alpha)$-almost-minimizer of $\sF_A$, we can apply this lower bound of \eqref{volume bounds} to get $c\omega_n r^n\leq |E^c\cap\B(x_0,r)|$ which gives the upper bound of \eqref{volume bounds}.
			
			Federer's theorem \cite[Theorem 16.2]{maggi} states $\Hm(\p^e E\setminus\rb E)=0$. The volume density bounds \eqref{volume bounds} imply $\p E\cap U=\p^e E\cap U$ and hence $\Hm(U\cap (\p E\setminus\rb E))=0$
		\end{proof}		
		
		Hence, given any $(\kappa,\alpha)$-almost-minimizer $E$ of $\sF_A$ in $U$ at scale $r_0$, we can shrink $r_0$ by a fixed amount, depending only on the universal constants $n,\lambda,\Lambda,\kappa,\alpha$, and an upper bound for $||A||_{C^\alpha}$, so that at points $x_0\in U\cap\p E$ the volume and perimeter bounds hold for all $r<\min\{r_0,\dist(x_0,\p U)\}$. Throughout the rest of this paper, we will work at this smaller scale and use the volume and perimeter bounds.
		
	\subsection{Compactness for the class of $\sF_A$-energies}\
	
	In addition to having fixed $n$, $\lambda$, $\Lambda$, $\alpha$, fix a positive constants $M_1$ and $M_2$. Define the class of admissible matrix-valued functions $\sA$ by 
		\begin{align}
			\sA=\Big\{A\in C^\alpha(\R^n;\R^n\otimes\R^n): \begin{array}{lr} A \text{ is symmetric},\ ||A||_{C^\alpha}\leq M_1,\ ||A(x)||\leq M_2,\text{ and }\\
			\lambda|\xi|^2\leq \la A(x)\xi,\xi\ra\leq \Lambda |\xi|^2\text{ for all } x,\xi\in\R^n
    	    \end{array}\Big\},
 		\end{align}
		
	\begin{lemma}\label{matrix compactness}
		$\sA$ is compact in the topology of uniform convergence on compact sets as a subspace of $C(\R^n;\R^n\otimes\R^n)$.
	\end{lemma}
	
	\begin{proof}
	Let $\{A_h\}_{h\in\NN}$ be a sequence in $\sA$. We apply Arzel\`a-Ascoli to  $\{A_h\}_{h\in\NN}$ noting that pointwise-boundedness follows from  $||A_h(x)||\leq M_2$ and equicontinuity follows from $||A_h||_{C^\alpha}\leq M_1$. Hence there is a subsequence $\{A_{h(k)}\}_{k\in\NN}$ and $A\in C(\R^n;\R^n\otimes\R^n)$ such that $A_{h(k)}\to A$ uniformly on compact sets. It follows that $A$ is symmetric and $||A(x)||\leq \lim_{k\to\infty} ||A_{h(k)}(x)||\leq M_2$ for any $x\in\R^n$. For any $x,y,\xi\in \R^n$,
	\begin{align}
		||A(x)-A(y)||&\leq ||A(x)-A_{h(k)}(x)||+M_1|x-y|^\alpha+||A_{h(k)}(y)-A(y)||,\nonumber \\
		\lambda|\xi|^2 &\leq \la A(x)\xi,\xi\ra+\la (A_{h(k)}(x)-A(x))\xi,\xi\ra\leq \Lambda |\xi|^2 .
	\end{align}
	Sending $h(k)\to\infty$, gives $||A||_{C^\alpha}\leq M_1$ and $\lambda|\xi|^2\leq \la A(x)\xi,\xi\ra\leq \Lambda |\xi|^2$. Thus $A\in\sA$ and so $\sA$ is compact.
	\end{proof}
	
	Fix an open set $U$ and $\kappa\geq 0$. Define the class $\sM$ of almost-minimizers of $\sF_A$ in $U$ for $A\in\sA$ by 
	\begin{align}
		\sM=\Big\{E\subset\R^n: E\text{ is a $(\kappa,\alpha)$-almost-min. of } \sF_A \text{ in }U\text{ at scale } r_0 \text{ for some } A\in\sA\text{ and }r_0>0\Big\}. 
	\end{align} 
	We will show that $\sM$ is compact by separately proving precompactness and closedness.
	
	\begin{proposition}[Precompactness of $\sM$]\label{precompactness}
		Suppose that $\{E_h\}_{h\in\NN}\subset\sM$ (that is, $E_h$ is a $(\kappa,\alpha)$-almost-minimizer of $\sF_{A_h}$ in $U$ at scale $r_h$ for some $A_h\in\sA$ and $r_h>0$), and that $r_0=\liminf_{h\to\infty} r_h>0$. For any open $V\subsetcc U$ with $P(V)<\infty$, there exist $h(k)\to\infty$ as $k\to\infty$, a set of finite perimeter $E\subset V$, and $A\in\sA$ such that
		\begin{align}
			V\cap E_{h(k)}\to E,\qquad \mu_{V\cap E_{h(k)}}\wkly \mu_E,\nonumber\\
			A_{h(k)}\to A\text{ uniformly on compact sets}.
		\end{align}
	\end{proposition}
	
	\begin{proof}
		First we choose $h(k)\to \infty$ as $k\to\infty$ such that $\lim_{k\to\infty} r_{h(k)}=r_0$.
		Let $x\in V$ and $\B(x,4r)\subset U$ with $2r<r_0$. Let $k_0$ be such that $2r<r_{h(k)}$ for $k\geq k_0$. If $P(E_{h(k)}; \B(x,r))>0$, there is $y\in \B(x,r)\cap \p E_{h(k)}$ and so $
			P(E_{h(k)}; \B(x,r))\leq P(E_{h(k)}; \B(y,2r))\leq C (2r)^{n-1}$
		by upper density bound (\ref{perimeter bounds})
		since $\B(y,2r)\subset \B(x,4r)\subset U$. By \cite[(16.10)]{maggi}, we have for $k\geq k_0$ that
		\begin{align}
			P(E_{h(k)}\cap \B(x,r))\leq P(E_{h(k)}; \B(x,r))+P(\B(x,r))\leq C r_0^{n-1}<\infty. 
		\end{align}
		and so $\sup_k P(E_{h(k)}\cap \B(x,r))<\infty$. Since $V$ is open and compactly contained in $U$, the balls with centers in $V$ that are contained in $U$ form a covering for $\bar V$. Hence we may cover $V$ by finitely many balls $\{\B_j\}_{j=1}^N$ where  $\B_j=\B(x_j,s_j)$ satisfy $\B(x_j, 4s_j)\subset U$ with $x_j\in V$ and $2s_j<r_0$ for $1\leq j\leq N$. Choose $R>0$ so that $\bigcup_{j=1}^N \B_j\subset \B_R$. Then
		\begin{align}
			P\Big(E_{h(k)}\cap \big(\bigcup_{j=1}^N\B_j\big)\Big)&\leq P\Big(E_{h(k)}; \bigcup_{j=1}^N \B_j\Big)+P\Big(\bigcup_{j=1}^N \B_j\Big)\leq\sum_{j=1}^N P(E_{h(k)}; \B_j)+P\Big(\bigcup_{j=1}^N \B_j\Big)\nonumber\\
			&\leq  C r_0^{n-1}N+P(\B_R)<\infty 
		\end{align}
		and so we may apply Theorem \ref{compactness} to construct a set $F\subset \B_R$ of finite perimeter and a further subsequence indices $h(k)$ such that $E_{h(k)}\cap \big(\bigcup_{j=1}^N \B_j\big)\to F$.
		Setting $E=V\cap F$, we have that $V\cap E_{h(k)}\to E$ and $
			\sup_k P(V\cap E_{h(k)})\leq \sup_k P(E_{h(k)}\cap \big(\bigcup_{j=1}^N \B_j\big))<\infty$.
		Finally, given $\phi \in C_c^0(\R^n)$ and $\psi\in C_c^1(\R^n)$, we have
		\begin{align}
			\bigg|\int_{\R^n}& \phi\:d\mu_{V\cap E_{h(k)}}-\int_{\R^n}\phi\: d\mu_E \bigg|\nonumber\\
			&\leq \bigg|\int_{\R^n}(\phi-\psi)\:d\mu_{V\cap E_{h(k)}} \bigg|
			+\bigg|\int_{\R^n}\psi\:d\mu_{V\cap E_{h(k)}}-\int_{\R^n}\psi\: d\mu_E \bigg|+\bigg|\int_{\R^n}(\psi-\phi)\:d\mu_E \bigg|\nonumber\\
				&\leq ||\phi-\psi||_{\sup}\sup_k|\mu_{V\cap E_{h(k)}}|(\R^n)+\bigg|\int_{V\cap E_{h(k)}}\nabla\psi\:dx-\int_{E}\nabla\psi\: dx\bigg|+||\phi-\psi||_{\sup} |\mu_{E}|(\R^n)\nonumber\\
			&\leq ||\phi-\psi||_{\sup}\sup_k|\mu_{V\cap E_{h(k)}}|(\R^n)+||\nabla\psi||_{\sup} |(V\cap E_{h(k)})\Delta E|+||\phi-\psi||_{\sup}|\mu_{E}|(\R^n). 
		\end{align}
		Since $V\cap E_{h(k)}\to E$, this gives
			\begin{align}
			\limsup_{k\to\sup}\bigg|\int_{\R^n}\phi\:d\mu_{V\cap E_{h(k)}}-\int_{\R^n}\phi\: d\mu_E \bigg|
			\leq ||\phi-\psi||_{\sup}\big(\sup_k|\mu_{V\cap E_{h(k)}}|(\R^n)+|\mu_{E}|(\R^n)\big). 
		\end{align}
	So by density of $C_c^1(\R^n)$ in $C_c^0(\R^n)$ in the sup norm, we have $\mu_{V\cap E_{h(k)}}\wkly\mu_E$. Finally, by Lemma \ref{matrix compactness} we may extract a further subsequence such that, up to relabeling, we also have $A_{h(k)}\to A$ uniformly on compact sets.
	\end{proof}
	
	\begin{proposition}[Closedness of $\sM$]\label{closedness}
		Suppose that $\{E_h\}_{h\in\N}\subset \sM$ (that is, $E_h$ is a $(\kappa,\alpha)$-almost-minimizer of $\sF_{A_h}$ in $U$ at scale $r_h$ for some $A_h\in\sA$ and $r_h>0$), $r_0=\liminf_{h\to\infty} r_h>0$, $V\subsetcc U$ is an open set with $P(V)<\infty$ such that $V\cap E_h\to E$ for a set of finite perimeter $E$, and $A_h\to A$ uniformly on compact sets for some $A\in\sA$. Then $E$ is a $(\kappa,\alpha)$-almost-minimizer of $\sF_A$ in $V$ at scale $r_0$. Moreover,
		\begin{align}
			& \mu_{V\cap E_h}\wkly \mu_E,\label{weak convergence}\\
			& \sF_{A_h}(E_h;\:\cdot\:)\wkly\sF_A(E;\:\cdot\:)\text{ in } V\label{weak convergence of energies}
		\end{align}
		where we view $\sF_{A_h}(E_h;\:\cdot\:)$ and $\sF_{A}(E;\:\cdot\:)$ as Radon measures. In particular,
		\begin{enumerate}
			\item[(i)] if $x_h\in V\cap \p E_h$, $x_h\to x$, and $x\in V$, then $x\in V\cap \p E$;	
			\item[(ii)] if $x\in V\cap \p E$, then there exists $\{x_h\}_{h\in\N}$ with $x_h\in V\cap \p E_h$ such that $x_h\to x$.
		\end{enumerate}
	\end{proposition}

	\begin{proof}
		By the same argument as in the proof of Proposition \ref{precompactness} we can show $\sup_h P(V\cap E_h)<\infty$. The weak convergence $\mu_{V\cap E_h}\wkly \mu_E$ of \eqref{weak convergence} follows from $V\cap E_h\to E$ as also shown in the proof of Proposition \ref{precompactness}.
	
		To show that $E$ is a $(\kappa,\alpha)$-almost-minimizer of $\sF_A$ our strategy is as follows. Given a competitor $F$ for $E$, we modify $F$ to construct competitors $F_h$ and apply the almost-minimality of $E_h$ with respect to $\sF_{A_h}$. We then pass the minimality inequalities through limits to obtain the desired almost-minimality inequality for $E$.

		Suppose $E\Delta F\subsetcc V\cap \B(x,r)$ with $x\in V$ and $r<r_0$. For $y\in V$, set $d(y)=\min\{r_0, \dist(y,\p V)\}>0$. Since $\Hm\rstr\: \rb E_h$ and $\Hm\rstr\: \rb F$ are Radon measures, we have that for a.e. $s\in (0,d(y))$,
		\begin{align}\label{boundaries}
			\Hm(\p \B(y,s)\cap \rb F)=\Hm(\p \B(y,s)\cap \rb E_h)=0,\ \forall h\in\NN.
		\end{align}
		Note that $|(E^{(1)}\Delta E_h^{(1)})\cap \B(y, d(y))|=|(E\Delta E_h)\cap \B(y, d(y))|$ because a Lebesgue measurable set is equivalent to its set of points of full density. By the coarea formula, $V\cap E_h\to E$, and $\B(y, d(y))\subset V$, it follows that
		\begin{align}
			\int_0^{d(y)} \Hm((E^{(1)}\Delta E_h^{(1)})\cap \p \B(y,s))\:ds=|(E^{(1)}\Delta E_h^{(1)})\cap \B(y, d(y))|\to 0 
		\end{align}
		as $h\to\infty$. Consequently, by Fatou's lemma,
		\begin{align}\label{liminf}
			\liminf_{h\to\infty} \Hm((E^{(1)}\Delta E_h^{(1)})\cap \p \B(y,s))=0
		\end{align}
		for a.e. $s\in (0, d(y))$. Since $E\Delta F$ is compactly contained in $V\cap \B(x,r)$, we may find finitely many balls $\{\B(y_j, s_j)\}_{j=1}^N$ with $y_j\in V$ and $s_j\in (0,d(y_j))$ satisfying \eqref{boundaries} and \eqref{liminf} with $y=y_j$, $s=s_j$ such that, setting  $G=\bigcup_{j=1}^N \B(y_j,s_j)$, we have $E\Delta F\subsetcc G\subsetcc V\cap \B(x,r)$.
		Now consider the comparison sets $F_h$ defined by $F_h=(F\cap G)\cup (E_h\setminus G)$. Since $\p G\subset\bigcup_{j=1}^N \p \B(y_j,s_j)$, by \eqref{boundaries} there holds
		\begin{align}\label{boundaries measure}
			\Hm(\p G\cap \rb F)=\Hm(\p G\cap \rb E_h)=0,\ \forall h\in\NN.
		\end{align}
		Additionally, $E^{(1)}\cap \p G=F^{(1)}\cap \p G$ since $E\Delta F\subsetcc G$  so that by \eqref{liminf} there holds
		\begin{align}\label{liminf measure}
			\liminf_{h\to\infty} \Hm((F^{(1)}\Delta E_h^{(1)})\cap \p G)=0.
		\end{align}
		Observe that $E_h\Delta F_h\subset G\subsetcc U\cap \B(x,r)$ with $x\in U$. Since $r<r_0=\liminf_{h\to\infty} r_h$, there is $h_0$ such that $r<r_h$ for all $h\geq h_0$. For now fix $h\geq h_0$. By \eqref{boundaries measure} we can apply Proposition \ref{comparison sets} to obtain
		\begin{align}
			\sF_{A_h}(F_h;\B(x,r))&=\sF_{A_h}(F;G)+\sF_{A_h}(E_h;\B(x,r)\setminus\bar G)+\sF_{A_h}(G; F^{(1)}\Delta E_h^{(1)})\nonumber\\
			&\leq\sF_{A_h}(F;G)+\sF_{A_h}(E_h;\B(x,r)\setminus G)+\Lambda^{1/2}\Hm((F^{(1)}\Delta E_h^{(1)})\cap \p G). 
		\end{align}
		Since $F_h$ is a competitor for the $\sF_{A_h}$-almost-minimality of $E_h$, we have
		\begin{align}
			\sF_{A_h}(E_h;\B(x,r))\leq \sF_{A_h}(F;G)+\sF_{A_h}(E_h;\B(x,r)\setminus G)+\Lambda^{1/2}\Hm((F^{(1)}\Delta E_h^{(1)})\cap \p G)
			+\kappa r^{\alpha+n-1} 
		\end{align} which simplifies to 
		\begin{align}\label{comp ineq 1}
			\sF_{A_h}(E_h;G)\leq \sF_{A_h}(F;G)+\Lambda^{1/2}\Hm((F^{(1)}\Delta E_h^{(1)})\cap \p G)	+\kappa r^{\alpha+n-1}.	
		\end{align}
		Similar to \eqref{integrand ineq1} we have $\la A_h(y)\nu_F,\nu_F\ra^{1/2}\leq \la A(y)\nu_F,\nu_F\ra^{1/2}+C||A_h(y)-A(y)||$. Integrating yields
		\begin{align}
			\sF_{A_h}(F;G)
			&\leq\sF_A(F;G)+C||A_h-A||_{\sup G} P(F; G) 
		\end{align}
		where we set $||A_h-A||_{\sup{G}}=\sup_{y\in G} ||A_h(y)-A(y)||$. Taking the $\limsup$ as $h\to\infty$ gives
		\begin{align}\label{comp ineq 2}
			\limsup_{h\to\infty}\sF_{A_h}(F;G)
			&\leq\sF_A(F;G)		
		\end{align}
		because $\limsup_{h\to 0} ||A_h-A||_{\sup G}=0$ by the uniform convergence $A_h\to\ A$ on compact sets. Similarly,
		\begin{align}
			\sF_{A}(V\cap E_h; G)\leq\sF_{A_h}(V\cap E_h;G)+C||A-A_h||_{\sup{G}}\:P(V\cap E_h; G).
		\end{align} 
		Using the fact that $\liminf (a_h+b_h)\leq\liminf a_h+\limsup b_h$ for nonnegative sequences $\{a_h\},\{b_h\}$, we have
		\begin{align}
			\liminf_{h\to\infty}\sF_A(V\cap E_h; G)\leq \liminf_{h\to\infty} \sF_{A_h}(V\cap E_h;G)+\limsup_{h\to\infty} \big(C||A-A_h||_{\sup G}\: P(V\cap E_h;G)\big) .
		\end{align}		
		By $\sup_h P(V\cap E_h; G)<\infty$ and $\limsup_{h\to 0} ||A_h-A||_{\sup G}=0$, this becomes
		\begin{align}\label{liminf comp ineq on compact sets}
			\liminf_{h\to\infty}\sF_A(V\cap E_h; G)\leq \liminf_{h\to\infty} \sF_{A_h}(V\cap E_h;G). 
		\end{align}
	    Noting $\sF_{A_h}(V\cap E_h; G)=\sF_{A_h}(E_h; G)$ since $G\subsetcc V$ and using the lower semicontinuity of $\sF_A$ with respect to $\mu_{V\cap E_h}\wkly \mu_E$ by Proposition \ref{lsc}, this implies
		\begin{align}\label{comp ineq 3}
			\sF_A(E;G)\leq \liminf_{h\to\infty} \sF_{A}(E_h;G)\leq  \liminf_{h\to\infty} \sF_{A_h}(E_h;G).
		\end{align}
		Now we combine our estimates, again using that
		 $\liminf_h (a_h+b_h)\leq\limsup_h a_h+\liminf_h b_h$, and obtain
		\begin{align}\label{comp ineq for limits}
			\sF_A(E;G) &\leq\liminf_{h\to\infty} \sF_{A_h}(E_h;G)&&\text{ by }\eqref{comp ineq 3}\nonumber\\
			&\leq \liminf_{h\to\infty} \big(\sF_{A_h}(F;G)+\Lambda^{1/2}\Hm((F^{(1)}\Delta E_h^{(1)})\cap \p G)+\kappa r^{\alpha+n-1}\big)&&\text{ by }\eqref{comp ineq 1}\nonumber\\
			&\leq \limsup_{h\to\infty} \sF_{A_h}(F;G)+\Lambda^{1/2}\liminf_{h\to\infty}\Hm((F^{(1)}\Delta E_h^{(1)})\cap \p G)+\kappa r^{\alpha+n-1}\nonumber\\
			&\leq \sF_{A}(F;G)+\kappa r^{\alpha+n-1}.&&\text{ by }\eqref{comp ineq 2}\text{ and }\eqref{liminf measure} 
		\end{align}
		Since $E\Delta F\subsetcc G$, we can add $\sF_A(E; \B(x,r)\setminus G)=\sF_A(F; \B(x,r)\setminus G)$ to obtain
		\begin{align}
			\sF_A(E;\B(x,r))\leq  \sF_A(F; \B(x,r))+\kappa r^{\alpha+n-1} 
		\end{align}
		as desired.
 		
		Next we prove the weak convergence of energy measures \eqref{weak convergence of energies}. Let $\Phi_h=\sF_{A_h}(V\cap E_h;\:\cdot\:)$ and $\Phi=\sF_A(E;\:\cdot\:)$ which are Radon measures on $\R^n$. It suffices to show the following claim.
		
		\begin{claim}
		If $\Psi$ is a Radon measure and $\{\Phi_{h(k)}\}_{k\in\NN}$ is a subsequence such that $\Phi_{h(k)}\wkly\Psi$, then $\Phi\rstr V=\Psi\rstr V$. 
		\end{claim}
		
		Indeed, suppose the claim is true. By sequential compactness of Radon measures (which applies since $\sup_h \Phi_h(\R^n)\leq \Lambda^{1/2}\sup_h P(V\cap E_h)<\infty$), for each subsequence of $\{\Phi_{h}\}_{h\in\NN}$ there exists a further subsequence that converges weakly to some Radon measure $\Psi$. By the claim $\Psi\rstr V=\Phi\rstr V$ and so $\Phi_h\rstr V$ converges weakly to $\Phi\rstr V$. Since $\Phi_h=\sF_{A_h}(V\cap E_h;\:\cdot\:)\rstr V=\sF_{A_h}(E_h;\:\cdot\:)\rstr V$ by the decomposition formula of the Gauss-Green measure for the intersection of two sets of locally finite perimeter (see (16.4) of \cite[Theorem 16.3]{maggi}), this will complete the proof of \eqref{weak convergence of energies}.
		
		Now we prove the above claim. Suppose $\Phi_{h(k)}\wkly\Psi$ for some Radon measure $\Psi$ and subsequence $\{\Phi_{h(k)}\}_{k\in\NN}$. For convenience we will just write the indices as $k$ instead of $h(k)$. 
		
		Let us show $\Phi\leq\Psi$ on $\sB(\R^n)$ where  $\sB(\R^n)$ denotes the Borel sets of $\R^n$. Let $W$ be an open bounded set and set $W_t=\{x\in W: \dist(x,\p W)>t\}$ for $t>0$. Choose $\phi\in C_c(W;[0,1])$ with $1_{W_t}\leq \phi$. Note that \eqref{liminf comp ineq on compact sets} holds for any bounded set in place of $G$ by the same argument. So applying \eqref{liminf comp ineq on compact sets} with $W_t$ in conjunction with the lower semicontinuity of $\sF_A$ with respect to $\mu_{V\cap E_h}\wkly \mu_E$ by Proposition \ref{lsc} gives
		\begin{align}
			\Phi(W_t) &=\sF_A(E; W_t)\leq \liminf_{k\to\infty} \sF_{A_k}(V\cap E_k; W_t)\nonumber\\
			&\leq \liminf_{k\to\infty} \int_{\rb (V\cap E_k)} \phi(x) \la A_k(x)\nu_{V\cap E_k},\nu_{V\cap E_k}\ra^{1/2}\dH=\int \phi\: d\Psi\leq \Psi(W).
		\end{align}
		By monotone convergence, taking $t\to 0^+$ gives $\Phi(W)\leq\Psi(W)$. Since $W$ was an arbitrary open, bounded set, it follows that $\Phi\leq\Psi$ on $\sB(\R^n)$. 
		
		Now let $\B(x,s_0)\subsetcc V$ with $s_0<r_0$. Define
		\begin{align}
			F_k=\big(E\cap \B(x,s))\cup (E_k\setminus \B(x,s))
		\end{align}
		for $s\in (0,s_0)$ with
		\begin{align}
			& \Hm(\rb E\cap \p \B(x,s))=\Hm(\rb E_k\cap \p \B(x,s))=0,\qquad \forall k\in\NN\nonumber\\
			& \liminf_{k\to\infty}\Hm(\p \B(x,s)\cap (E_k^{(1)}\Delta E^{(1)}))=0.	
		\end{align}
		This holds for a.e. $s\in (0,s_0)$. Then $E_k\Delta F_k\subset \B(x, s)\subsetcc U\cap \B(x, s_0)$ with $x\in U$ and $s_0<r_k$ for all $k$ larger than some $k_0$. By the same argument as with $G$ above to prove \eqref{comp ineq 1}, for such $k$ there holds
		\begin{align}\label{comp ineq 4}
			\sF_{A_k}(E_k;\B(x,s))\leq \sF_{A_k}(E;\B(x,s))+\Lambda^{1/2}\Hm((E^{(1)}\Delta E_k^{(1)})\cap \p\B(x,s))+\kappa s^{\alpha+n-1}.	
		\end{align}
		Since $\B(x,s)\subsetcc V$, 
		\begin{align}\label{comp ineq 5}
			\Phi_k(\B(x,s))=\sF_{A_k}(V\cap E_k; \B(x,s))=\sF_{A_k}(E_k; \B(x,s)).
		\end{align}
		Sending $k\to\infty$ and using the lower semicontinuity of weak convergent Radon measures, we have by the same reasoning as for \eqref{comp ineq for limits} that
		\begin{align}\label{comp ineq for limits 2}
			\Psi(\B(x,s))&\leq \liminf_{k\to\infty}\Phi_k(\B(x,s))\nonumber\\
    		&\leq \liminf_{k\to\infty}(\sF_{A_k}(E; \B(x,s))+\Lambda^{1/2}\Hm(\p \B(x,s)\cap (E_k^{(1)}\Delta E^{(1)}))+\kappa s^{\alpha+n-1})\nonumber\\
			 &= \limsup_{k\to\infty}\sF_{A_k}(E; \B(x,s))+\Lambda^{1/2}\liminf_{h
			\to\infty}\Hm(\p \B(x,s)\cap (E_k^{(1)}\Delta E^{(1)}))+\kappa s^{\alpha+n-1}\nonumber\\
			&\leq \sF_A(E; \B(x,s))+C \limsup_{k\to\infty}||A_k-A||_{\sup \B(x,s)}\: P(E; \B(x,s))+\kappa s^{\alpha+n-1}\nonumber\\
			&= \Phi(\B(x,s))+\kappa s^{\alpha+n-1}.
		\end{align}
		The lower perimeter bound \eqref{perimeter bounds} and comparability to perimeter give $c s^{n-1}\leq\Phi(\B(x,s))$. So by \eqref{comp ineq for limits 2} and $\Phi(\B(x,s))\leq \Psi(\B(x,s))$, which we know because $\Phi\leq\Psi$ on $\sB(\R^n)$, it follows that
		\begin{align}
			1-C s^\alpha\leq \frac{\Phi(\B(x,s))}{\Psi(\B(x,s))}\leq 1
		\end{align}
		for a.e. $s\in (0, s_0)$. Sending $s\to 0^+$ gives $D_{\Psi} \Phi=1$ for $\Psi$-a.e. $x\in V\cap \spt\Psi$. Since $\Phi\ll\Psi$, we have that $\Psi=\Phi$ on $\sB(V)$, the Borel subsets of $V$. This completes the proof of our claim.
		
		We finish by showing $(i)$ and $(ii)$.	For $(i)$, suppose $x_h\in V\cap \p E_h$ and $x_h\to x$ for $x\in V$. Let $r>0$ with $\B(x,r)\subsetcc V$. Then $\B(x_h,r/4)\subset \B(x,r/2)$ for large enough $h$. So by the weak convergence of the measures $\sF_{A_h}(E_h;\:\cdot\:)\wkly \sF_A(E;\:\cdot\:)$ in $V$, the lower perimeter bound of \eqref{perimeter bounds}, and $\lambda^{1/2}P(E_h;\B(x_h,r/4))\leq \sF_{A_h}(E_h; \B(x_h,r/4))$, we have
		\begin{align}
			0 &<c\: r^{n-1}\leq \limsup_{h\to\infty}\sF_{A_h}(E_h; {\B(x_h,r/4)})\leq\limsup_{h\to\infty}\sF_{A_h}(E_h; \overline{\B(x,r/2)})\nonumber\\
			&\leq \sF_A(E; \overline{\B(x,r/2)})\leq \Lambda^{1/2}P(E;{\B(x,r)}).
		\end{align}
		Hence $x\in\spt\mu_E=\p E$. For $(ii)$, suppose $x\in V\cap \p E$ and by way of contradiction that there does not exist a sequence $\{x_h\}_{h\in\NN}$ with $x_h\in V\cap \p E_h$ and $x_h\to x$. Then there is some $r>0$ and $h(k)\to\infty$ as $k\to\infty$ such that $\B(x,r)\subsetcc V$ and $\B(x,r)\cap \p E_{h(k)}=\varnothing$ for every $k\in\NN$. It follows that
		\begin{align}
			P(E; \B(x,r))&\leq \lambda^{-1/2}\sF_A(E; \B(x,r))\leq \lambda^{-1/2}\liminf_{k\to\infty} \sF_{A_{h(k)}}(E_{h(k)}; \B(x,r))\nonumber\\
			&\leq (\Lambda/\lambda)^{1/2}\liminf_{k\to\infty} P(E_{h(k)}; \B(x,r))=0,
		\end{align}
		contradicting the fact that $x\in\spt\mu_E=\p E$.
	\end{proof}
	
	\section{The Excess and the Height Bound}\label{excess and height bound section}
	
	The concept of the excess is a common key tool in the study of regularity for minimizers for many geometric variational problems. This quantity measures the average $L^2$-oscillation of outward unit normal vector $\nu_E$ with respect to a fixed direction $\nu$ and will eventually allow us to control the average $L^2$-oscillation of $\nu_E$ from its average. Our aim is to show decay estimates for the excess of almost-minimizers. For our variable coefficient surface energies and the change of variable, it will be useful to measure this oscillation over balls, ellipsoids, and cylinders.
	
	\subsection{Definition of the excess and basic properties}\
	
	Given $\nu\in\SS^{n-1}$ we decompose $\R^n$ into $\R^{n-1}\times\R$ by identifying $\R^{n-1}$ with $\nu^\perp$ and $\R$ with $\text{span}\ \nu$. With a slight abuse of notation, we write $x=(\pp x,\qq x)$ where
	$\pp:\R^n\to\R^{n-1}$ and $\qq:\R^n\to\R$ are the horizontal and vertical projections defined by
	\begin{align}
		\pp x=x-(x\cdot\nu)\:\nu\qquad\text{and}\qquad\qq x=x\cdot \nu. 
	\end{align}
	We define the \textbf{open cylinder} centered at $x_0\in\R^n$ of radius $r>0$ in the direction $\nu\in\SS^{n-1}$ by
	\begin{align}
		\C(x_0,r,\nu)=\big\{x\in\R^n: |\pp (x-x_0)|<r,\ |\qq(x-x_0)|<r\big\}.
	\end{align}
	Note that balls and cylinders are comparable as we have
	\begin{align}\label{comparability of balls and cylinders}
		\B(x_0,r)\subset \C(x_0,r,\nu)\subset \B(x_0,\sqrt 2\: r)
	\end{align}	
	and we have by \eqref{comparability to Wulff shapes} that balls and the ellipsoids $\W_{x_0}(x_0,r)$ are comparable. Thus balls, ellipsoids, and cylinders can all be mutually contained in each other by shrinking or enlarging them by fixed scales.
	
	Given a set of locally finite perimeter $E$, a point $x_0\in\p E$, a radius $r$ and direction $\nu\in\SS^{n-1}$, we define the \textbf{spherical excess} by
	\begin{align}\label{excess def}
		\ex_{\B}(E, x_0,r,\nu)=\frac{1}{r^{n-1}}\int_{\B(x_0,r)\cap \rb E} \frac{|\nu_{E}(x)-\nu|^2}{2}\dH(x),
	\end{align}
	the \textbf{ellipsoidal excess} by
	\begin{align}
		\ex_{\W}(E,x_0,r,\nu)=\frac{1}{r^{n-1}}\int_{\W_{x_0}(x_0,r)\cap\rb E} \frac{|\nu_E(x)-\nu|^2}{2}\dH(x),
	\end{align}	
	and the \textbf{cylindrical excess} by
	\begin{align}
		\ex_{\C}(E,x_0,r,\nu)=\frac{1}{r^{n-1}}\int_{\C(x_0,r,\nu)\cap \rb E} \frac{|\nu_E(x)-\nu|^2}{2}\dH(x).
	\end{align}
    Since balls, ellipsoids, and cylinders are comparable, if we can control one of these types of excess, we can control all of them. 
    
    As mentioned in Section \ref{basic properties section}, it will often be convenient to prove estimates at points $x_0\in \p E$ with the assumption $A(x_0)=I$. To do this, we make the change of variable under the transformation $T_{x_0}$. The next proposition shows that the excess of the image set under this transformation is comparable to that of the original set. 
	
	\begin{proposition}[Comparability of excess under change of variable $T_{x_0}$] There exists a positive constant $C=C(n,\lambda,\Lambda)$ with the following property. If $E$ is a set of locally finite perimeter, $E_{x_0}=T_{x_0}(E)$ for some $x_0\in \p E$, then for any $r>0$ and $\nu\in\SS^{n-1}$,
		\begin{align}\label{comparability of excess}
			C^{-1} \ex_\B(E_{x_0},x_0,r,\tilde \nu)\leq \ex_\W(E,x_0,r,\nu)\leq C\ex_\B(E_{x_0}, x_0,r,\tilde \nu)
		\end{align}
		where $\tilde \nu\in \SS^{n-1}$ is defined by
		\begin{align}
			\tilde \nu=\frac{A^{1/2}(x_0) \nu}{|A^{1/2}(x_0) \nu|},\text{ or equivalently }
			\nu=\frac{A^{-1/2}(x_0) \tilde \nu}{|A^{-1/2}(x_0) \tilde \nu|}.
		\end{align}
	\end{proposition}
	\begin{proof} 
		Without loss of generality assume $x_0=0$ and to simplify notation write $S=A^{1/2}(0)$, $\W_r=\W_0(0,r)$, and $\B_r=\B(0,r)$. Noting that $S$ is symmetric, the change of variable $y=T_{0}(x)=S^{-1}x$ gives by Proposition \ref{change of variable formula} that $
		\nu_{E_0}(T_{0}(x))=S\nu_E(x)/|S\nu_E(x)|\text{ for all }x\in\rb E$
	and
	\begin{align}
		\int_{\B_r\cap\rb E_{0}}|\nu_{E_{0}}(y)-\tilde\nu|^2\dH(y)=\int_{\W_r\cap \rb E}\bigg|\frac{S\nu_{E}(x)}{|S\nu_{E}(x)|}-\frac{S\nu}{|S\nu|}\bigg|^2\det S^{-1}\:|S \nu_E(x)|\dH(x). 
	\end{align}
	Note that $\det S^{-1}=\det A^{-1/2}(0)\leq \lambda^{-n/2}$, $|S\nu_E(x)|\leq \Lambda^{1/2}$, and
	\begin{align}
		\bigg|\frac{S\nu_{E}}{|S\nu_{E}|}-\frac{S\nu}{|S\nu|}\bigg|^2  \leq  2\bigg|\frac{S\nu_{E}}{|S\nu_{E}|}-\frac{S\nu}{|S\nu_{E}|}\bigg|^2
		+2\bigg|\frac{S\nu}{|S\nu_{E}|}-\frac{S\nu}{|S\nu|}\bigg|^2. 
	\end{align}
	For the first term, we have the estimate
		\begin{align}
		\bigg|\frac{S\nu_{E}}{|S\nu_{E}|}-\frac{S\nu}{|S\nu_{E}|}\bigg|^2 \leq  \frac{1}{|S\nu_E|^2}|S(\nu_{E}-\nu)|^2  \leq  \frac{\Lambda}{\lambda}|\nu_{E}-\nu|^2 
	\end{align}
	since the maximum eigenvalue of $S$ is bounded by $\Lambda^{1/2}$ and its minimum eigenvalue is bounded by $\lambda^{1/2}$. For the second term, we have the estimate
	\begin{align}
		\bigg|\frac{S\nu}{|S\nu_{E}|}-\frac{S\nu}{|S\nu|}\bigg|^2
		&\leq  |S\nu|^2\bigg|\frac{1}{|S\nu_E|}-\frac{1}{|S\nu|}\bigg|^2=\frac{|S\nu|^2}{|S\nu_E|^2|S\nu|^2}\bigg||S\nu|-|S\nu_E|\bigg|^2\nonumber \\
		&\leq \frac{1}{|S\nu_E|^2}|S(\nu-\nu_E)|^2 \leq \frac{\Lambda}{\lambda}|\nu_{E}-\nu|^2 
	\end{align}
	as above. It follows that
	\begin{align}
	    \bigg|\frac{S\nu_{E}}{|S\nu_{E}|}-\frac{S\nu}{|S\nu|}\bigg|^2  \leq  \frac{4\Lambda}{\lambda}|\nu_E-\nu|^2.
	\end{align}
	Hence
	\begin{align}
		\int_{\B_r\cap\rb E_{0}}|\nu_{E_{0}}-\tilde\nu|^2\dH \leq \frac{4\Lambda^{3/2}}{\lambda^{n/2+1}}\int_{\W_r\cap \rb E} |\nu_E-\nu|^2\dH, 
	\end{align}
	or equivalently, $\ex_\B(E_0, 0,r,\tilde\nu)\leq (4\Lambda^{n/2+1}/\lambda^{3/2}) \ex_{\W}(E,0,r,\nu)$.
		The upper bound for \eqref{comparability of excess} follows by a symmetric argument.
	\end{proof}
	
	We now recall several known properties of the excess, referring readers to \cite[Chapter 22]{maggi} for proofs of these facts.
	
	\begin{proposition}[Scaling of the excess]
	    If $E$ is a set of locally finite perimeter in $\R^n$, $x_0\in\p E$, $r>0$, $\nu\in\SS^{n-1}$, then
	    \begin{align}
	        \ex_{\C}(E,x_0,r,\nu)=\ex_{\C}(E_{x_0,r},0,1,\nu)
	    \end{align}
	    where $E_{x_0,r}=(E-x_0)/r$ as in Section \ref{basic properties section}.
	\end{proposition}
	
	\begin{proposition}[Zero excess implies being a half-space]
	    If $E$ is a set of locally finite perimeter in $\R^n$, with $\spt\mu_E=\p E$, $x_0\in\p E$, $r>0$, and $\nu\in\SS^{n-2}$, then $\ex_{\C}(E,x_0,r,\nu)=0$ if and only if $E\cap\C(x_0,r,\nu)$ is equivalent to the set $\big\{x\in\C(x_0,r,\nu): (x-x_0)\cdot\nu\leq 0\big\}$.
	\end{proposition}
	
	\begin{proposition}[Vanishing of the excess at the reduced boundary]\label{vanishing of the excess at the reduced boundary} If $E$ is a set of locally finite perimeter in $\R^n$ and $x_0\in\rb E$, then
	\begin{align}
	   \lim_{r\to 0^+} \inf_{\nu\in\SS^{n-1}} \ex_{\C}(E,x_0,r,\nu)=0.
	\end{align}
	\end{proposition}
	
	\begin{proposition}[Excess at different scales]\label{excess at different scales} If $E$ is a set of locally finite perimeter in $\R^n$, $x_0\in\p E$, $0<s<r$, $\nu\in\SS^{n-1}$, then
	\begin{align}
	    \ex_{\C}(E,x_0,s,\nu)\leq\Big(\frac{r}{s}\Big)^{n-1}\ex_{\C}(E,x_0,r,\nu).
	\end{align}
	\end{proposition}
	
	\begin{proposition}[Excess and changes of direction]\label{excess and changes of direction} For every $n\geq 2$, there exists a constant\\ $C=C(n,\lambda,\Lambda,\kappa,\alpha, r_0)$ with the following property. If $E$ is a $(\kappa,\alpha)$-almost-minimizer of $\sF_A$ in $U$ at scale $r_0$, then 
	\begin{align}
    	\ex_{\C}(E,x_0,r,\nu)\leq C\big(\ex_{\C}(E,x_0,r,\nu_0)+|\nu-\nu_0|^2\big)
	\end{align}
	whenever $x_0\in U\cap\p E$, $\B(x,2r)\subsetcc U$, $\nu,\nu_0\in\SS^{n-1}$.
	\end{proposition}
	
	\begin{proof}
	   It follows from the proof of \cite[Proposition 22.5]{maggi} using the upper density estimate of Proposition \ref{volume and perimeter bounds}.
	\end{proof}
	
	\subsection{Small-Excess Position and the Height Bound}\
	
	We now recall some standard lemmas we will need about the excess and almost-minimizers and recall the height bound. The first lemma states that if the excess of an almost-minimizer in a given cylinder is small enough, then in a smaller cylinder the topological boundary sits within a narrow strip. 

	\begin{lemma}[Small-excess position]\label{small-excess position} Given $n\geq 2$ and $t_0\in (0,1)$, there is a positive constant $\omega=\omega(t_0,n,\lambda,\Lambda,\kappa,\alpha,r_0)$ with the following property. If $E$ is a $(\kappa,\alpha)$-almost-minimizer of $\sF_A$ in $\C(x_0,2r, \nu)$ with $x_0\in\p E$, $2r<r_0$, $\nu\in\SS^{n-2}$, and 
		\begin{align}
			\ex_{\C}(E,x_0,2r,\nu)\leq\omega,
		\end{align}
		then
		\begin{align}
			&\frac{|\qq (x-x_0)|}{r}<t_0,\qquad\forall x\in \C(x_0,r,\nu)\cap\p E,\\
			& \Big|\Big\{x\in \C(x_0,r,\nu)\cap E: \frac{\qq(x-x_0)}{r}>t_0\Big\}\Big|=0,\text{ and }\\
			&\Big|\Big\{x\in \C(x_0,r,\nu)\setminus E: \frac{\qq(x-x_0)}{r}<-t_0\Big\}\Big|=0.
		\end{align}
	\end{lemma}
	
	\begin{proof}
		\cite[Lemma 3.8]{begtz} which applies by Proposition \ref{volume and perimeter bounds}.
	\end{proof}
	
	We define the \textbf{open disk} in $\R^{n-1}$ centered at $z\in\R^{n-1}$ and of radius $r>0$ by
	\begin{align}
	 	\D(z,r)=\big\{w\in\R^{n-1}: |z-w|<r\big\},
	\end{align}
	Thus we may write $\C(x_0,r,\nu)=\D(\pp x_0,r)\times (-r,r)$. 
	
	    The second lemma states that if the a set of locally finite perimeter satisfies the separation property of Lemma \ref{small-excess position}, then the difference of measure of perimeter sitting above a set $G\subset \D(\pp x_0,r)$ and $\Hm(G)$ defines a measure which we call the \textbf{excess measure}.
	
	\begin{lemma}[Excess measure]\label{excess measure} If $E$ is a set of locally finite perimeter in $\R^n$, with $0\in\p E$, and such that, for some $t_0\in (0,1)$,
		\begin{align}
			&\frac{|\qq (x-x_0)|}{r}<t_0,\qquad\forall x\in \C(x_0,r,\nu)\cap\p E,\\
			& \Big|\Big\{x\in \C(x_0,r,\nu)\cap E: \frac{\qq(x-x_0)}{r}>t_0\Big\}\Big|=0,\text{ and }\\
			&\Big|\Big\{x\in \C(x_0,r,\nu)\setminus E: \frac{\qq(x-x_0)}{r}<-t_0\Big\}\Big|=0,
		\end{align}
		then, setting for brevity $M=\C(x_0,r,\nu)\cap\rb E$, we have for every Borel set $G\subset \D(\pp x_0,r)$, function $\phi\in C_c(\D(\pp x_0,r))$, and $t\in (-1,1)$ that
		\begin{align}
			& \Hm(G)\leq\Hm(M\cap\pp^{-1}(G)),\\
			& \Hm(G)=\int_{M\cap\pp^{-1}(G)} (\nu_E\cdot \nu)\dH(x),\\
			& \int_{\D}\phi\: dx=\int_{M}\phi(\pp x)(\nu_E(x)\cdot \nu)\dH(x),\\
			&\int_{E_t\cap\D}\phi\: dx=\int_{M\cap\{\qq x>t\}} \phi(\pp x)(\nu_E(x)\cdot \nu)\dH(x),
		\end{align}
	    where $E_t=\{z\in\R^{n-1}| (z,t)\in E\}$. In fact, the set function
	    \begin{align}
	    	\zeta(G)&=P(E;\C(x_0,r,\nu)\cap\pp^{-1}(G))-\Hm(G)\\
	    	&=\Hm(M\cap\pp^{-1}(G))-\Hm(G)
	    \end{align}
	    defines a Radon measure on $\R^{n-1}$, concentrated on $\D(\pp x_0,r)$, called the \textbf{excess measure} of $E$ over $\D(\pp x_0,r))$.
	\end{lemma}
	
	\begin{proof}
		\cite[Theorem A.1]{begtz} which applies by Proposition \ref{volume and perimeter bounds}.
	\end{proof}
	
	We now state the main result we need from this section which is a strengthening of Lemma \ref{small-excess position} to quantitatively control the height of an almost-minimizer in a cylinder by the excess on a larger cylinder.

	\begin{proposition}[Height bound]\label{height bound}
		Given $n\geq 2$, there exist positive constants $\epsilon_0=\epsilon_0(n,\lambda,\Lambda,\kappa,\alpha,r_0)$ and $C_0=C_0(n,\lambda,\Lambda,\kappa,\alpha,r_0)$ with the following property. If $E$ is a $(\kappa,\alpha)$-almost-minimizer of $\sF_A$ in $\C(x_0,4r,\nu)$ at scale $r_0$ with $x_0\in\p E$, $4r<r_0$, and 
		\begin{align} 
		\ex_{\C}(E, x_0, 4r,\nu)\leq \epsilon_0,
		\end{align}
		then
		\begin{align}\label{height bound estimate}
		\sup\Big\{\frac{|\qq(x-x_0)|}{r}:x\in\C(x_0,r,\nu)\cap\p E\Big\}\leq C_0\ex_{\C}(E,x_0,4r,\nu)^{1/(2(n-1))}.
		\end{align}
	\end{proposition}
	
	\begin{proof}
		\cite[Theorem A.2]{begtz} which applies by Proposition \ref{volume and perimeter bounds}.
	\end{proof}
	
	Throughout the course of the proof of our regularity result, we shall keep track of a number of specific constants for which certain estimates hold. The estimate \eqref{height bound estimate} with the constants $C_0$ and $\epsilon_0$ from Lemma \ref{height bound} are the first of these. Subsequent $C_i$'s will be chosen to be larger than previous ones, i.e. $C_0\leq C_1\leq\dots$ and subsequent $\epsilon_i$'s will be chosen to be smaller than the previous ones, i.e. $\epsilon_0\geq \epsilon_1\geq \dots$. This way previous estimates will also hold under any smallness of the excess assumptions. We shall also choose $\epsilon_0\leq \omega(1/4,n,\lambda,\Lambda,\kappa,\alpha,r_0)$ so that the height of our topological boundary is at most $1/4$ of the cylinder.
	
	\section{The Lipschitz Approximation Theorem}\label{Lipschitz approximation theorem section}
	
	The next step in our proof is to show that, given a small excess assumption of an almost-minimizer in a cylinder, a large portion of the topological boundary can be covered by the graph of a Lipschitz function in a smaller cylinder. Moreover, if we assume $A(x_0)=I$, this Lipschitz function is quantitatively `almost-harmonic' at $x_0$ with an error controlled in terms of the excess and the scale. Given a direction $\nu\in\SS^{n-1}$ which decomposes $\R^n$ into $\R^{n-1}\times \R$, we denote the gradient in the first $n-1$ directions by $\nabla'$, that is, $\nabla'=(\p_1,\dots,\p_{n-1})$.

	\begin{theorem}[Lipschitz approximation theorem]\label{Lipschitz approximation theorem}
		There exist positive constants $\epsilon_1=\epsilon_1(n,\lambda,\Lambda,\kappa,\alpha,r_0)$, $\delta_0=\delta_0(n,\lambda,\Lambda,\kappa,\alpha,r_0)$, and $C_1=C_1(n,\lambda,\Lambda,\kappa,\alpha,r_0, ||A||_{C^\alpha})$ with the following property. If $E$ is a $(\kappa,\alpha)$-almost-minimizer of $\sF_A$ in $\C(x_0,r_0,\nu)$ at scale $r_0$ with $x_0\in\p E$, $13r<r_0$, and 
		\begin{align}
			\ex_{\C}(E,x_0,13r,\nu)\leq \epsilon_1,
		\end{align}
		then, setting
		\begin{align} 
			M=\C(x_0,r,\nu)\cap\p E\text{ and }M_0=\{x\in M: \sup_{0< s< 8r}\ex_{\C}(E,x,s,\nu)\leq \delta_0\},
		\end{align}
		there is a Lipschitz function $u:\R^{n-1}\to\R$ with $\Lip u\leq 1$ satisfying
		\begin{align}
		\sup_{\R^{n-1}} \frac{|u|}{r}\leq C_1\ex_{\C}(E,x_0,13r,\nu)^{1/2(n-1)}
		\end{align} such that the translation $\Gamma=x_0+\{(z,u(z)): z\in \D_r\}$ of the graph of $u$ over $\D_r$ contains $M_0$, that is, $M_0\subset M\cap\Gamma$, and covers a large portion of $M$ in the sense that
		\begin{align}\label{large portion}
			 \frac{\Hm(M\Delta \Gamma)}{r^{n-1}}\leq C_1 \ex_{\C}(E,x_0,13r,\nu).
		\end{align}
		Moreover, $u$ is `almost harmonic' in $\D_r$ in the sense that
		\begin{align}\label{small Dirichlet energy}
		\frac{1}{r^{n-1}}\int_{\D_r} |\nabla' u|^2\leq C_1 \ex_{\C}(E,x_0,13r,\nu)
		\end{align}
		and if  $A(x_0)=I$, then
		\begin{align}
 		 \frac{1}{r^{n-1}}\bigg|\int_{\D_r} \nabla' u\cdot \nabla' \phi\bigg|\leq C_1\sup_{\D_r}|\nabla'\phi| \big(\ex_{\C}(E,x_0,13r,\nu)+r^{\alpha/2})\qquad\forall\phi\in C_c^1(\D_r).\label{almost-harmonocity}
		\end{align}
	\end{theorem}

	\begin{proof}
		Without loss of generality we may assume $x_0=0$ and $\nu=e_n$. We simplify notation by setting $\C_r=\C(0,r,\nu)$.
		Everything up to and including \eqref{small Dirichlet energy} follows from \cite[Theorem A.3]{begtz} by Proposition \ref{volume and perimeter bounds}, for an $\epsilon_1$ chosen sufficiently small. We also choose $\epsilon_1$ small enough so that
		\begin{align}
		\boxed{\epsilon_1\leq \epsilon_0\leq \omega(1/4,n,\lambda,\Lambda,\kappa,\alpha,r_0)} 
		\end{align} 
		where $\omega$ is the constant from Lemma \ref{small-excess position} with $t_0=1/4$. It follows that
		\begin{align}
			&\frac{|\qq x|}{r}< \frac{1}{4}\ \ \forall x\in \C_r\cap \p E,\ \Big\{x\in \C_r\cap E: \frac{\qq x}{r}>\frac{1}{4}\Big\}=\varnothing,\text{ and }\Big\{x\in \C_r\setminus E: \frac{\qq x}{r}<-\frac{1}{4}\Big\}=\varnothing.\label{boundary position} 
		\end{align}
		Let $\phi\in \C_c^1(D_r)$. By considering $\nabla'\phi/\sup_{\D_r}|\nabla'\phi|$, we may assume $\sup_{\D_r}|\nabla'\phi|=1$ and reduce to proving 
 		\begin{align}
 		\frac{1}{r^{n-1}}\bigg|\int_{\D_r} \nabla' u\cdot \nabla' \phi\bigg|\leq C_1\big( \ex_{\C}(E,x_0,13r,e_n)+r^{\alpha/2}). 
		\end{align}	
		By the Fundamental Theorem of Calculus and the fact that $\phi=0$ on 
		$\p \D_r$, we have $\sup_{\D_r} |\phi|\leq r$. Let $\eta \in C_c^1((-3r/4,3r/4))$ be a cutoff function such that 
		\begin{align}
		\eta= 1 \text{ on } [-r/2,r/2],\ |\eta|\leq 1, \text{ and }|\eta'|\leq 5/r. 
		\end{align}
		and define $T:\R^n\to\R^n$ by $T(x)=\eta(\qq x)\phi(\pp x)e_n$. Then $T\in C_c^1(\D_r\times (-3r/4,3r/4);\R^n)$, $\sup_{\C_r}|T|\leq r$, and $
		\nabla T(x)=\eta(\qq x)\:\nabla'\phi(\pp x)\otimes e_n+\eta'(\qq x)\phi(\pp x)\: e_n\otimes e_n$.
		Hence
		\begin{align}
		|\nabla T(x)|=\sqrt{|\eta(\qq x)|^2 |\nabla'\phi(\pp x)|^2+|\eta'(\qq x)|^2 |\phi(\pp x)|^2}\leq 6 
		\end{align}
		and so $\sup_{\C_r}|\nabla T|\leq 6$. Consider the family of maps $f_t\colon\R^n\to\R^n$ defined by $
			f_t(x)=x+t\:T(x)$.
		Then $\nabla f_t=\Id+t\nabla T$ and so $Jf_t=\det(\Id+t\nabla T)$. We have that $||\nabla T(x)||\leq |\nabla T(x)|\leq 6$ where $||\cdot||$ denotes the operator norm and $|\cdot|$	denotes the Frobenius norm. It then follows by \cite[Lemma 17.4]{maggi} that there are positive constants $\epsilon(n)$, $C(n)$ such that
		\begin{align}\label{Jf_t}
			Jf_t=(1+t\divr T)+ O(C(n)t^2).
		\end{align}
		for $|t|<\epsilon(n)$. Since $\divr T$ is bounded, we can choose $\epsilon(n)$ so that $f_t$ is a diffeomorphism for $|t|<\epsilon(n)$. Letting $g_t=f_t^{-1}$, we also have by \cite[Lemma 17.4]{maggi} that $\nabla g_t\circ f_t=\Id-t\nabla T +O(C(n)t^2$ for $t<\epsilon(n)$. Choosing $\epsilon(n)\leq 1/8$, we claim that $E\Delta f_t(E)\subsetcc \C_r$ for $|t|<\epsilon(n)$. 
		
		To see why this is the case, take $y\in E\Delta f_t(E)$. Then $y=x+tT(x)$ for some $x\in\spt T$. By definition of $T$,  $\pp y=\pp x\in\spt\phi$ and $\qq x\in (-3r/4,3r/4)$. So $|\qq y|\leq |\qq x|+|\qq y-\qq x|\leq (3/4)r+|t| \sup_{\C_r} |T|<7r/8$
		since $|t|<\epsilon(n)\leq 1/8$ and $\sup_{\C_r} |T|\leq r$. Hence $y\in \spt\phi\times (-7r/8,7r/8)\subsetcc\C_r$.
		
		By \cite[Proposition 17.1]{maggi} we have that
		\begin{align}
			P(f_t(E);\C_r)=\int_{\C_r\cap\rb E} |(\nabla g_t\circ f_t)^*\nu_E| Jf_t \dH. 
		\end{align}
	
		\begin{claim}
		We can choose $\epsilon(n)$ small enough so that
		\begin{align}
			P(f_t(E); \C_r)=P(E; \C_r)+t\int_{\C_r\cap\rb E}\divr_E T(x)\dH +O(C(n)P(E; \C_r)t^2)
		\end{align}
		for all $|t|<\epsilon(n)$, where $\divr_E T=\divr T-\nu_E\cdot (\nabla T)^*\nu_E$.
		\end{claim}

		To prove the claim, observe that $
			\big|(\Id +t\nabla T)^*\nu_E\big|^2 = 1-2t\: \nu_E\cdot (\nabla T)^*\nu_E+t^2 |(\nabla T)^* \nu_E|^2$
		and so, since $\sqrt{1+x}=1-x/2+O(x^2)$ for small $|x|$ by Taylor's theorem, shrinking $\epsilon(n)$ as necessary, we have
		\begin{align}
			\big|(\nabla g_t\circ f_t)^*\nu_E\big|= \big|(\Id-t\nabla T)^*\nu_E\big|+O(C(n)t^2)= 1-t\: \nu_E\cdot (\nabla T)^*\nu_E+O(C(n)t^2) 
		\end{align}
		whenever $|t|<\epsilon(n)$. Combining this with \eqref{Jf_t} gives
		\begin{align}
		\big|(\nabla g_t\circ f_t)^*\nu_E\big|Jf_t=1+t(\divr T-\: \nu_E\cdot (\nabla T)^*\nu_E)+O(C(n)t^2) 
		\end{align}
		for $|t|<\epsilon(n)$.
		Integrating with respect to $\Hm\rstr (\C_r\cap \rb E)$ completes the proof of the claim.
	
		By the claim, Proposition \ref{volume and perimeter bounds}, and Lemma \ref{perturbation lemma}, it follows that
		\begin{align}
			|t|\:\Big|\int_{\C_r\cap\rb E}\divr_E T(x)\dH\Big| &\leq \big|P(f_t(E; \C_r))-P(E; \C_r))\big|+C(n) P(E; \C_r)|t|^2\nonumber\\
			 &\leq Cr^{\alpha+n-1}+C t^2 r^{n-1} 
		\end{align}
		whenever $|t|<\epsilon(n)$. Choosing $t=\epsilon(n) (r/r_0)^{\alpha/2}<\epsilon(n)$ gives that
		\begin{align}\label{div_E bound}
	    	\Big|\int_{\rb E\cap \C_r}\divr_E T(x)\dH\Big|\leq C r^{\alpha/2+n-1}.
		\end{align}
		
		Now, for $\Hm$-a.e. $x\in M\cap\Gamma$, there is $\lambda(x)\in\{-1,1\}$ such that
		\begin{align}
			\nu_E(x)=\lambda(x)\frac{(-\nabla' u(\pp x),1)}{\sqrt{1+|\nabla'u(\pp x)|^2}}. 
		\end{align}
		By \eqref{boundary position} and definition of $\eta$, we have $\eta (\qq x)=1$ on a neighborhood $M$ and so $\divr T(x)=0$ and $\nabla T(x)=\nabla'\phi(\pp x)\otimes e_n$ for $x\in M$. Hence for $\Hm$-a.e. $x\in M\cap\Gamma$, there holds
		\begin{align}
			\divr_E T(x)=\divr T(x)-\nu_E(x)\cdot ((\nabla T(x))^*\nu_E(x))=\frac{\nabla' u(\pp x)\cdot \nabla'\phi(\pp x)}{1+|\nabla'u(\pp x)|^2} 
		\end{align}
		since $\lambda(x)^2=1$. Thus 
		\begin{align}
		\bigg|\int_{\pp(M\cap\Gamma)} \frac{\nabla' u\cdot \nabla'\phi}{\sqrt{1+|\nabla'u|^2}}\bigg|&=\bigg|\int_{M\cap\Gamma} \divr_E T\dH\bigg|\nonumber\\
		&\leq \bigg|\int_{M} \divr_E T\dH\bigg|+\bigg|\int_{M\setminus\Gamma} \divr_E T\dH\bigg|\nonumber\\
		&\leq C (r^{\alpha/2+n-1}+\Hm(M\setminus\Gamma)) 
		\end{align}
		by \eqref{div_E bound} and since $|\divr_E T|\leq C(n)|\nabla T|\leq C(n)$. Since $\D_r=\pp(\Gamma)$, it follows that
		\begin{align}
			\bigg|\int_{\D_r} \frac{\nabla' u\cdot \nabla'\phi}{\sqrt{1+|\nabla'u|^2}}\bigg|
			&\leq \bigg|\int_{\pp(M\cap\Gamma)}  \frac{\nabla' u\cdot \nabla'\phi}{\sqrt{1+|\nabla'u|^2}}\bigg|+
			\bigg|\int_{\pp(\Gamma\setminus M)}  \frac{\nabla' u\cdot \nabla'\phi}{\sqrt{1+|\nabla'u|^2}}\bigg|\nonumber\\
			&\leq C (r^{\alpha/2+n-1}+\Hm(M\setminus\Gamma))+C\Hm(\Gamma\setminus M)\nonumber\\
			&\leq C(\Hm(M\Delta \Gamma)+ r^{\alpha/2+n-1}) 
		\end{align}
		where we used the fact that $\Lip u\leq 1$, $|\nabla'\phi|\leq 1$, and $\Hm(\pp (\Gamma\setminus M))\leq\Hm(\Gamma\setminus M)$. Again, using $\Lip u\leq 1$, $|\nabla'\phi|\leq 1$, we have
		\begin{align}
			\int_{\D_r}\bigg|\nabla'u\cdot\nabla'\phi-\frac{\nabla' u\cdot \nabla'\phi}{\sqrt{1+|\nabla'u|^2}}\bigg|&\leq \int_{\D_r}|\nabla'u|\:|\nabla'\phi|\:\Big|1-\frac{1}{\sqrt{1+|\nabla'u|^2}}\Big|\nonumber\\
			&\leq \int_{\D_r}\frac{\sqrt{1+|\nabla'u|^2}-1}{\sqrt{1+|\nabla'u|^2}}\nonumber\\
		&=\int_{\D_r}\frac{|\nabla'u|^2}{\sqrt{1+|\nabla'u|^2}(\sqrt{1+|\nabla'u|^2}+1)}\nonumber\\
		&\leq \frac{1}{2}\int_{\D_r}|\nabla'u|^2	. 
		\end{align}
		By \eqref{large portion} and \eqref{small Dirichlet energy}, it follows that
		\begin{align}
			\bigg|\int_{\D_r}\nabla'u\cdot\nabla'\phi \bigg|&\leq \bigg|\int_{\D_r}\Big(\nabla'u\cdot\nabla'\phi-\frac{\nabla' u\cdot \nabla'\phi}{\sqrt{1+|\nabla'u|^2}}\Big)\bigg|+\bigg|\int_{\D_r}\frac{\nabla' u\cdot \nabla'\phi}{\sqrt{1+|\nabla'u|^2}}\bigg|\nonumber\\
			&\leq \frac{1}{2}\int_{\D_r}|\nabla'u|^2+C(\Hm(M\Delta \Gamma)+ r^{\alpha/2+n-1})\nonumber\\
			&\leq  C(\ex_{\C}(E,x_0,13r,e_n)+r^{\alpha/2})r^{n-1}.
		\end{align}
	completing the proof.
	\end{proof}
	
\section{The Reverse Poincar\'e Inequality}\label{reverse Poincare inequality section}

    In Section \ref{excess and height bound section} we saw that a small excess controls the height of the topological boundary of an almost-minimizer. In this section we show that given a small excess assumption on a cylinder, the flatness of the topological boundary controls the excess on a smaller cylinder. Recall that the \textbf{cylindrical flatness} of a set of locally finite perimeter $E$ at a point $x_0\in\p E$, radius $r>0$, in the direction $\nu\in\SS^{n-1}$ is defined by
	\begin{align}
		\fl(E,x_0,r,\nu)=\frac{1}{r^{n-1}}\inf_{c\in\R}\int_{\C(x_0,r,\nu)\cap\rb E}\frac{|(x-x_0)\cdot\nu-c|^2}{r^2}\dH(x). 
	\end{align}
	This quantity measures how far in an $L^2$ sense the boundary of $E$ is from the best approximating plane with normal $\nu$.

	\begin{theorem}[Reverse Poincar\'e Inequality]\label{reverse poincare} Given $n\geq 2$, there is a positive constant\\ $C_2=C_2(n,\lambda,\Lambda,\kappa,\alpha, r_0, ||A||_{C^\alpha})$ with the following property. If $E$ is a $(\kappa,\alpha)$-almost-minimizer in $\C(x_0,4r,\nu)$ with $x_0\in\p E$, $A(x_0)=I$, $4r<r_0$, and
		\begin{align}
			\ex_{\C} (E,x_0,4r,\nu)\leq \omega(1/8, n,\lambda,\Lambda,\kappa,\alpha,r_0), 
		\end{align}
		where $\omega$ is the constant from Lemma \ref{small-excess position}, then
		\begin{align}\label{actual reverse poincare inequality}
			\ex_{\C}(E,x_0,r,\nu)\leq C_2\big(\fl(E,x_0,2r,\nu)+r^\alpha\big). 
		\end{align}
	\end{theorem}
	
	To prove this we modify the proofs presented in \cite[Chapter 24]{maggi}. First we need several lemmas. Given $\nu\in\SS^{n-1}$ and the decomposition of $\R^n$ into $\R^{n-1}\times\R$, we define the narrow cylinders
	\begin{align}
		\K(z,s)=\D(z,s)\times (-1,1)
	\end{align}
	for $z\in\R^{n-1}$ and $s>0$.
	
	\begin{lemma}[Cone-like competitors {\cite[Lemma 24.8]{maggi}}]\label{cone-like competitors}
	If $s>0$ and $E$ is an open set with smooth boundary in $\R^n$ such that
	\begin{align}
		&|\qq x|<\frac{1}{4}, \qquad\forall x\in \K(z,s)\cap \p E,\nonumber\\
		&\Big\{x\in \K(z,s): \qq x<-\frac{1}{4}\Big\}\subset \K(z,s)\cap E\subset \Big\{x\in \K(z,s): \qq x<\frac{1}{4}\Big\}, 
	\end{align}
	then, for every $t\in (0,1/4)$ and $|c|<1/4$, there exists $I\subset (2/3, 3/4)$ with $|I|\geq 1/24$ such that for every $r\in I$, there exists an open set $F$ of locally finite perimeter in $\R^n$, satisfying,
	\begin{align}
		F\cap \p \K(z,rs)=E\cap \p \K(z,rs), \label{boundaries intersection}
	\end{align}
	\begin{align}
			\Hm(\p F\cap \p \K(z,rs))=\Hm(\p E\cap \p \K(z,rs))=0, \label{boundaries intersection measure zero}
	\end{align}
	\begin{align}
		\K(z,s/2)\cap \p F= \D(z,s/2)\times \{c\}, \label{interior disk}
	\end{align}
	\begin{align}\label{wrpi est}
		P(F; \K(z,rs))&-\Hm(\D(z,rs))\nonumber\\
		&\leq C(n)\Big\{t  \big(P(E; \K(z,s))-\Hm(\D(z,s)\big)+\frac{1}{t}\int_{\K(z,s)\cap \p E}\frac{|\qq x-c|^2}{s^2}\dH(x)\Big\}. 
	\end{align}
	\end{lemma}
	
	\begin{proof}
	    This is proved in \cite[Lemma 24.8]{maggi}, though we point out that \eqref{boundaries intersection measure zero} follows by line (24.29) in \cite[Lemma 24.8]{maggi} and the fact that $F$ is the cone-like extension of $E\cap \p \K(z,rs)$ over the disk $\D(z,{(1-t)rs})\times\{c\}$ (see \cite[Lemma 24.6]{maggi}).
	\end{proof}

	\begin{lemma}[Weak reverse Poincar\'e inequality]\label{weak reverse Poincare} If $E$ is a $(\kappa,\alpha)$-almost-minimizer of $\sF_A$ in $\C_4$, $A(0)=I$, at scale $r_0>4$, such that
		\begin{align}\label{wrp1}
			|\qq x|<\frac{1}{8},\qquad \forall x\in \C_2\cap \p E,
		\end{align}
		\begin{align}\label{wrp2}
			\Big|\Big\{x\in \C_2\setminus E: \qq x<-\frac{1}{8}\Big\}\Big|=\Big| \Big\{x\in \C_2\cap E: \qq x>\frac{1}{8}\Big\}\Big|=0,
		\end{align}
		and if $z\in\R^{n-1}$ and $s>0$ are such that
		\begin{align}
			\K(z,s)\subset \C_2,\qquad\Hm(\rb E\cap\p 
			\K(z,s))=0, 
		\end{align}
		then, for every $|c|<1/4$,
		\begin{align}
		&P(E; \K(z,s/2))-\Hm(\D(z,s/2))\nonumber\\
		&\leq C\bigg[\Big(\big[P(E; \K(z,s))-\Hm(\D(z,s))\big]\int_{\K(z,s)\cap\rb E}\frac{(\qq x -c)^2}{s^2}\dH(x)\Big)^{1/2}+\kappa+||A||_{C^\alpha}\bigg]
		\end{align}
		where $C=C(n,\lambda,\Lambda,\kappa,\alpha,r_0)$.
	\end{lemma}
	
	\begin{proof} Properties \eqref{wrp1} and \eqref{wrp2} imply by the divergence theorem that
		\begin{align}\label{zeta measure}
			\zeta(G)=P(E; \C_2\cap \pp^{-1}(G))-\Hm(G),\qquad G\subset \D_2, 
		\end{align}
		defines a Radon measure on $\R^{n-1}$ concentrated on $\D_2$ as in Lemma \ref{excess measure}. By \cite[Theorem 13.8]{maggi}, given $\epsilon_h\to 0^+$ there exists a sequence $\{E_h\}_{h\in\NN}$ of open sets with smooth boundary such that
		\begin{align}\label{wrpi convergence}
			E_h\loc E,\qquad \Hm\rstr\:\p E_h\wkly \Hm\rstr\: \rb E, \qquad \p E_h\subset I_{\epsilon_h}(\p E), 
		\end{align}
		where $I_{\epsilon_h}(\p E)$ denotes the $\epsilon_h$-neighborhood of $\p E$. The coarea formula and Fatou's lemma give
		\begin{align}
			\int_{2/3}^{3/4}\liminf_{h\to\infty}\Hm(\p\K_{rs}\cap (E^{(1)}\Delta E_h))dr &\leq \liminf_{h\to\infty}\int_{2/3}^{3/4}\Hm(\p\K_{rs}\cap (E^{(1)}\Delta E_h))dr\nonumber\\
			&\leq \lim_{h\to\infty}|(E^{(1)}\Delta E_h)\cap B_s|= 0 .
		\end{align}
		So for a.e. $r\in (2/3,3/4)$, there holds
		\begin{align}
			\liminf_{h\to\infty}\Hm(\p\K_{rs}\cap (E^{(1)}\Delta E_h))=0. 
		\end{align}
		Provided that $h$ is large enough, $E_h\loc E$ and  $\p E_h\subset I_{\epsilon_h}(\p E)$ imply by \eqref{wrp1} and \eqref{wrp2} that
		\begin{align}
				|\qq x|<\frac{1}{4},\qquad \forall x\in \C_2\cap \p E_h, 
		\end{align}
		\begin{align}
			\Big|\Big\{x\in \C_2\setminus E_h: \qq x<-\frac{1}{4}\Big\}\Big|=\Big| \Big\{x\in \C_2\cap E_h: \qq x>\frac{1}{4}\Big\}\Big|=0. 
		\end{align}	
		Given $t\in (0,1/4)$ and $|c|<1/4$ we can apply Lemma \ref{cone-like competitors} to each $E_h$ for $z\in\R^{n-1}$ and $s>0$ to find the sets $I_h\subset(2/3,3/4)$ with $|I_h|\geq 1/24$ such that for each $r\in I_h$ there exists an open set $F_h$ satisfying \eqref{boundaries intersection}, \eqref{boundaries intersection measure zero}, \eqref{interior disk}, and \eqref{wrpi est}. For each $h\in\NN$, we have the containment $\bigcup_{k\geq h} I_k\supset \bigcup_{k\geq h+1} I_k$ and so
		\begin{align}
			\Big|\bigcap_{h\in\N}\bigcup_{k\geq h} I_k\Big|=\lim_{h\to\infty}\Big|\bigcup_{k\geq h} I_k\Big|\geq \frac{1}{24}>0.  
		\end{align}
		It follows that there exists a subsequence $h(k)\to\infty$ as $k\to\infty$ and $r\in (2/3,3/4)$ such that
		\begin{align}
			r\in \bigcap_{k\in\NN} I_{h(k)}, \qquad\lim_{k\to\infty}\Hm(\p\K_{rs}\cap (E^{(1)}\Delta E_{h(k)}))=0,\text{ and}\nonumber
		\end{align}	
		\begin{align}	
			\Hm(\rb E\cap \K(z,rs))=\Hm(\p E_{h(k)}\cap\K(z,rs))=0. 
		\end{align}	
		By Lemma \ref{cone-like competitors} there exist a sequence of open sets $F_k$ of locally finite perimeter in $\R^n$ such that
		\begin{align}
			F_k\cap\p\K_{rs}=E_{h(k)}\cap\p\K_{rs}, 
		\end{align}
		\begin{align}
			\Hm(\p F_k\cap\p\K(z,rs))=\Hm(\p E_{h(k)}\cap\p\K(z,rs))=0, 
		\end{align}
		and 
		\begin{align}
		    &P(F_k; \K(z,rs))-\Hm(\D(z,rs))\nonumber\\
			&\leq  C(n)\Big\{t \big(P(E_{h(k)}; \K(z,s))-\Hm(\D(z,s)\big)+\frac{1}{t}\int_{\K(z,s)\cap \p E_{h(k)}}\frac{|\qq x-c|^2}{s^2}\dH(x)\Big\}. 
		\end{align}
		Now consider the comparison sets
		\begin{align}
			G_k=(F_k\cap\K(z,{rs}))\cup (E\setminus\K(z,rs)). 
		\end{align}	
		Since $E\Delta G_k\subset\subset\C_2$ and 	
		\begin{align}
			\Hm(\p F_k\cap\p\K(z,rs))=\Hm(\rb E\cap\p\K(z,rs))=0, 
		\end{align}
		we have that
		\begin{align}
			P(G_k;\C_2)=P(F_k;\K(z,rs))+P(E;\C_2\setminus\overline{\K(z,rs)})+\Hm(\p\K(z,rs)\cap(E^{(1)}\Delta F_k)). 
		\end{align}
		By Proposition \ref{perturbation lemma} we have 
		\begin{align}
			P(E;\C_2)\leq P(G_k;\C_2)+C(\kappa+||A||_{C^\alpha}) 
		\end{align}
		for some $C=C(n,\lambda,\Lambda,\kappa,\alpha,r_0)$. Hence
		\begin{align}
			P(E;\K(z,rs))\leq P(F_k; \K(z,rs))+\Hm(\p\K(z,rs)\cap(E^{(1)}\Delta F_k))+C(\kappa+||A||_{C^\alpha}). 
		\end{align}
		It follows that
		\begin{align}
			 & P(E;\K(z,rs))-\Hm(\D(z,rs))\nonumber
			\\ &\leq C(n)\Big\{t \big(P(E_{h(k)}; \K(z,s))-\Hm(\D(z,s)\big)+\frac{1}{t}\int_{\K(z,s)\cap \p E_{h(k)}}\frac{|\qq x-c|^2}{s^2}\dH(x)\Big\}
			\nonumber\\
			& \qquad+ \Hm(\p\K(z,rs)\cap(E^{(1)}\Delta F_k))+C(\kappa+||A||_{C^\alpha}). 
		\end{align}
		Taking the limit as $k\to \infty$, using the weak convergence of \eqref{wrpi convergence} since $\Hm(\p E\cap\K(z,s))=0$, and $\lim_{k\to\infty}\Hm(\p\K_{rs}\cap (E^{(1)}\Delta E_{h(k)}))=0$, we have
		\begin{align}
			& P(E;\K(z,rs))-\Hm(\D(z,rs))\nonumber\\
			&\leq  C(n)\Big\{t  \big(P(E; \K(z,s))-\Hm(\D(z,s)\big)
			 +\frac{1}{t}\int_{\K(z,s)\cap \rb E}\frac{|\qq x-c|^2}{s^2}\dH(x)\Big\}\nonumber\\
			 &\qquad+C(\kappa+||A||_{C^\alpha}). 
		\end{align}
		Hence 
		\begin{align}
			\zeta(\K(z,s/2)) &\leq \zeta(\K(z,rs))\nonumber\\ &\leq  C\Big\{t \zeta(\K(z,s)+\frac{1}{t}\int_{\K(z,s)\cap \rb E}\frac{|\qq x-c|^2}{s^2}\dH(x)+\kappa+||A||_{C^\alpha}\Big\} 
		\end{align}
		By \eqref{zeta measure}, $\zeta(\K(z,s/2))\leq \zeta(\K(z,s))$ and so this inequality also holds for $t>1/4$ provided we take $C=C(n,\lambda,\Lambda,\kappa,\alpha,r_0)\geq 4$. Hence it holds for all $t \in(0,\infty)$. Minimizing the right hand side over all $t$ yields 
		\begin{align}
		\zeta(\K(z,s/2))\leq  C\Big\{\Big(\zeta(\K(z,s))\int_{\K(z,s)\cap \rb E}\frac{|\qq x-c|^2}{s^2}\dH\Big)^{1/2}+\kappa+||A||_{C^\alpha}\Big\} 
		\end{align}
		as desired.
	\end{proof}

	\begin{proof}[Proof of Theorem \ref{reverse poincare}]
	        
	   By the scaling given in Proposition \ref{scaling of the anisotropic energy}, we have that $E_{x_0,r}$ is a $(\kappa r^\alpha,\alpha)$-almost-minimizer of $\sF_{A_{x_0,r}}$ in $\C_4=\C(0,4,\nu)$ at scale $r_0/r$ with $0\in \p E_{x_0,r}$, $A_{x_0,r}(0)=I$, $||A_{x_0,r}||_{C^\alpha}=||A||_{C^\alpha}r^\alpha$, and $4<r_0/r$. Thus to prove \eqref{actual reverse poincare inequality}, we may assume $\ex_{\C}(E_{x_0,r},0,4,\nu)=\ex_{\C}(E,x_0,4r,\nu)\leq \omega$
	   and show
	   \begin{align}\label{simplified poincare inequality}
			\ex_{\C}(E_{x_0,r},0,1,\nu)\leq C\big(\fl(E_{x_0,r},0,2,\nu)+\kappa r^\alpha +||A_{x_0,r}||_{C^\alpha}\big).
	   \end{align}
	   By Proposition \ref{small-excess position} and Proposition \ref{excess measure}, it follows that
		\begin{align}
			|\qq x|<\frac{1}{8},\qquad \forall x\in \C_2\cap \p E_{x_0,r}, 
		\end{align}
		\begin{align}
			\Big|\Big\{x\in \C_2\setminus E_{x_0,r}: \qq x<-\frac{1}{8}\Big\}\Big|=\Big| \Big\{x\in \C_2\cap E_{x_0,r}: \qq x>\frac{1}{8}\Big\}\Big|=0, 
		\end{align}
		and
		\begin{align}
			\Hm(G)=\int_{\C_2\cap \rb E_{x_0,r}\cap \pp^{-1}(G)} (\nu_E\cdot \nu)\dH,\qquad\forall\ G\subset \D_2. 
		\end{align}
		Hence $\ex_{\C}(E_{x_0,r},0,1,\nu)=P(E_{x_0,r}; \C_1)-\Hm(\D_1)$ and so it suffices to show that for every $c\in\R$,
		\begin{align}\label{simplified poincare inequality2}
			P(E_{x_0,r}; \C_1)-\Hm(\D_1)\leq C\Big\{\int_{\C_2\cap\p E_{x_0,r}} |\qq x-c|^2\dH+\kappa r^\alpha+||A_{x_0,r}||_{C^\alpha}\Big\}. 
		\end{align}
		If $|c|>1/4$, then $|\qq x-c|\geq 1/8$ and so
		\begin{align}
			\int_{\C_2\cap\p E_{x_0,r}} |\qq x-c|^2\dH\geq \frac{P(E_{x_0,r};\C_1)}{8^2} 
		\end{align}
		and we are done provided we take $C\geq 64$. Thus we are left with the case $|c|<1/4$. Set
		\begin{align}
			\zeta(G)=P(E_{x_0,r}; \C_2\cap \pp^{-1}(G))-\Hm(G),\qquad\text{ for } G\subset \D_2, 
		\end{align}
		which defines a Radon measure on $\R^{n-1}$, concentrated on $\D_2$. We apply Lemma \ref{weak reverse Poincare} in every cylinder $\K(z,s)$ with $z\in\R^{n-1}$ and $s>0$ such that
		\begin{align}
			\D(z,2s)\subset \D_2,\qquad\Hm(\rb E_{x_0,r}\cap\p \K(z,2s))=0,
		\end{align}
		to find that
		\begin{align}\label{poincare cone ineq1}
			\zeta(\D(z,s))\leq C\Big\{\Big(\zeta(\D(z,2s)\inf_{c<1/4}\int_{\K(z,2s)\cap\rb E_{x_0,r}}\frac{|\qq x-c|^2}{(2s)^2}\dH\Big)^{1/2}+\kappa r^\alpha +||A_{x_0,r}||_{C^\alpha}\Big\}.
		\end{align}
		An approximation argument, setting 
		\begin{align}
			h=\inf_{|c|<1/4}\int_{\C_2\cap\rb E_{x_0,r}}|\qq x-c|^2\dH, 
		\end{align}
		for brevity, implies by \eqref{poincare cone ineq1} that
		\begin{align}\label{poincare cone ineq2}
			s^2\zeta(\D(z,s))\leq C\Big(\sqrt{s^2\zeta(\D(z,2s))h}+\kappa r^\alpha+||A_{x_0,r}||_{C^\alpha}\Big)
		\end{align}
		whenever $\D(z,2s)\subset \D_2$ since $s<1$. We now use a covering argument to complete the proof. Let
		\begin{align}
			Q=\sup\{s^2\zeta(\D(z,s)): \D(z,2s)\subset \D_2\}, 
		\end{align}
		and notice $Q<\infty$ since for every $\D(z,2s)\subset \D_2$, we have
		\begin{align}
			s^2\zeta(\D(z,s))\leq \zeta(\D_2)\leq P(E_{x_0,r}; \C_2)<\infty. 
		\end{align}
		Given $\D(z,2s)\subset \D_2$, cover $\D(z,s)$ by finite many balls $\{\D(z_k, s/4)\}_{k=1}^N$ with centers $z_k\in \D(z,s)$. This can be done with a bounded number of balls depending only on the dimension $n$, that is, $N\leq N(n)$. So by the subadditivity of the measure $\zeta$, \eqref{poincare cone ineq2}, and the definition of $Q$, we have
		\begin{align}
			s^2\zeta(\D(z,s))&\leq 16\sum_{k=1}^N\Big(\frac{s}{4}\Big)^2\zeta\Big(\D\Big(z_k,\frac{s}{4}\Big)\Big)\nonumber\\
			&\leq C\sum_{k=1}^N\Big(\sqrt{\Big(\frac{s}{2}\Big)^2\zeta\Big(\D\Big(z_k,\frac{s}{2}\Big)\Big)h}+\kappa r^\alpha+||A_{x_0,r}||_{C^\alpha}\Big)\nonumber\\
			&\leq CN(n)\Big(\sqrt{Qh}+\kappa r^\alpha+||A_{x_0,r}||_{C^\alpha}\Big) 
		\end{align}
		where we used that $\D(z_k,s/4)\subset \D(z,2s)\subset \D_2$. Hence $Q\leq C(\sqrt{Qh}+\kappa r^\alpha+||A_{x_0,r}||_{C^\alpha})$
	    for some $C=C(n,\lambda,\Lambda,\kappa,\alpha,r_0)$. By Cauchy-Schwarz, we have $C\sqrt{Qh}=\sqrt{Q C^2h}\leq \frac{1}{2}Q+\frac{1}{2}C^2 h$. Combining these gives $Q\leq \frac{1}{2}Q+C(h +\kappa r^\alpha+||A_{x_0,r}||_{C^\alpha})$. Thus $\zeta(\D_1)\leq Q\leq C(h+\kappa r^\alpha+||A_{x_0,r}||_{C^\alpha})$	for some $C=C(n,\lambda,\Lambda,\kappa,\alpha,r_0)$. Recalling the definitions of $\zeta(\D_1)$ and $h$, we see that this completes the proof of \eqref{simplified poincare inequality2}.
	\end{proof}
	
	\section{Tilt-Excess Decay}\label{tilt excess decay section}
		
		We showed in Section \ref{Lipschitz approximation theorem section} that almost-minimizers can be approximated by `almost-harmonic' Lipschitz functions at points with small excess. Now we approximate these Lipschitz with harmonic functions which allow us to find new directions for which the excess experiences quadratic decay. 
		
		First we recall a couple lemmas about harmonic functions. These are just the rescaled versions of \cite[Lemma 25.1, Lemma 25.2]{maggi}. Note that 
		\begin{align}
		    \fint_{\D_s}=\frac{1}{\omega_{n-1}s^{n-1}}\int_{\D_s}
		\end{align}
		denotes the integral average.

    \begin{lemma}\label{harmonic decay} There is a positive constant $C(n)$ with the following property. If $v\colon\R^{n-1}\to\R$ is harmonic in $\D_r$ and $w\colon\R^{n-1}\to\R$ is defined by $w(z)=v(0)+\nabla v(0)\cdot z$, then
    	\begin{align}
    		\sup_{\D_{\theta r}}\frac{|v-w|}{r}\leq C(n)\theta^2 \bigg(\fint_{\D_r} |\nabla'v|^2\bigg)^{1/2}
    	\end{align}
    	for every $\theta\in (0,1/2]$. In particular, 
    	\begin{align}
    		\fint_{\D_{\theta r}}\bigg(\frac{|v-w|}{\theta r}\bigg)^2\leq C(n)\theta ^2\fint_{\D_r}|\nabla'v|^2.
    	\end{align}
    \end{lemma}

    \begin{lemma}[Harmonic approximation]\label{harmonic approximation} For every $\tau>0$ there exists $\sigma>0$ with the following property. If $u\in W^{1,2}(\D_r)$ is such that
    \begin{align}
    	\fint_{\D_r} |\nabla' u|^2\leq 1,\qquad \bigg|\fint_{\D_r}\nabla' u\cdot\nabla'\phi\bigg|\leq\sup_{\D_r}|\nabla'\phi|\: \sigma\qquad\forall\phi\in C_c^\infty(\D_r),
    \end{align}
     then there exists a harmonic function $v$ on $\D_r$ such that
    \begin{align}
    	\fint_{\D_r} |\nabla' v|^2\leq 1,\qquad\text{ and }\qquad\fint_{\D_r} |v-u|^2\leq \tau r^2.
    \end{align}
    \end{lemma}

    We now prove the excess improvement by tilting. This states that if the excess is small enough in a given direction, then there is a nearby direction in which the excess at a definite smaller scale sees quadratic decay with the error term seeing $\alpha$th power decay. 
    Note in the theorem below, the fraction $1/104$ comes from the rough bound $13\cdot 4\cdot \sqrt 2\leq 13\cdot 4\cdot 2= 104$ where the $13$ comes from the Lipschitz approximation theorem, the $4$ comes from small excess assumption in the reverse Poincar\'e inequality, and the $\sqrt 2$ comes from containing one cylinder inside of another cylinder that is tilted in a different direction.

    \begin{theorem}[Excess improvement by tilting]\label{excess-improvement by tilting} Given $\theta\in (0,1/104]$, there exist positive constants $\epsilon_2=\epsilon_2(n,\lambda,\Lambda,\kappa,\alpha,r_0,||A||_{C^\alpha},\theta)$ and $C_3=C_3(n,\lambda,\Lambda,\kappa,\alpha,r_0,||A||_{C^\alpha})$ with the following property. If $E$ is a $(\kappa,\alpha)$-almost-minimizer of $\sF_A$ in $\C(x_0,r_0,\nu_0)$, $x_0\in\p E$, $A(x_0)=I$, and $r<r_0$ with
	\begin{align}
		\ex_{\C}(E,x_0,r,\nu_0)+ r^{\alpha/2}\leq \epsilon_2,
	\end{align}
	then there exists $\nu_1\in\SS^{n-1}$ such that
	\begin{align}
		\ex_{\C}(E,x_0,\theta r,\nu_1)\leq  C_3(\theta^2\ex_{\C}(E,x_0,r,\nu_0)+\theta^\alpha r^{\alpha/2})
	\end{align}
    \end{theorem} 

\begin{proof} Assuming without loss of generality that $x_0=0$ and $\nu_0=e_n$, it suffices to prove that given $\theta\in (0,1/8]$, there exist positive constants $\epsilon_2=\epsilon_2(n, \lambda,\Lambda,\kappa,\alpha,r_0,||A||_{C^\alpha},\theta)$ and $C_3=C_3(n, \lambda,\Lambda,\kappa,\alpha,r_0,||A||_{C^\alpha})$ with the following property. If $E$ is a $(\kappa,\alpha)$-almost-minimizer of $\sF_A$ in $\C(0,r_0,e_n)$, $A(0)=I$, $0\in\p E$, and $13r<r_0$ with
	\begin{align}
		\ex_{\C}(E,0,13r,e_n)+ r^{\alpha/2}\leq \epsilon_2,
	\end{align}
	then there exists $\nu_1\in\SS^{n-1}$ such that
	\begin{align}
		\ex_{\C}(E,0,\theta r,\nu_1)\leq C_3(\theta^2\ex_{\C}(E,0,13r,e_n)+\theta^\alpha r^{\alpha/2})
	\end{align}
    We set $\C_s=\C(0, s,e_n)$ for brevity. 
    
    We shall select a number of criteria for $\epsilon_2$ to satisfy which together give the desired result. We place a \boxed{box} around each of these choices to make it easy for the reader to check that all of these choices are consistent.
    
    Choose $\epsilon_2$ to satisfy
	\begin{align}
	\boxed{\epsilon_2\leq \epsilon_1}
	\end{align}
	where $\epsilon_1$ is from the Lipschitz approximation theorem. Then $\ex_{\C}(E,0,13r,e_n)\leq \epsilon_1$ and thus there is a Lipschitz function $u\colon\R^{n-1}\to\R$ such that $\Lip u\leq 1$ such that
	\begin{align}\label{la1}
		\sup_{\D_r}\frac{|u(z)|}{r}\leq C_1\ex_{\C}(E,0,13r,e_n)^{1/(2(n-1))},
	\end{align}
	\begin{align}\label{la2}
		\frac{\Hm(M\Delta \Gamma)}{r^{n-1}}\leq C_1\ex_{\C}(E,0,13r,e_n),
	\end{align}
	\begin{align}\label{la3}
		\fint_{D_r} |\nabla' u|^2\leq C_1 \ex_{\C}(E,0,13r,e_n), \text{ and }
	\end{align}
	\begin{align}\label{la4}
 		\bigg|\fint_{\D_r} \nabla' u\cdot \nabla' \phi\bigg|\leq C_1\sup_{\D_r}|\nabla'\phi| \big(\ex_{\C}(E,0,13r,e_n)+r^{\alpha/2})\qquad\text{ for all }\phi\in C_c^1(\D_r).
	\end{align}
	where $C_1$ is the constant from the Lipschitz approximation theorem, $M=\C_r\cap\p E$, and $\Gamma$ is the graph of $u$. 
	Choose $\epsilon_2$ to also satisfy
	\begin{align}
		\boxed{C_1 \epsilon_2\leq 1}.
	\end{align}
	Then $C_1(\ex_{\C}(E,0,13r,e_n)+ r^{\alpha/2})\leq 1$ and so setting 
	\begin{align}
		\beta=C_1(\ex_{\C}(E,0,13r,e_n)+ r^{\alpha/2})\qquad\text{and}\qquad u_0=u/\sqrt\beta,
	\end{align}
	we have
	\begin{align}
		\fint_{\D_r} |\nabla' u_0|^2\leq 1\qquad\text{and}\qquad 
 	 \bigg|\fint_{\D_r} \nabla' u_0\cdot \nabla' \phi\bigg|\leq \sup_{\D_r}|\nabla'\phi|\sqrt\beta\qquad\text{ for all }\phi\in C_c^1(\D_r).
	\end{align}
	By Lemma \ref{harmonic approximation}, for every $\tau>0$ there is $\sigma(\tau)>0$ such that if 
	\begin{align}\label{sigma ineq}
		\sqrt\beta\leq \sigma(\tau)
	\end{align}
	then there is $v_0:\R^{n-1}\to\R$ which is harmonic in $\D_r$ such that
	\begin{align}
		\fint_{\D_r} |u_0-v_0|^2\leq \tau r^2\qquad\text{and}\qquad \fint_{\D_r} |\nabla'v_0|^2\leq 1
	\end{align}
	Setting $v=\sqrt\beta\: v_0$, we have that $v$ is harmonic in $\D_r$ and
	\begin{align}\label{ha1}
		\fint_{\D_r}|u-v|^2\leq \tau r^2\beta\qquad\text{and}\qquad \fint_{\D_r}|\nabla'v|^2\leq \beta.
	\end{align}
	Since $4\theta\leq 1/2$, setting $w(z)=v(0)+\nabla'v(0)\cdot z$ for $z\in \D_r$, we see by Lemma \ref{harmonic decay} that
	\begin{align}
		\fint_{\D_{4\theta r}}\frac{|v-w|^2}{(\theta r)^2}\leq C(n)\theta^2\fint_{\D_r} |\nabla'v|^2\leq C(n)\theta^2\beta.
	\end{align}
	By \eqref{ha1} and $\D_{4\theta r}\subset\D_r$,
	\begin{align}
		\int_{\D_{4\theta r}}\frac{|u-v|^2}{(\theta r)^2}\leq \int_{\D_r}\frac{|u-v|^2}{(\theta r)^2}\leq\frac{\tau}{\theta^2}\beta r^{n-1}
	\end{align}
	and so
	\begin{align}
		\fint_{\D_{4\theta r}}\frac{|u-v|^2}{(\theta r)^2}\leq C(n)\frac{\tau}{\theta^{n+1}}\beta.
	\end{align}	
	Noting $|u-w|^2\leq 2|u-v|^2+2|v-w|^2$, we have that
	\begin{align}
		\fint_{\D_{4\theta r}}\frac{|u-w|^2}{(\theta r)^2}\leq C(n)\Big(\frac{\tau}{\theta^{n+1}}+\theta^2\Big)\beta.
	\end{align}
	We apply the above with $\tau=\theta^{n+3}$ and choose $\epsilon_2$ to also satisfy
	\begin{align}
		\boxed{\sqrt{C_1 \epsilon_2}\leq \sigma(\theta^{n+3})}
	\end{align}
	with $\sigma(\:\cdot\:)$ as in \eqref{sigma ineq}. Then $\sqrt{\beta}\leq \sigma(\theta^{n+3})$ and so
	\begin{align}\label{ha3}
		\int_{\D_{4\theta r}}|u-w|^2\leq  C(n)\theta^{n+3}\beta r^{n+1}.
	\end{align}
	Now set
	\begin{align}\label{ha direction and height}
		\nu_1=\frac{(-\nabla'v(0),1)}{\sqrt{1+|\nabla'v(0)|^2}}\in\SS^{n-1},\qquad c_1=-\frac{v(0)}{\sqrt{1+|\nabla'v(0)|^2}}\in\R
	\end{align}
	and let's estimate $\fl(E,0,2\theta r,\nu_1)$. Since $\C(0,2\theta r,\nu_1)\subset \C_{4\theta r}$, we have that
	\begin{align}
		\fl(E,0,2\theta r,\nu_1)&=\frac{1}{(2\theta r)^{n+1}}\inf_{c\in\R} \int_{\C(0,2\theta r,\nu_1)\cap\rb E} |x\cdot\nu_1-c|^2 \dH(x)\nonumber\\	
		&\leq \frac{C(n)}{(\theta r)^{n+1}}\int_{\C_{4\theta r}\cap\rb E} |x\cdot \nu_1-c_1|^2 \dH(x).
	\end{align}
	This last integral we split in terms of $M\cap\Gamma$ and $M\setminus\Gamma$. 
	
	For $M\cap\Gamma$, by $\Lip(u)\leq 1$, \eqref{ha direction and height}, and \eqref{ha3}, we have
	\begin{align}
	\int_{\C_{4\theta r}\cap M\cap\Gamma} |x\cdot \nu_1-c_1|^2 \dH(x)&=\int_{\D_{4\theta r}\cap\pp(M\cap \Gamma)} |(z,u(z))\cdot\nu_1-c_1|^2\sqrt{1+|\nabla'u(z)|^2} dz\nonumber\\
		&\leq  \sqrt{2}\int_{\D_{4\theta r}\cap\pp(M\cap \Gamma)} \frac{|u-w|^2}{1+|\nabla' v(0)|^2}\nonumber\\
		&\leq \sqrt{2}\int_{\D_{4\theta r}\cap\pp(M\cap\Gamma)} |u-w|^2	\nonumber\\
		&\leq  C(n)\theta^{n+3}\beta r^{n+1}.
	\end{align}

     For $M\setminus\Gamma$, observe that
	\begin{align}
		\int_{\C_{4\theta r}\cap(M\setminus\Gamma)} |x\cdot \nu_1-c_1|^2 \dH(x)&= \int_{\C_{4\theta r}\cap(M\setminus\Gamma)}\frac{|\qq x+v(0)-\pp x\cdot\nabla'v(0)|^2}{1+|\nabla' v(0)|^2} \dH(x)\nonumber\\
		&\leq \int_{\C_{4\theta r}\cap(M\setminus\Gamma)}|\qq x+v(0)-\pp x\cdot\nabla'v(0)|^2\dH(x)\nonumber\\
			&\leq  3\Hm(M\setminus\Gamma) (\sup_{x\in M}|\qq x|
			^2+|v(0)|^2+\sup_{x\in M}|\pp x|^2|\nabla'v(0)|^2).
	\end{align}
	By the height bound, we have
	\begin{align}
		\sup_{x\in M}|\qq x|^2\leq C_0^2\ex_{\C}(E,0,4r,e_n)^{1/(n-1)}r^2\leq C_0^2 (13/4)\ex_{\C}(E,0,13r,e_n)^{1/(n-1)}\leq C\beta^{1/(n-1)}r^2.
	\end{align}
	Also, $\sup_{x\in M}|\pp x|^2\leq r^2$.	Since $v$ is harmonic,
	\begin{align}
	|v(0)|^2\leq C(n)\fint_{\D_r}|v|^2\qquad\text{and}\qquad|\nabla'v(0)|^2\leq 
	\frac{C(n)}{r^2}\fint_{\D_r} |v|^2.
	\end{align}
	By \eqref{ha1} and $\sup_{\D_r} |u|^2\leq C_1^2\beta^{1/(n-1)} r^2$ from \eqref{la1}, it follows that
	\begin{align}
		|v(0)|^2+\sup_{x\in M}|\pp x|^2 |\nabla'v(0)|^2 &\leq C(n)\fint_{\D_r}|v|^2\leq C(n)\Big(\fint_{\D_r} |u-v|^2+\fint_{\D_r} |u|^2\Big)\nonumber\\
		&\leq C(\theta^{n+3}\beta+\beta^{1/(n-1)})r^2.
	\end{align}
    Since $\Hm(M\setminus\Gamma)\leq\Hm(M\Delta \Gamma)\leq\beta r^{n-1}$, we have
	\begin{align}
		\int_{\C_{4\theta r}\cap(M\setminus\Gamma)} |x\cdot \nu_1-c_1|^2 \dH(x)
			\leq  C\beta r^{n-1}(\theta^{n+3}\beta+\beta^{1/(n-1)})r^2.
	\end{align}
	Choose $\epsilon_2$ to also satisfy
	\begin{align}
		\boxed{\epsilon_2^{1/(n-1)}\leq \theta^{n+3}}.
	\end{align}
	Then
	\begin{align}
		\beta^{1/(n-1)}\leq\theta^{n+3}.
	\end{align}
	which gives
	\begin{align}
		\int_{\C_{4\theta r}\cap(M\setminus\Gamma)} |x\cdot \nu_1-c_1|^2 \dH(x)\leq C \theta^{n+3}\beta r^{n+1}.
	\end{align}
	Combining these estimates we have
	\begin{align}
		\fl(E,0,2\theta r,\nu_1)
		&\leq \frac{C(n)}{(\theta r)^{n+1}}\bigg(\int_{\C_{4\theta r}\cap M\cap\Gamma} |x\cdot \nu_1-c_1|^2\dH+\int_{C_{4\theta r}\cap (M\setminus\Gamma)} |x\cdot \nu_1-c_1|^2\dH\bigg) \nonumber\\
		&\leq C \theta^2\beta.
	\end{align}
	
	Next, we show that provided $\epsilon_2$ is suitably small,
	then
	\begin{align}
		\ex_{\C}(E,0,4\theta r,\nu_1)\leq\omega(1/8,n,\lambda,\Lambda,\kappa,\alpha,r_0).
	\end{align}
	By Proposition \ref{excess at different scales} and Proposition \ref{excess and changes of direction}, we have 
	\begin{align}
		\ex_{\C}(E,0,4\theta r,\nu_1)\leq \Big(\frac{13r}{4\theta r}\Big)^{n-1}\ex_{\C}(E,0,13r,\nu_1)\leq \tilde C\big(\ex_{\C}(E,0,13r,e_n)+|e_n-\nu_1|^2\big).
	\end{align}
	where $\tilde C=\tilde C(n,\lambda,\Lambda,\kappa,\alpha, r_0,\theta)$
    Additionally,
	\begin{align}
		|e_n-\nu_1|^2 &= \bigg|(0,1)-\frac{(-\nabla'v(0),1)}{\sqrt{1+|\nabla'v(0)|^2}}\bigg|^2\nonumber\\
			&=\frac{|\nabla'v(0)|^2+(\sqrt{1+|\nabla'v(0)|^2}-1)^2}{1+|\nabla'v(0)|^2} \nonumber\\
			&\leq 2 |\nabla'v(0)|^2\leq \frac{C(n)}{r^2}\fint_{\D_r}|v|^2\nonumber\\
	         &\leq\frac{C(n)}{r^2}\Big(\fint_{\D_r} |u-v|^2+\fint_{\D_r} |u|^2\Big)\nonumber\\
		&\leq C(\theta^{n+3}\beta+\beta^{1/(n-1)})\leq  C\beta^{1/(n-1)}
	\end{align}
	where the last several inequalities follow as above. Hence 	
	\begin{align}
		\ex_{\C}(E,0,4\theta r,\nu_1)\leq  \tilde C\beta^{1/(n-1)}.\label{constantbeta}
	\end{align}
	for some $\tilde C=\tilde C(n,\lambda,\Lambda,\kappa,\alpha, r_0,\theta)$. We choose $\epsilon_2$ to also satisfy 
	\begin{align}
		\boxed{\tilde C\epsilon_2^{1/(n-1)}\leq \omega(1/8,n,\lambda,\Lambda,\kappa,\alpha,r_0)}
	\end{align}
	so that 
	\begin{align}
		\ex_{\C}(E,0,4\theta r,\nu_1)\leq  \omega(1/8,n,\lambda,\Lambda,\kappa,\alpha,r_0)
	\end{align}
	since $\beta\leq \epsilon_2$. The reverse Poincar\'e inequality, Theorem \ref{reverse poincare}, implies that
	\begin{align}
		\ex_{\C}(E,0,\theta r,\nu_1)&\leq C_2(\fl(E,0,2\theta r,\nu_1)+(\theta r)^\alpha)\nonumber\\
				&\leq C(\theta^2\beta+ \theta^\alpha r^\alpha)\nonumber\\
				&\leq C(\theta^2\ex_{\C }(E,0,13r,e_n)+\theta^2 r^{\alpha/2}+ \theta^\alpha r^\alpha)\nonumber\\
			&\leq C(\theta^2\ex_{\C }(E,0,13r,e_n)+\theta^\alpha r^{\alpha/2})
	\end{align}
	as desired.
\end{proof}

\section{Regularity of Almost-Minimizers}\label{regularity section}
	
    We are almost in the position to prove our main regularity result. All we first need is to prove the following lemma which allows us to remove the assumption $A(x_0)=I$ and obtain an excess-decay estimate which we will iterate in the proof of our regularity theorem.
	
    \begin{lemma}\label{regularity lemma} For each $\beta\in(0,\alpha/4]$, there exist positive constants $\theta_1=\theta_1(n,\lambda,\Lambda,\kappa,\alpha,r_0,||A||_{C^\alpha},\beta)<1$, $\epsilon_3=\epsilon_3(n,\lambda,\Lambda,\kappa,\alpha,r_0,||A||_{C^\alpha},\beta)$, and $C_4=C_4(n,\lambda,\Lambda,\kappa,\alpha,r_0,||A||_{C^\alpha},\beta)$ with the following property. Let $E$ be a $(\kappa,\alpha)$-almost-minimizer of $\sF_A$ in $\C(x_0,r,\nu_0)$ with $r<r_0$ and $x_0\in\p E$, and set
    \begin{align}
    	\ex_{\C}^*(E,x,s,\nu)=\max\Big\{\ex_{\C}(E,x,s,\nu),\frac{s^{\alpha/2}}{\theta_1^{n-1+2\beta}}\Big\},\qquad\text{ for } x\in\R^n,\ s>0,\ \nu\in\SS^{n-1}.
    \end{align}
    If $r<r_0$ and 
    \begin{align}
        \ex_{\C}^*(E,x_0,r,\nu_0)\leq \epsilon_3,
    \end{align}
	then there exists $\nu_1\in\SS^{n-1}$ such that
	\begin{align}
		\ex_{\C}^*(E,x_0,\theta_1 r,\nu_1) & \leq \theta_1^{2\beta}\ex_{\C}^*(E,x_0,r,\nu_0),\text{ and }\label{tilt-decay}\\
		|\nu_1-\nu_0|^2 & \leq C_4\ex_{\C}^*(E,x_0,r,\nu_0).\label{tilt-control}
	\end{align}
\end{lemma}

\begin{proof}
We will eventually make choices for positive constants $\tilde\theta=\tilde\theta(n,\lambda,\Lambda,\kappa,\alpha,r_0,||A||_{C^\alpha},\beta)<1$ and $\tilde C=\tilde C(n,\lambda,\Lambda,\kappa,\alpha,r_0,||A||_{C^\alpha},\beta)$ show that  \eqref{tilt-decay} holds if we set
\begin{align}
	\theta_1=\Big(\frac{\lambda}{4\Lambda}\Big)^{1/2}\tilde\theta\qquad\text{and}\qquad \epsilon_3={\tilde C}^{-1}\epsilon_2
\end{align}
where $\epsilon_2$ is the constant from Proposition \ref{excess-improvement by tilting} applied with $\theta=\tilde\theta$.

Since $2\beta\leq\alpha/2$ and $\theta_1<1$, we have
	\begin{align}
		\frac{(\theta_1 r)^{\alpha/2}}{\theta_1^{n-1+2\beta}}\leq \theta_1^{\alpha/2}\ex_{\C}^*(E,x_0,r,\nu_0)\leq \theta_1^{2\beta}\ex_{\C}^*(E,x_0,r,\nu_0).
	\end{align}
	Consequently, we only need to show the existence of $\nu_1\in\SS^{n-1}$ such that
	\begin{align}
				\ex_{\C}(E,x_0,\theta_1 r,\nu_1)\leq \theta_1^{2\beta}\ex_{\C}^*(E,x_0,r,\nu_0).
	\end{align}
	If $\ex_{\C}(E,x_0,r,\nu_0)\leq r^{\alpha/2}$, then by Proposition \ref{excess at different scales} we have
	\begin{align}
		\ex_{\C}(E,x_0,\theta_1 r,\nu_0)\leq \frac{1}{\theta_1^{n-1}}\ex_{\C}(E,x_0,r,\nu_0)\leq \theta_1^{2\beta}\frac{ r^{\alpha/2}}{\theta_1^{n-1+2\beta}}\leq \theta_1^{2\beta}\ex_{\C}^*(E,x_0,r,\nu_0)
	\end{align}
	and so we can take $\nu_1=\nu_0$. Otherwise, $\ex_{\C}(E,x_0,r,\nu_0)\geq r^{\alpha/2}$. We will proceed by applying Proposition \ref{excess-improvement by tilting}, but we need to use the change of variable $T_{x_0}$ since we are not assuming that $A(x_0)$ equals the $I$. This enables us to work with the set $E_{x_0}$ which is an almost-minimizer of $\sF_{A_{x_0}}$ with $A_{x_0}(x_0)=I$. Let $\widetilde\nu_0$ denote the image of $\nu_0$ under this change of variable, that is,
	\begin{align}
		\tilde \nu_0=\frac{A^{1/2}(x_0) \nu_0}{|A^{1/2}(x_0) \nu_0|}.
    \end{align}
	First note that
	\begin{align}\label{containment}
		\C\big(x_0,\frac{r}{(2\Lambda)^{1/2}},\widetilde\nu_0\big)\subset \B\big(x_0,\frac{r}{\Lambda^{1/2}}\big)\text{\ \ and \  }\W_{x_0}\big(x_0,\frac{r}{\Lambda^{1/2}}\big)\subset \B(x_0, r)\subset \C(x_0, r,\nu_0).
	\end{align}
	Then $E_{x_0}$ is an almost-minimizer of $\sF_{A_{x_0}}$ in $\C\big(x_0,\frac{r}{(2\Lambda)^{1/2}},\widetilde \nu_0\big)$ by Proposition \ref{invariance} since
	\begin{align}
		\C\big(x_0,\frac{r}{(2\Lambda)^{1/2}},\widetilde \nu_0\big)\subset T_{x_0}\big(\W_{x_0}\big(x_0,\frac{r}{\Lambda^{1/2}}\big)\big)\subset T_{x_0}\big(\C(x_0,r,\nu_0)\big).
	\end{align}
	It also follows by \eqref{containment} and Proposition \ref{comparability of excess} that
	\begin{align}
		\ex_{\C}\big(E_{x_0},x_0,\frac{r}{(2\Lambda)^{1/2}},\widetilde\nu_0\big)&\leq 2^{(n-1)/2}\ex_{\B}\big(E_{x_0},x_0,\frac{r}{\Lambda^{1/2}},\widetilde\nu_0\big)\nonumber\\
		&\leq 2^{(n-1)/2}C\ex_{\W}\big(E,x_0,\frac{r}{\Lambda^{1/2}},\nu_0\big)\nonumber\\
		&\leq C \ex_{\C}(E,x_0,r,\nu_0).
	\end{align}
	Hence by our assumption $\ex_{\C}(E,x_0,r,\nu_0)\geq r^{\alpha/2}$
	\begin{align}
		\ex_{\C}\big(E_{x_0},x_0,\frac{r}{(2\Lambda)^{1/2}},\widetilde\nu_0\big)+\frac{r^{\alpha/2}}{(2\Lambda)^{\alpha/4}} &\leq  (C+(2\Lambda)^{-\alpha/4})\ex_{\C}(E,x_0,r,\nu_0)
		\nonumber\\
		&\leq C\ex_{\C}(E,x_0,r,\nu_0)\nonumber\\
		&\leq \tilde C\epsilon_3\leq\epsilon_2
	\end{align}
	where at this step we make our choice for $\tilde C=C$. Thus Proposition \ref{excess-improvement by tilting} applies to $E_{x_0}$ with radius $r/(2\Lambda)^{1/2}$ and so there is $\widetilde\nu_1\in\SS^{n-1}$ such that
	\begin{align}
		\ex_{\C}\big(E_{x_0},x_0,\frac{\tilde\theta r}{(2\Lambda)^{1/2}},\widetilde\nu_1\big)&\leq C_3\Big(\tilde\theta^2\ex_{\C}(E_{x_0},x_0,\frac{r}{(2\Lambda)^{1/2}},\widetilde\nu_0)+\tilde\theta^\alpha\frac{r^{\alpha/2}}{(2\Lambda)^{\alpha/4}}\Big)\nonumber\\
		&\leq C\tilde\theta^{\alpha}\ex_{\C}(E,x_0,r,\nu_0).\label{decay ineq1}
	\end{align}
	Let $\nu_1\in\SS^{n-1}$ denote the preimage of $\widetilde\nu_1$ under the change of variable $T_{x_0}$, that is,
	\begin{align}
		\nu_1=\frac{A^{-1/2}(x_0) \tilde \nu_1}{|A^{-1/2}(x_0) \tilde \nu_1|}.
	\end{align}
	Note that
	\begin{align}
	\C\big(x_0,\Big(\frac{\lambda}{4\Lambda}\Big)^{1/2}\tilde\theta r, \nu_1)\subset \B\big(x_0,\Big(\frac{\lambda}{2\Lambda}\Big)^{1/2}\tilde\theta r\big)\subset \W_{x_0}\big(x_0,\frac{\tilde\theta r}{(2\Lambda)^{1/2}}\big).
	\end{align}
	So by definition of $\theta_1$ and Proposition \ref{comparability of excess}, we have that
	\begin{align}
		\ex_{\C}(E,x_0,\theta_1 r,\nu_1)&\leq (2/\lambda)^{(n-1)/2}\ex_{\W}\big(E,x_0, \frac{\tilde\theta r}{(2\Lambda)^{1/2}},\nu_1\big)\nonumber\\
		&\leq (2/\lambda)^{(n-1)/2}C\ex_{\B}\big(E_{x_0},x_0, \frac{\tilde\theta r}{(2\Lambda)^{1/2}},\widetilde\nu_1\big)\nonumber\\
		&\leq C\ex_{\C}\big(E_{x_0},x_0, \frac{\tilde\theta r}{(2\Lambda)^{1/2}},\widetilde\nu_1\big).
	\end{align}
	Combining this with \eqref{decay ineq1} yields
	\begin{align}
		\ex_{\C}(E,x_0,\theta_1 r,\nu_1) &\leq C\tilde\theta^{\alpha}\ex_{\C}(E,x_0,r,\nu_0) \leq C \tilde\theta^{\alpha-2\beta}\theta_1^{2\beta}\ex_{\C}(E,x_0,r,\nu_0)\label{decay ineq2}.
	\end{align}	
	Using this $C$ we now make our choice for $\tilde\theta$ by setting
	\begin{align}
		\tilde\theta=\min\Big\{\frac{1}{104},\Big(\frac{1}{C}\Big)^{1/(\alpha-2\beta)}\Big\}.
	\end{align}
	The condition $\tilde\theta\in (0,1/104]$ allows us to apply Proposition \ref{excess-improvement by tilting} as above and since $C\tilde\theta^{\alpha-2\beta}\leq 1$, \eqref{decay ineq2} implies
	\begin{align}
		\ex_{\C}(E,x_0,\theta_1 r,\nu_1) \leq \theta_1^{2\beta}\ex_{\C}(E,x_0,r,\nu_0)\leq \theta_1^{2\beta}\ex_{\C}^*(E,x_0,r,\nu_0),
	\end{align}
    completing the proof of \eqref{tilt-decay}.
	
	Now we turn to \eqref{tilt-control}. Integrating the inequality $|\nu_1-\nu_0|^2\leq 2|\nu_E-\nu_1|^2+2|\nu_E-\nu_0|^2$ over the set $\C(x_0,\theta_1r,\nu_1)$ which is contained in $\C(x_0,r,\nu_0)$ gives
	\begin{align}
		\frac{P(E;\C(x_0,\theta_1r,\nu_1))}{(\theta_1r)^{n-1}}|\nu_1-\nu_0|^2\leq 4\ex_{\C}(E,x_0,\theta_1r,\nu_1)+\frac{4}{\theta_1^{n-1}}\ex_{\C}(E,x_0,r,\nu_0).
	\end{align}
	The lower density estimate of \eqref{perimeter bounds} along with \eqref{tilt-decay} imply
	\begin{align}
		c |\nu_1-\nu_0|^2\leq 4(1+\theta_1^{1-n})\ex_{\C}(E,x_0,r,\nu_0)
	\end{align}
    completing the proof of \eqref{tilt-control}.
	\end{proof}
	
	Now we prove our main theorem. Before we start, let's briefly describe the structure of the argument. In the Lipschitz approximation theorem, Theorem \ref{Lipschitz approximation theorem}, we saw that given a small excess assumption, there is a Lipschitz function $u:\R^{n-1}\to\R$ such that, setting
		\begin{align} 
			M=\C(x_0,r,\nu)\cap \p E\text{ and }M_0=\{x\in M: \sup_{0< s< 8r}\ex_{\C}(E,x,s,\nu)\leq \delta_0\},
		\end{align}
		the translated graph graph $\Gamma=x_0+\{(z,u(z)): z\in \D_r\}$ of $u$ over $\D_r$ contains $M_0$.
    We proceed by iterating \eqref{tilt-decay} at points $x\in M$ to obtain a sequence of unit vectors $\nu_j(x)$ for which certain decay estimates of the excess hold, namely \eqref{iterative decay} and \eqref{iterative tilt}. Using this, we show that $x\in\rb E$ and that $\nu_j(x)$ converges to $\nu_E(x)$. Moreover, our iteration gives estimates for H\"older continuity of $\nu_E$. Lastly, we show $M_0$ in fact equals $M$, that is, $\C(x_0,r,\nu)\cap \p E$ equals the graph of $u$. H\"older estimates for $\nabla' u$ follow from the ones for $\nu_E$.
	
	\begin{theorem}[$C^{1,\alpha/4}$-regularity of almost-minimizers of $\sF_A$]\label{regularity theorem}
	There exist positive constants\\ $\epsilon_4 = \epsilon_4(n,\lambda,\Lambda,\kappa,\alpha,r_0,||A||_{C^\alpha})$ and $C_5 = C_5(n,\lambda,\Lambda,\kappa,\alpha,r_0,||A||_{C^\alpha})$ with the following property. If $E$ is a $(\kappa,\alpha)$-almost-minimizer of $\sF_A$ in $\C(x_0,13r,\nu_0)$ with $13r<r_0$ and $x_0\in\p E$ such that
	\begin{align}
		\ex_{\C}(E,x_0,13r,\nu_0)+ r^{\alpha/2}\leq \epsilon_4,\label{small excess hypothesis}
	\end{align}
	then there exists a Lipschitz function $u\colon\R^{n-1}\to\R$ with $\Lip(u)\leq 1$ satisfying
	\begin{align}
		\sup_{\R^{n-1}}\frac{|u|}{r}\leq C_5 \ex_{\C}(E,x_0,13r,\nu_0)^{1/(2(n-1))}
	\end{align}
	such that
	\begin{align}
		 \C(x_0,r,\nu_0)\cap \p E &= x_0+\big\{(z,u(z)): z\in \D_r\big\},\\
		\C(x_0,r,\nu_0)\cap E &= x_0+\big\{(z,t): z\in \D_r,\: -r<t<u(z)\big\}
	\end{align}
	and $u\in C^{1,\alpha/4}(\D_r)$ with
	\begin{align}
		|\nabla'u(z)-\nabla'u(w)| &\leq C_5\big(\ex_{\C}(E,x_0,13r,\nu_0)+ r^{\alpha/2}\big)^{1/2} \Big(\frac{|z-w|}{r}\Big)^{\alpha/4},\label{Holder cont of u estimate}\\
		|\nu_E(x)-\nu_E(y)| &\leq C_5\big(\ex_{\C}(E,x_0,13r,\nu_0)+r^{\alpha/2}\big)^{1/2}\Big(\frac{|x-y|}{r}\Big)^{\alpha/4},\label{Holder cont of outer normal estimate}
	\end{align}
	for every $z,w\in \D_r$ and $x,y\in \C(x_0,r,\nu_0)\cap \p E$.
\end{theorem}

	\begin{proof}
	Without loss of generality we may assume $x_0=0$. Let $\theta_1<1$, $\epsilon_3$, and $C_4$ denote the constants from Lemma \ref{regularity lemma} with the choice $\beta=\alpha/4$ which hence depend only on $n,\lambda,\Lambda,\kappa,\alpha,r_0$, and $||A||_{C^\alpha}$. As mentioned before, we will choose $\boxed{\epsilon_4\leq \epsilon_3}\leq\epsilon_2\leq \epsilon_1$ and apply the Lipschitz approximation theorem, Theorem \ref{Lipschitz approximation theorem}. This gives that there is a Lipschitz function $u:\R^{n-1}\to\R$ with $\Lip u\leq 1$ satisfying
		\begin{align}
		\sup_{\R^{n-1}} \frac{|u|}{r}\leq C_1\ex_{\C}(E,0,13r,\nu_0)^{1/2(n-1)}
		\end{align}  and such that, setting
		\begin{align} 
			M=\C(0,r,\nu_0)\cap \p E\text{ and }M_0=\{x\in M: \sup_{0< s< 8r}\ex_{\C}(E,x,s,\nu_0)\leq \delta_0\},
		\end{align}
		the translated graph graph $\Gamma=\{(z,u(z)): z\in \D_r\}$ of $u$ over $\D_r$ contains $M_0$, that is, $M_0\subset M\cap\Gamma$. 
		As in Lemma \ref{regularity lemma}, we define
		\begin{align}
    	\ex_{\C}^*(E,x,s,\nu)=\max\Big\{\ex_{\C}(E,x,s,\nu),\frac{s^{\alpha/2}}{\theta_1^{n-1+\alpha/2}}\Big\},\qquad\text{ for } x\in\R^n,\ s>0,\ \nu\in\SS^{n-1}.
        \end{align}
		Let $x\in M$. Then $\C(x,8r,\nu_0)\subset\C(0,13r,\nu_0)$ and so
		\begin{align}
			\ex_{\C}^*(E,x,8r,\nu_0) &\leq \ex_{\C}(E,x,8r,\nu_0)+\frac{ (8r)^{\alpha/2}}{\theta_1^{n-1+{\alpha/2}}}\leq \Big(\frac{13}{8}\Big)^{n-1}\ex_{\C}(E,0,13r,\nu_0)+\frac{ (8r)^{\alpha/2}}{\theta_1^{n-1+{\alpha/2}}}\nonumber\\
			&\leq\Big(\frac{13}{8}\Big)^{n-1} \frac{8^{\alpha/2}}{\theta_1^{n-1+{\alpha/2}}}(\ex_{\C}(E,0,13r,\nu_0)+ r^{\alpha/2})=C(\ex_{\C}(E,0,13r,\nu_0)+ r^{\alpha/2})\label{bound for excess on 8r}.
		\end{align}
		where $C=C(n,\lambda,\Lambda,\kappa,\alpha,r_0,||A||_{C^\alpha})$. For this constant $C$, choose $\epsilon_4$ to also satisfy 
		\begin{align}
			\boxed{\epsilon_4\leq C^{-1}\epsilon_3}
		\end{align}
	    so that $\ex_{\C}^*(E,x,8r,\nu_0) \leq \epsilon_3$.
		
		\begin{claim}
		 There exists a sequence $\{\nu_j(x)\}_{j=1}^\infty\subset\SS^{n-1}$ and $\nu(x)\in\SS^{n-1}$ with $\nu_j(x)\to\nu(x)$ such that for every $j\geq 0$,
		\begin{align}
			\ex_{\C}^*(E,x,\theta_1^{j}8r,\nu_{j}(x))&\leq \theta_1^{(\alpha/2)j}\ex_{\C}^*(E,x,8r,\nu_0)\label{iterative decay}\\
	        |\nu(x)-\nu_j(x)|^2 &\leq C\: \theta_1^{{(\alpha/2)} j}\ex_{\C}^*(E,x,8r,\nu_0)\label{iterative tilt}
		\end{align}
		for some constant $C=C(n,\lambda,\Lambda,\kappa,\alpha,r_0,||A||_{C^\alpha})$.
		\end{claim}
		
		\begin{claimproof}
		Since $\ex_{\C}^*(E,x,8r,\nu_0) \leq \epsilon_3$, we may apply Lemma \ref{regularity lemma} to find $\nu_1(x)\in\SS^{n-1}$ such that
		\begin{align}
			\ex_{\C}^*(E,x,\theta_1 8r,\nu_1(x))&\leq \theta_1^{{\alpha/2}}\ex_{\C}
			^*(E,x,8r,\nu_0),\\
			|\nu_1-\nu_0|^2&\leq C_4\ex_{\C}^*(E,x,8r,\nu_0).
		\end{align}
		In particular, since $\theta_1<1$,
		\begin{align}
			\ex_{\C}^*(E,x,\theta_1 8r,\nu_1(x))\leq \ex_{\C}^*(E,x,8r,\nu_0)\leq\epsilon_3.
		\end{align}
		Proceeding inductively we find a sequence $\{\nu_j(x)\}_{j=0}^\infty\subset\SS^{n-1}$ such that
		\begin{align}
			\ex_{\C}^*(E,x,\theta_1^{j+1}8r,\nu_{j+1}(x))&\leq \theta_1^{{\alpha/2}}\ex_{\C}^*(E,x,\theta_1^{j}8r,\nu_{j}(x))\leq \epsilon_3,\label{induction decay}\\
			|\nu_{j+1}(x)-\nu_{j}(x)|^2 &\leq C_4 \ex_{\C}^*(E,x,\theta_1^{j}8r,\nu_j(x))\label{induction tilt}.
		\end{align}
		for $j\geq 0$. Stringing together the inequalities of \eqref{induction decay} gives \eqref{iterative decay} and stringing together the inequalities of \eqref{induction tilt} gives
		\begin{align}
			|\nu_{j+1}(x)-\nu_{j}(x)|^2 &\leq C_4 \theta_1^{{(\alpha/2)} j} \ex_{\C}^*(E,x,8r,\nu_0)
		\end{align}
		for $j\geq 0$. Given $0\leq j<h$, it follows that
		\begin{align}
			|\nu_h(x)-\nu_j(x)|&\leq \sum_{k=j}^{h-1} |\nu_{k+1}(x)-\nu_{k}(x)|\leq \big(C_4\ex_{\C}^*(E,x,8r,\nu_0)\big)^{1/2}\sum_{k=j}^{h-1}\theta_1^{(\alpha/4) k}\nonumber\\
				&\leq \big(C_4\ex_{\C}^*(E,x,8r,\nu_0)\big)^{1/2} \sum_{k=j}^\infty \theta_1^{(\alpha/4) k}= \big(C_4\ex_{\C}^*(E,x,8r,\nu_0)\big)^{1/2}\frac{\theta_1^{(\alpha/4)j}}{1-\theta_1^{\alpha/4}}
		\end{align}
		and so 
		\begin{align}
			|\nu_h(x)-\nu_j(x)|^2 \leq C\: \theta_1^{{(\alpha/2)} j}\ex_{\C}^*(E,x,8r,\nu_0)\label{Cauchy seq}
		\end{align}
		where $C=C(n,\lambda,\Lambda,\kappa,\alpha,r_0,||A||_{C^\alpha})$ since $\theta_1$ depends only on these constants too. Hence $\{\nu_j(x)\}_{j=1}^\infty\subset\SS^{n-1}$ is Cauchy and so there is $\nu(x)\in\SS^{n-1}$ such that
		$\nu_j(x)\to\nu(x)$ as $j\to\infty$. Sending $h\to\infty$ in \eqref{Cauchy seq} gives \eqref{iterative tilt} and this first claim is proved. 
		\end{claimproof}
		
		\begin{claim} There is a constant $C=C(n,\lambda,\Lambda,\kappa,\alpha,r_0,||A||_{C^\alpha})$ such that
		    \begin{align}
			\ex_{\C}^*(E,x,s,\nu(x))&\leq     C\Big(\frac{s}{r}\Big)^{{\alpha/2}}\ex_{\C}^*(E,x,8r,\nu_0)\qquad\forall s\in (0,4r),\label{decay ineq}\\
			\ex_{\C}(E,x,s,\nu_0)&\leq C\ex_{\C}^*(E,x,8r,\nu_0)\qquad\forall s\in(0,8r).\label{uniform smallness of excess}
	    	\end{align}
	    \end{claim}
		
		\begin{claimproof}
		Given $s\in (0,4r)$, there is $j\geq 0$ such that 
		\begin{align}
			\theta_1^{j+1} 8r< 2s\leq \theta_1^{j} 8r.
		\end{align}
		Integrating $|\nu_E-\nu(x)|^2\leq 2|\nu_E-\nu_j(x)|^2+2|\nu(x)-\nu_j(x)|^2$ with respect to $\Hm\rstr \rb E$ over $\C(x,s,\nu(x))\subset \C(x,2s,\nu_j(x))$, and using the the perimeter bound \eqref{perimeter bounds} and \eqref{iterative tilt}, it follows that
		\begin{align}
			\ex_{\C}^*(E,x,s,\nu(x))&\leq 2^{n}\ex_{\C}^*(E,x,2s,\nu_j(x))+C|\nu(x)-\nu_j(x)|^2\nonumber\\
			&\leq 2^{n}\Big(\frac{\theta_1^j8r}{2s}\Big)^{n-1}\ex_{\C}^*(E,x,\theta_1^j8r,\nu_j(x))+C|\nu(x)-\nu_j(x)|^2\nonumber\\
			&\leq 2^{n}\Big(\frac{1}{\theta_1}\Big)^{n-1}\ex_{\C}^*(E,x,\theta_1^j8r,\nu_j(x))+C|\nu(x)-\nu_j(x)|^2\nonumber\\
			&\leq C\theta_1^{{(\alpha/2)} j}\ex_{\C}^*(E,x,8r,\nu_0)\leq \frac{C}{\theta_1^{{\alpha/2}}}(\theta_1^{j+1})^{{\alpha/2}} \ex_{\C}^*(E,x,8r,\nu_0)
			\nonumber\\
			&\leq C\Big(\frac{s}{r}\Big)^{{\alpha/2}} \ex_{\C}^*(E,x,8r,\nu_0)
		\end{align}
		which is \eqref{decay ineq}. Now, take $s\in (0,8r)$. In the case where $s\in (2r,8r)$, it follows that
		\begin{align}
			\ex_{\C}(E,x,s,\nu_0)\leq \Big(\frac{8r}{s}\Big)^{n-1}\ex_{\C}(E,x,8r,\nu_0)\leq 4^{n-1}\ex_{\C}^*(E,x,8r,\nu_0).
		\end{align}
		Otherwise, $s\in (0,2r)$ and so integrating $|\nu_E-\nu_0|^2\leq 2|\nu_E-\nu(x)|^2+2|\nu(x)-\nu_0|^2$ with respect to $\Hm\rstr\rb E$ over $\C(x,s,\nu_0)\subset \C(x,2s,\nu(x))$, using \eqref{decay ineq} with $2s\in (0,4r)$ and \eqref{iterative tilt} with $j=0$ gives
		\begin{align}
			\ex_{\C}(E,x,s,\nu_0) &\leq 2^n\ex_{\C}(E,x,2s,\nu(x))+C|\nu(x)-\nu_0|^2\nonumber\\
			 &\leq 2^n C\Big(\frac{2s}{r}\Big)^{{\alpha/2}}\ex_{\C}^*(E,x,8r,\nu_0)+C|\nu(x)-\nu_0|^2\leq C\ex_{\C}^*(E,x,8r,\nu_0).
		\end{align}
	    Hence \eqref{uniform smallness of excess} holds. This completes the proof of our second claim.
		\end{claimproof}
		
		Suppose $x,y\in M=\C(0,r,\nu_0)\cap\p E$. Then $|x-y|<\sqrt 2 r$ and so there is some $j\geq 0$ such that
        \begin{align}
            \theta_1^{j+1}\sqrt 2 r\leq|x-y|< \theta_1^{j}\sqrt 2 r.
        \end{align}
        Integrating $|\nu_j(x)-\nu_j(y)|^2\leq 2|\nu_E-\nu_j(x)|^2+2|\nu_E-\nu_j(y)|^2$ with respect to $\Hm\rstr \rb E$ over $\B(x,\theta_1^{j}r)\subset\C(x,\theta_1^{j} r,\nu_j(x))\subset \C(y,\theta_1^{j}8 r,\nu_j(y))$ and using the perimeter bounds \eqref{perimeter bounds} gives
        \begin{align}
            c|\nu_j(x)-\nu_j(y)|^2 &\leq 4\ex_{\C}(E, x, \theta_1^{j}r,\nu_j(x))+4\cdot 8^{n-1} \ex_{\C}(E, y, \theta_1^{j}8r,\nu_j(y))\nonumber\\
            &\leq 4\cdot 8^{n-1}\ex_{\C}(E, x, \theta_1^{j}8r,\nu_j(x))+4\cdot 8^{n-1} \ex_{\C}(E, y, \theta_1^{j}8r,\nu_j(y)).
        \end{align}
        Hence by \eqref{iterative decay} and the definition of $\ex_{\C}^*$ we have, $\ex_{\C}$
        \begin{align}
            |\nu_j(x)-\nu_j(y)|^2 &\leq C \theta_1^{(\alpha/2)j}(\ex_{\C}^*(E,x, 8r, \nu_0)+\ex_{\C}^*(E,y, 8r, \nu_0))
        \end{align}
        By this, \eqref{iterative tilt}, \eqref{bound for excess on 8r}, and $\theta_1^{j+1}\sqrt 2 r\leq|x-y|$, it follows that
        \begin{align}
            |\nu(x)-\nu(y)|^2 &\leq 3(|\nu(x)-\nu_j(x)|^2+|\nu_j(x)-\nu_j(y)|^2+|\nu_j(y)-\nu(y)|^2)\nonumber\\
            &\leq  C \theta_1^{(\alpha/2)j}(\ex_{\C}^*(E,x, 8r, \nu_0)+\ex_{\C}^*(E,y, 8r, \nu_0))\nonumber\\
            &\leq  C \theta_1^{(\alpha/2)j}(\ex_{\C}(E,0, 13r, \nu_0)+r^{\alpha/2})\nonumber\\
            &\leq C(\ex_{\C}(E,0, 13r, \nu_0)+r^{\alpha/2}) \Big(\frac{|x-y|}{r}\Big)^{\alpha/2}
        \end{align}
        and so
        \begin{align}
            |\nu(x)-\nu(y)|\leq C \big(\ex_{\C}(E,0,13r,\nu_0)+r^{\alpha/2}\big)^{1/2}\Big(\frac{|x-y|}{r}\Big)^{\alpha/4}.\label{Holder cont of nu estimate}
        \end{align}
        
        We now prove $x\in \rb E$ and $\nu_E(x)=\nu(x)$ so that \eqref{Holder cont of nu estimate} becomes \eqref{Holder cont of outer normal estimate}, proving the H\"older continuity of the outer normal to $E$.
        
        By \eqref{decay ineq}, $\lim_{s\to 0^+}\ex_{\B}(E,s,r,\nu(x))=0$. So by perimeter bounds \eqref{perimeter bounds}, we 
        \begin{align}
            \lim_{s\to 0^+}\frac{1}{P(E;\B(x,s))}\int_{\B(x,s)\cap\rb E}\frac{|\nu_E(z)-\nu(x)|^2}{2}\dH(z)=0.
        \end{align}
        Expanding $|\nu_E(z)-\nu(x)|^2=|\nu_E(z)|^2-2\nu_E(z)\cdot \nu(x)+|\nu(x)|^2=2-2\nu_E(z)\cdot \nu(x)$ implies
        \begin{align}
            \nu(x)\cdot\lim_{s\to 0^+}\frac{1}{P(E;\B(x,s))}\int_{\B(x,s)\cap\rb E}\nu_E(z)\dH(z)=1.
        \end{align}
        Since $|\nu(x)|=1$ and
        \begin{align}
            \Big|\lim_{s\to 0^+}\frac{1}{P(E;\B(x,s))}\int_{\B(x,s)\cap\rb E}\nu_E(z)\dH(z)\Big|\leq 1,
        \end{align}
        this implies 
        \begin{align}
            \nu(x)=\lim_{s\to 0^+}\frac{1}{P(E;\B(x,s))}\int_{\B(x,s)\cap\rb E}\nu_E\dH.
        \end{align}
        Since $\nu(x)\in\SS^{n-1}$, this by definition means $x\in\rb E$ with $\nu_E(x)=\nu(x)$ and hence \eqref{Holder cont of outer normal estimate} holds.

        Combining \eqref{uniform smallness of excess} with \eqref{bound for excess on 8r} gives
        \begin{align}
            \ex_{\C}(E,x,s,\nu_0)&\leq C(\ex_{\C}(E,0,13r,\nu_0)+r^{\alpha/2}),\qquad\forall s\in(0,8r).
        \end{align}
        Lastly, for this constant $C$, we choose $\epsilon_4$ to also satisfy
		\begin{align}
			\boxed{\epsilon_4\leq C^{-1}\delta_0}
		\end{align}
		where $\delta_0$ is the constant from the Lipschitz approximation theorem. It follows for $x\in M$ that
		\begin{align}
			\sup_{0<s<8r}\ex_{\C}(E,x,s,\nu_0)&\leq \delta_0
		\end{align}
		and so $M=M_0\subset \Gamma$. By the Lipschitz graph criterion, \cite[Theorem 23.1]{maggi}, the graph of the Lipschitz function $u$ coincides with $\p E$ in $\C(0,r,\nu_0)$. Moreover,
		\begin{align}\label{normal vector of graph}
		    \nu_E(x)=\frac{(-\nabla'u(\pp x),1)}{\sqrt{1+|\nabla'u(\pp x)|^2}}
		\end{align}
		for all $x\in\C(0,r,\nu_0)\cap\p E$. Since $\Lip u\leq 1$, for  $x,y\in\C(0,r,\nu_0)\cap\p E$, it follows that
		\begin{align}
	        &|\nabla'u(\pp x)-\nabla'u(\pp y)|^2 \leq 2\bigg|\frac{\nabla'u(\pp x)}{\sqrt{1+|\nabla'u(\pp x)|^2}}-\frac{\nabla'u(\pp y)}{\sqrt{1+|\nabla'u(\pp x)|^2}}\bigg|^2\nonumber\\
	        &\leq4\bigg|\frac{\nabla'u(\pp x)}{\sqrt{1+|\nabla'u(\pp x)|^2}}-\frac{\nabla'u(\pp y)}{\sqrt{1+|\nabla'u(\pp y)|^2}}\bigg|^2+4\bigg|\frac{\nabla'u(\pp y)}{\sqrt{1+|\nabla'u(\pp y)|^2}}-\frac{\nabla'u(\pp y)}{\sqrt{1+|\nabla'u(\pp x)|^2}}\bigg|^2\\
	        &\leq 4|\nu_E(x)-\nu_E(y)|^2
         \end{align} 
         and $|x-y|^2=|\pp x-\pp y|^2+|u(\pp x)-u(\pp y)|^2\leq 2|\pp x-\pp y|^2$. So by \eqref{Holder cont of outer normal estimate} we have $u$ is $C^{1,\alpha/4}$ with the estimate \eqref{Holder cont of u estimate}.
		\end{proof}

		\begin{theorem}[Regularity of the reduced boundary and characterization of the singular set]\label{regularity of reduced boundary and characterization of sing set} If $U$ is an open set in $\R^n$, $n\geq 2$, and $E$ is a $(\kappa,\alpha)$-almost-minimizer of $\sF_A$ in $U$, then $U\cap\rb E$ is a $C^{1,\alpha/4}$-hypersurface that is relatively open in $U\cap \p E$, and it is $\Hm$-equivalent to $U\cap\p E$. Hence the singular set $\Sigma(E; U)$ of $E$ in $U$, 
			\begin{align}
				\Sigma(E;U)=U\cap (\p E\setminus\rb E),
			\end{align}
			is closed. Moreover, $\Sigma(E;U)$ is characterized in terms of the excess as follows:
			\begin{align}\label{characterization of singular set}
				\Sigma(E; U)=\Big\{x\in U\cap \p E: \inf_{0<13r<r_0, \B(x,13\sqrt{2}r)\subsetcc U} \big(\inf_{\nu_0\in\SS^{n-1}}\ex_{\C}(E,x,13r,\nu_0)+r^{\alpha/2}\big)\geq \epsilon_4\Big\}
			\end{align}
			where $\epsilon_4=\epsilon_4(n,\lambda,\Lambda,\kappa,\alpha,r_0,||A||_{C^\alpha})$ is the positive constant from Theorem \ref{regularity theorem}.
		\end{theorem}
		
		\begin{proof}
			The regularity and relative openness of $U\cap\rb E$ follows from Theorem \ref{regularity theorem} and the $\Hm$-equivalence follows from Proposition \ref{volume and perimeter bounds}. Consequently, $\Sigma(E;U)$ is closed. Hence all we need to show is \eqref{characterization of singular set}. Consider the set defined by
	    	\begin{align}
				\Sigma=\Big\{x\in U\cap \p E: \inf_{0<13r<r_0, \B(x,13\sqrt{2}r)\subsetcc U} \big(\inf_{\nu_0\in\SS^{n-1}}\ex_{\C}(E,x,13r,\nu_0)+r^{\alpha/2}\big)\geq \epsilon_4\Big\}.
			\end{align}
			We show $\Sigma=\Sigma(E;U)$.
				
			By Proposition \ref{vanishing of the excess at the reduced boundary}, for each $x\in U\cap\rb E$, we have
			\begin{align}
				\lim_{r\to 0^+} \big(\inf_{\nu_0\in\SS^{n-1}} \ex_{\C}(E,x,13r,\nu_0)+r^{\alpha/2}\big)=0	
			\end{align}
			and so $x\in (U\cap\p E)\setminus\Sigma$. Hence $U\cap\rb E\subset (U\cap\p E)\setminus\Sigma$.
			
			If $x\in (U\cap\p E)\setminus\Sigma$, then there is $0<13r<r_0$, $\nu_0\in\SS^{n-1}$, with $\C(x,13r,\nu_0)\subsetcc U$ such that
				\begin{align}
					\ex_{\C}(E,x,13r,\nu_0)+r^{\alpha/2}<\epsilon_5.
				\end{align}
				By Theorem \ref{regularity theorem}, $\C(x,r,\nu_0)\cap\p E$ coincides with the graph of a $C^{1,\alpha/4}$-function and so $x\in \rb E$. Hence  $(U\cap\p E)\setminus\Sigma\subset U\cap\rb E$.
		\end{proof}
		
		Now that we have established regularity of almost-minimizers at points in the reduced boundary, we wish to study the singular set which we do in the next section. However, before we move on to that, we prove the convergence of the outer unit normal vectors along sequences of almost-minimizers and points in the reduced boundaries. The contrapositive of this will be a useful tool in showing that the blow-ups at a singular point must converge to a singular point.
		
		We first need the following lemma regarding almost upper semicontinuity of the excess. Recall from Section \ref{basic properties section} the class $\sA$ of uniformly elliptic, H\"older continuous matrices with respect to the given universal constants and the class $\sM$ of $(\kappa,\alpha)$-almost-minimizers of $\sF_A$ with $A\in\sA$.
		
		\begin{lemma}[Almost upper semicontinuity of the excess]\label{excess usc lemma}
		    Suppose that $\{E_h\}_{h\in\NN}\subset\sM$ is a sequence of $(\kappa,\alpha)$-almost-minimizers of $\sF_{A_h}$ in $U$ at scale $r_h$, $r_0=\liminf_{h\to\infty} r_h>0$, $V\subsetcc U$ is an open set with $P(V)<\infty$ such that $V\cap E_h\to E$ for a set $E$ of finite perimeter, and $A_h\to A$ uniformly on compact sets for some $A\in\sA$. Furthermore suppose $x_0\in V\cap\p E$ and $r<r_0$ with $A(x_0)=I$,  $\C(x_0,r,\nu_0)\subsetcc V$, and 
			\begin{align}
			   \Hm(\rb E\cap\p\C(x_0,r,\nu_0))=0,
			\end{align} 
			then
			\begin{align}\label{limsup of excess}
				\limsup_{h\to\infty}	 \ex_{\C}(E_h,x_0,r,\nu_0)\leq \ex_{\C}(E,x_0,r,\nu_0)+Cr^\alpha.
			\end{align}
			for some positive constant $C=C(n,\lambda,\Lambda,\kappa,\alpha,r_0,||A||_{C^\alpha})$.
		\end{lemma}
		
		\begin{proof}
			By Proposition \ref{closedness}, $E$ is a $(\kappa,\alpha)$-almost-minimizer of $\sF_A$ in $V$ at scale $r_0$, satisfying 
			\begin{align}
		    	& \mu_{V\cap E_h}\wkly \mu_E,\label{weak convergence1}\\
			    & \sF_{A_h}(E_h;\:\cdot\:)\wkly\sF_A(E;\:\cdot\:)\text{ in } V.\label{weak convergence of energies1}
		    \end{align}
		    We write $\C_r$ for $\C(x_0,r,\nu)$ to simplify notation and claim
			\begin{align}\label{lemma claim}
			\limsup_{h\to\infty} P(E_h;\C_r)\leq P(E;\C_r)+Cr^{\alpha+n-1}
			\end{align}
			for some $C=C(n,\lambda,\Lambda,\kappa,\alpha,r_0,||A||_{C^\alpha})$.
			
			To show this, note
			\begin{align}\label{triangle ineq1}
			 |P(E_h;\C_r)-P(E;\C_r)|\leq  |P(E_h;\C_r)-\sF_{A_h}(E_h;\C_r)|+ |\sF_{A_h}(E_h;\C_r)-P(E;\C_r)|.
			\end{align}
			Noting $|A(x)-I|\leq ||A||_{\C^\alpha} |x-x_0|^\alpha$, we bound the first term of \eqref{triangle ineq1} by			
			\begin{align}\label{triangle ineq2}
			|P(E_h;\C_r)-\sF_{A_h}(E_h;\C_r)|&\leq |P(E_h;\C_r)-\sF_A(E_h;\C_r)|+|\sF_A(E_h;\C_r)-\sF_{A_h}(E_h;\C_r)|\nonumber\\
				 &\leq C||A||_{C^\alpha} r^\alpha P(E_h;\C_r)+||A-A_h||P(E_h;\C_r).
			\end{align}
			Similarly, for the second term of \eqref{triangle ineq1}, we have
			\begin{align}\label{triangle ineq3}
				|\sF_{A_h}(E_h;\C_r)-P(E;\C_r)|&\leq |\sF_{A_h}(E_h;\C_r)-\sF_{A}(E;\C_r)|+|\sF_{A}(E;\C_r)-P(E;\C_r)|\nonumber\\
				 &\leq |\sF_{A_h}(E_h;\C_r)-\sF_{A}(E;\C_r)|+C||A||_{C^\alpha} r^\alpha P(E;\C_r).
			\end{align}
			Since $x_0\in V\cap\rb E$, by Proposition \ref{closedness} there is a sequence $x_h\in V\cap\rb E_h$ such that $x_h\to x_0$. Given $r<s$, we have $\C_r=\C(x_0,r,\nu)\subset \C(x_h,s,\nu)\subsetcc U$ for large $h$. So
			\begin{align}
				\limsup_{h\to\infty} P(E_h;\C_r)\leq \limsup_{h\to\infty} P(E_h;\C(x_h,s,\nu))\leq Cs^{n-1}
			\end{align}
			by the upper perimeter bound \eqref{perimeter bounds} for $E_h$. Hence $\limsup_{h\to\infty} P(E_h;\C_r)\leq Cr^{n-1}$. We also have $P(E;\C_r)\leq Cr^{n-1}$. So by \eqref{triangle ineq1}, \eqref{triangle ineq2}, and \eqref{triangle ineq3}, we have
			\begin{align}
			 |P(E_h;\C_r)-P(E;\C_r)|\leq  |\sF_{A_h}(E_h;\C_r)-\sF_{A}(E;\C_r)|+C||A||_{C^\alpha} r^{\alpha+n-1}+C||A-A_h||r^{n-1}.
			\end{align}
	    	Note $\sF_A(V\cap E_h;\C_r)=\sF_A(E_h;\C_r)$ because $\C_r\subsetcc V$ and so $\sF_A(E_h;\C_r)\to \sF_{A}(E;\C_r)$  by \eqref{weak convergence of energies1} and $\Hm(\rb E\cap\p\C_r)=0$. This and the uniform convergence of $A_h\to A$ on $\C_r$ complete the proof of our claim.
			
    		Note $\mu_{V\cap E_h}(\C_r)=\mu_{E_h}(\C_r)$ because $\C_r\subsetcc V$ and so 
			\begin{align}\label{dot prod convergence}
				\nu\cdot\mu_{E_h}(\C_r)\to \nu\cdot \mu_E(\C_r)
			\end{align}
			by \eqref{weak convergence1} and $\Hm(\rb E\cap\p\C_r)=0$. By $|\nu-\nu_E|^2=2(1-(\nu\cdot\nu_E))$ and $|\nu-\nu_{E_h}|^2=2(1-(\nu\cdot\nu_{E_h}))$, we have
			\begin{align}
				\ex_{\C}(E, x_0,r,\nu)=\frac{P(E;\C_r)-\nu\cdot\mu_{E}(\C_r)}{r^{n-1}}\qquad\text{and}\qquad	\ex_{\C}(E_h, x_0,r,\nu)=\frac{P(E_h;\C_r)-\nu\cdot\mu_{E_h}(\C_r)}{r^{n-1}}.	
			\end{align}
			From this, \eqref{lemma claim}, and \eqref{dot prod convergence}, we obtain \eqref{limsup of excess}.
		\end{proof}
	
		\begin{theorem}[Convergence of outer unit normals]\label{convergence of outer unit normals}
		If $\{E_h\}_{h\in\NN}$ and $E$ are $(\kappa,\alpha)$-almost-minimizers of $\sF_{A_h}$ and $\sF_A$, respectively, in the open set $U\subset\R^n$ at scale $r_0$, and
		\begin{align}
			E_h\loc E,\ \ A_h\to A\ \text{uniformly on compact sets},\ \ x_h\in U\cap\p E_h, \ \ x_0\in U\cap \rb E, \ \ x_h\to x_0,
		\end{align}
		then $x_h\in U\cap\rb E_h$ for $h$ large enough. Moreover,
		\begin{align}
			\lim_{h\to\infty} \nu_{E_h}(x_h)=\nu_E(x_0).
		\end{align}
		\end{theorem}
	
		\begin{proof}
			Considering the translated sets $E_h+(x_0-x_h)$, note that
			\begin{align}
			    \nu_{E_h}(x_h)=\nu_{E_h+(x_0-x_h)}(x_0),
			\end{align}
		    $E_h+(x_0-x_h)\loc E$, and $A_h(\:\cdot\:+(x_0-x_h))\to A$ uniformly on compact sets. Hence by replacing $E_h$ with $E_h+(x_0-x_h)$ and $A_h$ with $A_h(\:\cdot\:+(x_0-x_h))$, and $U$ with $\{x\in U:\dist(x,\p U)>\delta\}$ for some sufficiently small $\delta>0$, we may assume that $x_h=x_0$ for every $h$. 

		    By applying the change of variable $T_{x_0}$ on $E_h$, $E$ and $A_h$, $A$, we may assume without loss of generality that $A(x_0)=I$. Choose an open set $V\subsetcc U$ with $x_0\in V$ and $P(V)<\infty$. Lemma \ref{excess usc lemma} with $E_h\cap V\loc E\cap V$ implies there is a constant $C$ for which 
			\begin{align}
				\limsup_{h\to\infty} \ex_{\C}(E_h,x_0,13r,\nu_0)+r^{\alpha/2}\leq \ex_{\C}(E,x_0,13r,\nu_0)+Cr^{\alpha/2}
			\end{align}
			holds for every $r>0$ such that $\C(x_0,13r,\nu_0)\subsetcc V$ and $\Hm(\rb E\cap\p \C(x_0,13r,\nu_0))=0$. Since $x_0\in U\cap\rb E$, by Proposition \ref{vanishing of the excess at the reduced boundary} there is $r>0$ and $\nu_0\in\SS^{n-1}$ with $0<13r<r_0$, $\C(x_0,13r,\nu_0)\subsetcc V$, $\Hm(\rb E\cap\p \C(x_0,13r,\nu_0))=0$, and
			\begin{align}
			 \ex_{\C}(E,x_0,13r,\nu_0)+Cr^{\alpha/2}<\epsilon_4
			\end{align}
			where $\epsilon_4$ is the constant from Theorem \ref{regularity theorem}. Then $\ex_{\C}(E_h,x_0,13r,\nu_0)+r^{\alpha/2}<\epsilon_4$ for large $h$ and so by Theorem \ref{regularity of reduced boundary and characterization of sing set}, $x_0\in U\cap\rb E_h$ for large $h$. Moreover, by Theorem \ref{regularity theorem} there exist Lipschitz functions $u_h,u;\D_r\to\R$ with $\Lip u_h,\Lip u\leq 1$ such that
			\begin{align}
				\C(x_0,r,\nu_0)\cap E_h &=x_0+\Big\{(z,t): z\in\D_r,\  -r<t<u_h(z)\Big\},	\\
				\C(x_0,r,\nu_0)\cap E &=x_0+\Big\{(z,t): z\in\D_r,\  -r<t<u(z)\Big\},
			\end{align}
			and such that for $z,w\in\D_r$,
			\begin{align}
				|\nabla'u_h(z)-\nabla'u_h(w)|&\leq C\Big(\frac{|z-w|}{r}\Big)^{\alpha/4}\label{Holder bound on nabla u_h}
			\end{align}
			where $C=C(n,\lambda,\Lambda,\kappa,\alpha,r_0, \sup_h ||A_h||_{C^\alpha})$.
			Then
			\begin{align}
				\int_{\D_r} |u_h-u|=|(E_h\Delta E)\cap \C(x_0,r,\nu_0)|\to 0.
			\end{align}
			It follows by integration by parts and the density of $C_c^1(\D_r)$ in $C_c(\D_r)$ that
			\begin{align}
				\int_{\D_r} \phi\nabla' u_h\to \int_{\D_r}\phi\nabla' u\label{nabla u_h convergence}
			\end{align}
			for every $\phi\in C_c(\D_r)$. By \eqref{Holder bound on nabla u_h}, $\{\nabla' u_h\}$ is equicontinuous and it is bounded by $\Lip{u_h}\leq 1$. Thus by Arzel\`a-Ascoli it is compact under uniform convergence. By \eqref{nabla u_h convergence}, $\nabla' u$ is the only possible limit point of $\{\nabla' u_h\}$. Hence $\nabla' u_h\to \nabla'u$ uniformly on $\D_r$. Consequently, as $x_0\in \rb E_h\cap \rb E$, it follows that
			\begin{align}
				\nu_{E_h}(x_0)=\frac{(-\nabla'u_h(0),1)}{\sqrt{1+|\nabla' u_h(0)|^2}}	\to\frac{(-\nabla'u(0),1)}{\sqrt{1+|\nabla' u(0)|^2}}=\nu_E(x_0)
			\end{align}
			as desired.
		\end{proof}
		
		\section{Analysis of the Singular Set}\label{analysis of singular set section}
		
	    In this final section, we turn to the portion of Theorem \ref{main theorem} which addresses the size of singular set. 
	    
		\begin{theorem}[Dimensional estimates of singular sets of $(\kappa,\alpha)$-almost-minimizers]\label{dimension estimates}
		If $E$ is a $(\kappa,\alpha)$-almost-minimizer of $\sF_A$ in the open set $U\subset\R^n$ at scale $r_0$, then the following hold true:
		\begin{enumerate}
			\item[(i)] if $2\leq n\leq 7$, then $\Sigma(E;U)$ is empty;
			\item[(ii)] if $n=8$, then $\Sigma(E;U)$ has no accumulation points in $U$;
			\item[(iii)] if $n\geq 9$, then $\sH^s(\Sigma(E;U))=0$ for every $s>n-8$.
		\end{enumerate}
		\end{theorem}
		This result is known to be sharp in the case of perimeter minimizers in the sense that Simons' cone,
		\begin{align}
		    \Sigma=\Big\{x\in\R^8: x_1^2+x_2^2+x_3^2+x_4^2=x_5^2+x_6^2+x_7^2+x_8^2\Big\},
		\end{align}
		is a perimeter minimizer in $\R^8$ with singular set $\{0\}$, and for $n\geq 9$, $\Sigma\times \R^{n-8}$ is perimeter minimizer in $\R^n$ that gives $\sH^{n-8}(\Sigma\times \R^{n-8})>0$. Since our surface energies $\sF_A$ include perimeter when $A=I$, our theorem is also sharp.
		
		We use blow-up analysis and a standard Federer dimension reduction argument to prove Theorem \ref{dimension estimates}. The next theorem shows the convergence of the singular set along sequences of almost-minimizers. Recall again from Section \ref{basic properties section} the class $\sA$ of uniformly elliptic, H\"older continuous matrices with respect to the given universal constants and the class $\sM$ of $(\kappa,\alpha)$-almost-minimizers of $\sF_A$ with $A\in\sA$.
	
		\begin{theorem}[Closure and local uniform convergence of singularities]\label{closure of singularities} If $\{E_h\}_{h\in\NN}\subset\sM$ and $E\in\sM$ are $(\kappa,\alpha)$-almost-minimizers of $\sF_{A_h}$ and $\sF_A$, respectively, in the open set $U$ at scale $r_0$, and
			\begin{align}
			E_h\loc E,\ \ A_h\to A\text{ uniformly on compact sets},\ \ x_h\in \Sigma(E_h;U), \ \ x_0\in U\cap \p E, \ \ x_h\to x_0,
			\end{align}
			then $x_0\in\Sigma(E;U)$. Moreover, given $\epsilon>0$ and $H\subset U$ compact, then
			\begin{align}
				\Sigma(E_h;U)\cap H\subset I_\epsilon(\Sigma(E;U)\cap H)\label{sing set convergence}
			\end{align}
			for all large $h$ where $I_\epsilon$ denotes the $\epsilon$ neighborhood of a set.
		\end{theorem}
		
		\begin{proof}
			It must be that $x_0\in \Sigma(E;U)$ since $x_0\in U\cap\rb E$ would contradict Theorem \ref{convergence of outer unit normals}. We prove \eqref{sing set convergence} by contradiction. Indeed, assume there exist $\epsilon>0$, $H\subset U$ compact, $h(k)\to\infty$ as $k\to\infty$, and $y_k\in\Sigma(E_{h(k)};U)\cap H$ such that $\dist(y_k,\Sigma(E;U)\cap H)\geq \epsilon$. 
			By compactness of $H$ and reducing to a further subsequence, we may assume $y_k\to y_0$ for some $y_0\in H\subset U$. By Proposition \ref{closedness} $(i)$, we have $y_0\in U\cap \p E$.	By the first part of this theorem, we have $y_0\in \Sigma(E;U)$. This implies $y_k\in I_{\epsilon}(\Sigma(E;U)\cap H)$ for large $k$, a contradiction.
		\end{proof}

		\subsection{Existence of blow-up limits}\label{existence of blow-ups section}\
		
		We now prove the existence of blow-up limits along subsequences and their convergence to singular minimizing cones. We say that $F$ is a \textbf{cone with vertex $x_0$} if it is invariant under blow-ups at $x_0$, that is, if
		\begin{align}
			F=F_{x_0,r}=\frac{F-x_0}{r}
		\end{align}
		for all $r>0$. If $F$ is a cone which is a (global) minimizer of $\sF_A$ in $\R^n$ and $\Sigma(F)=\Sigma(F;\R^n)\not=\varnothing$, we say that $F$ is a \textbf{singular minimizing cone} of $\sF_A$. A singular minimizing cone of perimeter we simply refer to as a singular minimizing cone.
		
		Note that if $F$ is a singular minimizing cone, then $0\in \Sigma(F)$, for otherwise the blow-ups $F_{0,r}$ would converge to a half-space, implying that $F=F_{0,r}$ is a half-space, in contradiction with $\Sigma(F)\not=\varnothing$.
		
		\begin{theorem}[Existence of blow-up limits at singular points]\label{existence of blow-ups}
		If $E$ is a $(\kappa,\alpha)$-almost-minimizer of $\sF_A$ in the open set $U$ at scale $r_0$, $x_0\in \Sigma(E;U)$, and $r_h\to 0^+$, then, setting 
		\begin{align}
			E_h=E_{x_0, r_h}=\frac{E-x_0}{r_h},\qquad\text{and}\qquad A_h(x)=A_{x_0,r_h}(y)=A(r_h x+x_0),
		\end{align}
		there exist $h(k)\to\infty$ as $k\to\infty$ and a set of locally finite perimeter $F$ in $\R^n$ such that
		\begin{align}
			E_{h(k)}\loc F, \qquad \mu_{E_{h(k)}}\wkly\mu_{K},\qquad \sF_{A_{h(k)}}(E_{h(k)};\:\cdot\:)\wkly\sF_{A(x_0)}(F;\:\cdot\:) \text{ on bounded subsets of }\R^n
		\end{align}
 		and $F$ is singular minimizing cone of $\sF_{A(x_0)}$ in $\R^n$ with vertex $0$.
	\end{theorem}
			
	\begin{proof}
		The change of variable $T_{x_0}$ allows us to assume without loss of generality that $A(x_0)=I$ since the convergence properties and cones are preserved under this affine transformation. For each $R>0$, $\B_R$ is eventually contained in $U_{x_0,r_h}$ for large $h$. Note that $E_h$ is a $(\kappa r_h^\alpha,\alpha)$-almost-minimizer of $\sF_{A_h}$ in $U_{x_0,r_h}$ at scale $r_0/r_h$ by Proposition \ref{scaling of the anisotropic energy} and $||A_h||_{C^\alpha}\leq r_h^\alpha ||A||_{C^\alpha}$. Hence $||A_h||_{C^\alpha}\leq M_1$ and $||A_h(x)||\leq ||A(r_h x+x_0)||\leq M_2$ for some positive constants $M_1$ and $M_2$. Once $r_h^\alpha\leq 1$, we have $E_h$ is a $(\kappa,\alpha)$-almost-minimizer of $\sF_{A_h}$. Thus we may apply the compactness of Proposition \ref{precompactness} and Proposition \ref{closedness} to obtain $h(k)\to\infty$ as $k\to \infty$, a set of locally finite perimeter $F$, and a uniformly elliptic, H\"older continuous matrix $A_\infty$ such that $\B_R\cap E_{h(k)}\to F$ and $A_{h(k)}\to A_\infty$ uniformly on compact sets and $F$ is a minimizer of $\sF_{A_\infty}$ in $\B_R$ at scale $\liminf_{k\to\infty} r_0/r_{h(k)}=\infty$. Note that $A_\infty(x)=\lim_{k\to\infty} A(r_{h(k)} x+x_0)=A(x_0)=I$. Hence $\sF_{A_\infty}=P$. By Proposition \ref{closedness} and a diagonalization argument, we obtain a subsequence such that up to relabeling
		\begin{align}\label{diagonalization subsequence}
			& E_{h(k)}\loc F,\nonumber\\
			& A_{h(k)}\to I\text{ uniformly on compact sets},\nonumber\\
			& F \text{ is a (global) minimizer of perimeter in } \R^n,\nonumber\\
			& \sF_{A_{h(k)}}(E_{h(k)};\:\cdot\:)\wkly P(F;\:\cdot\:)\text{ on bounded subsets of } \R^n.
		\end{align}
		By Theorem \ref{closure of singularities} we have $0\in\Sigma(F)$ because $0\in\Sigma(E_{h(k)};U)$. All that remains is to show that $F$ is a cone with vertex $0$. Choose one of the a.e. $r>0$ for which  we have  $\Hm(\rb F\cap \p\B_r)=0$. By \eqref{diagonalization subsequence}, Proposition \ref{scaling of the anisotropic energy}, and Corollary \ref{existence of densities}, it follows that
		\begin{align}
			P(F; \B_r) &= \lim_{k\to \infty} \sF_{A_{h(k)}}(E_{h(k)}; \B_r)\nonumber\\
				&=\lim_{k\to\infty} \frac{ \sF_A(E; \B(x_0, r r_{h(k)}))}{r_{h(k)}^{n-1}}=r^{n-1}\theta_A(E, x_0)
		\end{align}
		since $\W_{x_0}(x_0,r r_{h(k)})=\B(x_0,r r_{h(k)})$ as $A(x_0)=I$. Hence
		\begin{align}\label{constant density}
				\frac{P(F; \B_r)}{r^{n-1}}=\theta_A(E,x_0)\qquad \text{for a.e. } r>0. 
		\end{align}
		The monotonicity formula for perimeter minimizers \cite[Theorem 28.9]{maggi} gives
		\begin{align}
			\frac{d}{dr}\frac{P(F; \B_r)}{r^{n-1}}=\frac{d}{dr}\int_{\B_r\cap\rb F} \frac{\big(\nu_{F}(x)\cdot x\big)^2}{|x|^{n+1}}\dH(x)\qquad \text{for a.e. } r>0. 
		\end{align}
	    So \eqref{constant density} implies
		\begin{align}
			\frac{d}{dr}\int_{\B_r\cap\rb F} \frac{\big(\nu_{F}(x)\cdot x\big)^2}{|x|^{n+1}}\dH(x)=0\qquad \text{for a.e. } r>0. 
		\end{align}
		Hence
		\begin{align}
		\int_{\B_s\cap\rb F} \frac{\big(\nu_{F}(x)\cdot x\big)^2}{|x|^{n+1}}\dH(x)=\int_{\B_r\cap\rb F} \frac{\big(\nu_{F}(x)\cdot x\big)^2}{|x|^{n+1}}\dH(x)
		\end{align}
		for a.e. $0<s<r$, and thus
		\begin{align}
			\int_{(\B_r\setminus\bar\B_s)\cap\rb F} \frac{\big(\nu_{F}(x)\cdot x\big)^2}{|x|^{n+1}}\dH(x)= 0.
		\end{align}
		This implies $\nu_{F}(x)\cdot x=0$ for $\Hm$-a.e. $x\in \rb F$. Thus $F^{(1)}$ is a cone with vertex $0$ by \cite[Proposition 28.8]{maggi}. 
	\end{proof}
	
	\subsection{Dimension reduction argument}\
	
	The next few results we recall from the standard argument for the characterization of the singular set for perimeter minimizers which we use for our adapted proof. We refer readers to \cite[Chapter 28]{maggi} for proofs.
	
	\begin{theorem} If $n\geq 2$ and there exists a singular minimizing cone $F\subset\R^n$ with $\Sigma(F)=\{0\}$, then $n\geq 8$.
	\end{theorem}
	
	\begin{theorem}[Dimension reduction theorem]\label{dimension reduction} If $F$ is a singular minimizing cone in $\R^n$, $x_0\in\Sigma(F)$, $x_0\not=0$, and if $r_h\to 0^+$, then, up to extracting a subsequence and up to rotation, the blow-ups $F_{x_0,r_h}$ locally converge to a cylinder $G\times \R$, where $G$ is a singular minimizing cone in $\R^{n-1}$.
	\end{theorem}
		
	\begin{lemma}[Half-lines of singular points]\label{half-lines of singular points} If $F$ is a singular minimizing cone in $\R^n$, $x_0\in\Sigma(F)$, and $x_0\not=0$, then $\{t\: x_0: t>0\}\subset\Sigma(F)$ and $n\geq 3$.
	\end{lemma}
	
	\begin{lemma}[Cylinders of locally finite perimeter]\
	
	\begin{enumerate}
		\item[(i)] If $F$ is a set of locally finite perimeter in $\R^{n-1}$, then $F\times\R$ is of locally finite perimeter in $\R^n$, with 
		\begin{align}
			\mu_{F\times\R}=(\nu_F(\pp x),0)\Hm\rstr \big((\rb F)\times\R\big).
		\end{align}
		Moreover, if $F$ is a perimeter minimizing in $\R^{n-1}$, then $F\times\R$ is a perimeter minimizer in $\R^n$.

		\item[(ii)] If $E$ is a set of locally finite perimeter in $\R^n$ such that
		\begin{align}
			\nu_E(x)\cdot e_n=0\qquad\text{for }\Hm-a.e.\ x\in\rb E,
		\end{align}
		then there exists a set of locally finite perimeter $F$ in $\R^{n-1}$ such that $E$ is equivalent to $F\times\R$. If, moreover, $E$ is a perimeter minimizer in $\R^n$, then $F$ is a perimeter minimizer in $\R^{n-1}$.
	\end{enumerate}	
	\end{lemma}
	
	\begin{lemma}\label{Hausdorff content lemma}\
		\begin{enumerate}
			\item[(i)] If $E$ is a Borel set such that $\sH^s(E)<\infty$, $s>0$, then
			\begin{align}
				\limsup_{r\to 0^+}\frac{\sH_\infty^s(E\cap \B(x,r))}{\omega_s r^s}\geq \frac{1}{2^s},\qquad\text{for }\sH^s-a.e.\ x\in E
			\end{align}
			\item[(ii)] If $E$ is an $(\kappa,\alpha)$-almost-minimizer of $\sF_A$ in the open set $U\subset\R^n$ at scale $r_0$, and $r_h\to 0^+$, then, setting 
		\begin{align}
			E_h=E_{x_0, r_h}=\frac{E-x_0}{r_h},\qquad\text{and}\qquad A_h(x)=A_{x_0,r_h}(x)=A(r_h x+x_0),
		\end{align}
		we have
			\begin{align}
				\sH_\infty^s(\Sigma(E;U)\cap H)\geq \limsup_{h\to\infty}\sH_\infty^s(\Sigma(E_h;U)\cap H)
			\end{align}
			 for every compact set $H\subset U$.
			\item[(iii)] If $s\geq 0$, $F\subset\R^{n-1}$, and $\sH_\infty^s(F)=0$, then $\sH_\infty^{s+1}(F\times\R)=0$.
		\end{enumerate}
	\end{lemma}

	\begin{proof}
		$(i)$ and $(iii)$ are proved in \cite[Lemma 28.14]{maggi} and we now adapt the proof of his version of $(ii)$ to the case of $(\kappa,\alpha)$-almost-minimizers. 
		
		Let $\cF$ be a finite covering by open sets of the compact set $\Sigma(E;U)\cap H$. Then there exists $\epsilon>0$ such that $I_\epsilon(\Sigma(E;U)\cap H)\subset\bigcup_{F\in\cF} F$. Eventually $U\subset U_{x_0,r_h}$ and $r_h^\alpha\leq 1$, and so by Proposition \ref{scaling of the anisotropic energy} $E_h$ is a $(\kappa,\alpha)$-almost-minimizer of $\sF_{A_h}$ in $U$. Then by Theorem \ref{closure of singularities},
		\begin{align}
		    \Sigma(E_h;U)\cap H\subset I_\epsilon(\Sigma(E;U)\cap H)\subset\bigcup_{F\in\cF} F
		\end{align}
		for large $h$. Hence by definition of $\sH_\infty^s$ it follows that
		\begin{align}
		    \limsup_{h\to\infty}\sH_\infty^s(\Sigma(E;U)\cap H)\leq \sH_\infty^s(\bigcup_{F\in\cF} F)\leq\omega_s\sum_{F\in\cF}\Big(\frac{\diam(F)}{2}\Big)^s
		\end{align}
		Taking the infimum over all such coverings $\cF$ proves the result.
	\end{proof}
	
	Lastly we recall a technical result \cite[Theorem 8.16]{mattila} before we jump into the proof of Theorem \ref{dimension estimates}.

	\begin{lemma}\label{mattila theorem} If $E\subset \R^n$ and $\sH^s(E)>0$, then there exists $F\subset E$ with $0<\sH^s(F)<\infty$.
	\end{lemma}
	
	\begin{proof}[Proof of Theorem \ref{dimension estimates}]\
	$(i)$ Let $E$ be a $(\kappa,\alpha)$-almost-minimizer of $\sF_A$ in $U$ with $2\leq n\leq 7$. By way of contradiction, suppose there exists $x_0\in\Sigma(E;U)$. As usual we may assume without loss of generality that $A(x_0)=I$ by the change of variable $T_{x_0}$. By Theorem \ref{existence of blow-ups} there exists a singular minimizing cone $F$ in $\R^n$, but this contradicts Simons' theorem on the nonexistence of singular minimizing cones in dimensions $2\leq n\leq 7$, see \cite[Theorem 28.1 (i)]{maggi}.
	
	$(ii)$ Let $E\subset\R^8$ be a $(\kappa,\alpha)$-almost-minimizer of $\sF_A$ in $U$. By way of contradiction, suppose $x_0\in U\cap\p E$ is an accumulation point of $\Sigma(E;U)$. Then there is a sequence $x_h\in\Sigma(E;U)$ such that $x_h\to x_0$. Again we may assume without loss of generality that $A(x_0)=I$. Set $r_h=|x_h-x_0|$ and consider the blow-ups $E_h=E_{x_0,r_h}$. By Theorem \ref{existence of blow-ups} there is a subsequence, which upon relabeling as $E_h$, converges locally to a singular minimizing cone $F$ in $\R^n$. Let $y_h=(x_h-x_0)/r_h$. Then $y_h\in\SS^{n-1}$ and so by compactness there exists $y_0\in \SS^{n-1}$ and a further subsequence so that, up to relabeling, we have $y_h\to y_0$. Note that $y_h\in \Sigma(E_h; U_{x_0,r_h})$ as $x_h\in\Sigma(E;U)$. So by Theorem \ref{closure of singularities} we have $y_0\in\Sigma(F)$. Since $y_0\not=0$, we have $\sH^8(F)>1$ and so by Theorem \ref{dimension reduction} there exists a singular minimizing cone $G$ in $\R^7$, contradicting $(i)$.
	
	$(iii)$ Let $E\subset\R^n$ be a $(\kappa,\alpha)$-almost-minimizer of $\sF_A$ in $U$ and suppose $\sH^s(\Sigma(E;U))>0$ with $s>0$. Then there exists $x_0\in\Sigma(E;U)$. By the change of variable $T_{x_0}$, we may assume without loss of generality that $A(x_0)=I$. By \cite[Theorem 8.16]{mattila} and Lemma \ref{Hausdorff content lemma} $(i)$, there exists $r_h\to 0^+$ such that
	\begin{align}
		\sH_\infty^s(\Sigma(E;U)\cap\bar\B(x_0,r_h))\geq \frac{\omega_s r_h^s}{2^{s+1}}
	\end{align}
	for all $h\in\NN$. This is equivalently rewritten in terms of the blow-ups $E_h=E_{x_0,r_h}$ as
	\begin{align}
		\sH_\infty^s(\Sigma(E_h;U_{x_0,r_h})\cap\bar\B_1)\geq \frac{\omega_s}{2^{s+1}}
	\end{align}
	for all $h\in\NN$. Eventually $\B_2\subset U_{x_0,r_h}$ when $h$ is sufficiently and thus
	\begin{align}
		\sH_\infty^s(\Sigma(E_h;\B_2)\cap\bar\B_1)\geq \frac{\omega_s}{2^{s+1}}.
	\end{align}
	By Theorem \ref{existence of blow-ups} there exists a subsequence, which up relabeling as $E_h$, converges locally to a singular minimizing cone $F$ in $\R^n$. By Lemma \ref{Hausdorff content lemma} $(ii)$ we have
	\begin{align}
		\sH_\infty^s(\Sigma(F)\cap\bar\B_1)\geq \limsup_{h\to\infty}	\sH_\infty^s(\Sigma(E_h;\B_2)\cap\bar\B_1)\geq \frac{\omega_s}{2^{s+1}}.
	\end{align}
	We may apply the above argument with $F$ and $\R^n$ in place of $E$ and $U$. By Theorem \ref{dimension reduction} there exists a singular minimizing cone $G\times \R$
	\begin{align}
		\sH_\infty^s(\Sigma(G\times\R)\cap\bar\B_1)\geq \frac{\omega_s}{2^{s+1}}.
	\end{align}
	By Lemma \ref{Hausdorff content lemma} $(iii)$ we must have $\sH^{s-1}(\Sigma(G))>0$. If we now assume that $n\geq 9$ and $s>n-8$, repeating this argument $n-8$ times gives the existence of a singular minimizing cone $G$ in $\R^8$ with $\sH^{s-(n-8)}(\Sigma(G))>0$, in contradiction to $(ii)$. Thus we conclude that
	$s\leq n-8$.
	\end{proof}
	
	\appendix
	
	\setcounter{section}{0}
    \renewcommand{\thesection}{\Alph{section}}
    
	\section{Appendix A}\label{appendix A}

	In this appendix we provide a proof for the change of variable formula for sets of locally finite perimeter with bounded, continuous integrands depending on both $x$ and $\nu_E$. This is a generalization of \cite[Proposition 17.1]{maggi} and \cite[Theorem A.1]{KinMagStu19}.
		
	\begin{proposition}[Change of variable for sets of locally finite perimeter]\label{change of variable}

	Suppose $f$ is a diffeomorphism of $\R^n$ and denote $g=f^{-1}$. If $E$ is a set of locally finite perimeter in $\R^n$, then $f(E)$ is a set of locally finite perimeter in $\R^n$ such that 
	\begin{align}\label{diffeo image of perimeter}
		\rb f(E)= f(\rb E)\qquad\text{and}\qquad \nu_{f(E)}(f(x))=\frac{(\nabla g\circ f)^t\nu_E(x)}{|(\nabla g\circ f)^t\nu_E(x)|} \text{ for all } x\in\rb E.
	\end{align} 
	If $\Phi\colon\R^{n}\times \SS^{n-1}\to [0,\infty)$ is a bounded and continuous function and $U$ is an open, bounded set satisfying $\Hm(\p U\cap\rb f(E))=0$, and $\nabla f$ and $\nabla g$ are bounded, then the change of variable $y=f(x)$ gives 
	\begin{align}\label{change of variable formula}
		\int_{U\cap\rb f(E)} \Phi(y,\nu_{f(E)})\dH(y)=\int_{g(U)\cap \rb E}\Phi\Big(f(x), \frac{(\nabla g\circ f)^t\nu_E}{|(\nabla g\circ f)^t\nu_E|}\Big) Jf\: |(\nabla g\circ f)^t\nu_{E}|\dH(x).
	\end{align}
	Note that $Jf\:|(\nabla g\circ f)^t\nu_{E}|$ is the tangential Jacobian of $f$ with respect to $\rb E$.
	\end{proposition}
	
	\begin{proof} The fact that $f(E)$ is a set of locally finite perimeter is shown in \cite[Proposition 17.1]{maggi} and \eqref{diffeo image of perimeter}  is proved in \cite[Theorem A.1]{KinMagStu19}. Hence we only need to show \eqref{change of variable formula}.
	
		Let $u_\epsilon=1_E*\rho_\epsilon$ where $\rho_\epsilon$ denotes the standard mollifier and let $v_\epsilon=u_\epsilon\circ g$.  Then $u_\epsilon\to 1_E$ in $L^1_{loc}(\R^n)$ and $v_\epsilon\to 1_E\circ g=1_{f(E)}$ in $L^1_{loc}(\R^n)$ as shown in the proof of \cite[Proposition 17.1]{maggi}. Note that $\nabla v_\epsilon=(\nabla g)^t(\nabla u_\epsilon\circ g)$ and so $\nabla v_\epsilon\circ f=(\nabla g\circ f)^t\nabla u_\epsilon$. By \cite[Remark 8.3]{maggi}, the change of variable $y=f(x)$ gives
		\begin{align}\label{approx change of variable}
		\int_{U} \Phi\Big(y,-\frac{\nabla v_\epsilon}{|\nabla v_\epsilon|}\Big) |\nabla v_\epsilon|dy&=\int_{g(U)} \Phi\Big(f(x),-\frac{\nabla v_\epsilon\circ f}{|\nabla v_\epsilon\circ f|}\Big)Jf|\nabla v_\epsilon\circ f| dx\nonumber\\
		&=\int_{g(U)} \Phi\Big(f(x),-\frac{(\nabla g\circ f)^t\nabla u_\epsilon}{|(\nabla g\circ f)^t\nabla u_\epsilon|}\Big) Jf|(\nabla g\circ f)^t\nabla u_\epsilon| dx.
		\end{align}
		We shall show that this equation converges to \eqref{change of variable formula} as $\epsilon\to 0^+$. To do this, we shall apply the version of Reshetynak's continuity theorem provided in \cite[Theorem 1.3]{spector}. Under the hypotheses that $\Phi$ is bounded and continuous and $U$ is open, this states that
		\begin{align}\label{Reshetnyak limit}
			\lim_{h\to\infty}\int_U \Phi(x, D_{|\mu_h|}\mu_h)d|\mu_h|=\int_U \Phi(x,D_{|\mu|}\mu)d|\mu|
		\end{align}
		whenever $\mu_h$, $\mu$ are finite $\R^n$-valued measures satisfying 
		\begin{align}\label{Reshetnyak measure convergence}
			\lim_{h\to\infty}\int T\cdot d\mu_h=\int T\cdot d\mu,\ \forall\: T\in C_0(U;\R^n), \qquad\text{and} \qquad |\mu_h|(U)\to |\mu|(U),
		\end{align}
		where $C_0(U;\R^n)$ denotes the completion, with respect to the sup norm, of the compactly supported continuous functions from $U$ to $\R^n$. 
		
		Starting with the right-hand side of \eqref{approx change of variable}, first note that $-(\nabla g\circ f)^t\nabla u_\epsilon\sL^n\wkly (\nabla g\circ f)^t\mu_E$ since $(\nabla g\circ f)^t$ is continuous. By \cite[Theorem 12.20]{maggi}, $-\nabla u_\epsilon \sL^n\wkly \mu_E$ and $|\nabla u_\epsilon| \sL^n\wkly |\mu_E|$ where $\sL^n$ denotes Lebesgue measure. Since $U$ is bounded and $g$ is continuous, $g(U)$ is bounded. Since $\nabla g$ is bounded, $\Hm(\p (g(U))\cap \rb E)=\Hm(g(\p U\cap \rb f(E)))\leq \Lip(g)^{n-1}\Hm(\p U\cap \rb f(E))=0$. Hence $(|\nabla u_\epsilon|\sL^n)(\p (g(U))\to |\mu_E|(\p (g(U)))$ and
		\begin{align}
			\lim_{\epsilon\to 0^+}\int T\cdot (-\nabla u_\epsilon)dx=\int T\cdot d\mu_E
		\end{align}
		for all $T\in C_0(g(U);\R^n)$, since $T\in C_c(\R^n;\R^n)$ as $\bar {g(U)}$ is compact. Thus $-\nabla u_\epsilon \sL^n\rstr g(U)$, $\mu_E\rstr g(U)$ are finite $\R^n$-valued measures which satisfy \eqref{Reshetnyak measure convergence} (where we take discrete sequences of $\epsilon_h\to 0^+$). Hence for each $\phi\in\C_c(\R^n)$, applying \cite[Theorem 1.3]{spector} to the bounded, continuous function $(x,\xi)\mapsto \phi(x)|(\nabla g\circ f(x))^t\xi|$ gives
		\begin{align}
			\lim_{\epsilon\to 0^+} \int_{g(U)}	\phi\Big|\Big(\nabla g\circ f\Big)^t\frac{(-\nabla u_\epsilon)}{|\nabla u_\epsilon|}\Big| |\nabla u_\epsilon|dx=\int_{g(U)\cap\rb E}\phi|(\nabla g\circ f)^t\nu_E| \dH,		
		\end{align}
		that is, $|(\nabla g\circ f)^t\nabla u_\epsilon|\sL^n\rstr g(U)\wkly |(\nabla g\circ f)^t\nu_E|\Hm\rstr\:( g(U)\cap \rb E)$. In particular, it follows from the fact $\Hm(\p (g(U))\cap\rb E)=0$ that  $(|(\nabla g\circ f)^t\nabla u_\epsilon|\sL^n) (g(U))\to (|(\nabla g\circ f)^t\nu_E|\Hm\rstr\:\rb E)(g(U))$. Hence \eqref{Reshetnyak measure convergence} holds for $(\nabla g\circ f)^t\nabla u_\epsilon\sL^n\rstr g(U)$, $(\nabla g\circ f)^t\nu_E\Hm\rstr\:( g(U)\cap \rb E)$ and so we can apply \cite[Theorem 1.3]{spector} to $(x,\xi)\mapsto\Phi(f(x),\xi)Jf(x)$ which is bounded continuous since $\Phi$ and $\nabla f$ are. We obtain
		\begin{align}
			\lim_{\epsilon\to 0^+}&\int_{g(U)} \Phi\Big(f(x),-\frac{(\nabla g\circ f)^t\nabla u_\epsilon}{|(\nabla g\circ f)^t\nabla u_\epsilon|}\Big) Jf|(\nabla g\circ f)^t\nabla u_\epsilon| dx\nonumber\\
			&=\int_{g(U)\cap \rb E}\Phi\Big(f(x), \frac{(\nabla g\circ f)^t\nu_E}{|(\nabla g\circ f)^t\nu_E|}\Big) Jf\: |(\nabla g\circ f)^t\nu_{E}|\dH(x).
		\end{align}
		This is the convergence of the right-hand side of \eqref{approx change of variable} to the right-hand side of \eqref{change of variable formula}.
	
		Now moving on to the left-hand side of \eqref{approx change of variable}, note that for all $\phi\in C_c^1(\R^n)$,
		\begin{align}
			\int \phi(-\nabla v_\epsilon)dy=\int (\nabla\phi) v_\epsilon	dy\to \int_{f(E)} \nabla\phi\: dy=\int \phi\: d\mu_{f(E)}
		\end{align}
		since $v_\epsilon\to 1_{f(E)}$ in $L^1_{loc}(\R^n)$. So by density of $C_c^1(\R^n)$ in $C_c(\R^n)$, we have $-\nabla v_\epsilon\sL^n\wkly \mu_{f(E)}$. The change of variable $y=f(x)$ (\cite[Remark 8.3]{maggi}) gives
		\begin{align}
			\int \phi|\nabla v_\epsilon|\:dy&=\int (\phi\circ f)|\nabla v_\epsilon\circ f|Jfdx=\int (\phi\circ f)|(\nabla g\circ f)^t\nabla u_\epsilon|Jf dx\nonumber\\
			&\to \int (\phi\circ f)|(\nabla g\circ f)^t\nu_E|Jf dx=\int\phi\: d|\mu_{f(E)}|
		\end{align}
		where the last equality is by \cite[Proposition 17.1]{maggi}. Hence $|\nabla v_\epsilon|\sL^n\wkly |\mu_{f(E)}|$. Thus $(|\nabla v_\epsilon|\sL^n)(U)\to |\mu_{f(E)}|(U)$ since $\Hm(\p U\cap\rb f(E))=0$ and
		\begin{align}
			\lim_{\epsilon\to 0^+}\int T\cdot (-\nabla v_\epsilon)dy=\int T\cdot d\mu_{f(E)}
		\end{align}
		for all $T\in C_0(U;\R^n)$, since $T\in C_c(\R^n;\R^n)$ as $\bar U$ is compact. Hence $-\nabla v_\epsilon\sL^n\rstr U, \mu_{f(E)}\rstr U$  satisfies \eqref{Reshetnyak measure convergence} so by \cite[Theorem 1.3]{spector}, we have
		\begin{align}
			\lim_{\epsilon\to 0^+}\int_{U} \Phi\Big(y,-\frac{\nabla v_\epsilon}{|\nabla v_\epsilon|}\Big) |\nabla v_\epsilon|dy=	\int_{U\cap\rb f(E)} \Phi(y,\nu_{f(E)})\dH(y).
		\end{align}
		Hence left hand side of \eqref{approx change of variable} converges to the left hand side of \eqref{change of variable formula} and we are done.
		\end{proof}
	
\bibliographystyle{amsalpha}
\bibliography{Regularity_of_Almost_Minimizers_of_Holder_Coefficient_Surface_Energies}

\end{document}